\documentclass[leqno,11pt]{amsart}
\usepackage{amsmath}
\usepackage{amssymb}
\usepackage{amsthm}

\usepackage{mathrsfs}
\usepackage{dsfont}
\usepackage{txfonts}
\usepackage{mathabx}

\usepackage{graphicx}
\usepackage[all]{xy}
\usepackage{pgf,tikz}
\usepackage{subfigure}
\usepackage{tikz-3dplot}
\usetikzlibrary{patterns,arrows}
\tikzset{
    partial ellipse/.style args={#1:#2:#3}{
        insert path={+ (#1:#3) arc (#1:#2:#3)}
    }
}

\usepackage{hyperref}
\usepackage{color}
\usepackage[shortlabels]{enumitem}

\newtheorem{theorem}{Theorem}[section]
\newtheorem{lemma}[theorem]{Lemma}
\newtheorem{corollary}[theorem]{Corollary}

\newtheorem{proposition}[theorem]{Proposition}

\newtheorem{claim}[theorem]{Claim}

\theoremstyle{definition}
\newtheorem{definition}[theorem]{Definition}

\theoremstyle{remark}

\newcommand{\cyl}{\mathrm{cyl}}
\newcommand{\tip}{\mathrm{tip}}

\newcommand{\bw}{\mathbf{w}}
\newcommand{\bg}{\mathbf{g}}
\newcommand{\bW}{\mathbf{W}}
\newcommand{\bG}{\mathbf{G}}
\newcommand{\bbf}{\mathbf{f}}

\newcommand{\bF}{\mathbf{F}}

\newcommand{\cC}{\mathcal{C}}
\newcommand{\cD}{\mathcal{D}}

\newcommand{\cT}{\mathcal{T}}

\newcommand{\cN}{\mathcal{N}}
\newcommand{\fH}{\mathfrak{H}}
\newcommand{\fp}{\mathfrak{p}}

\newcommand{\eps}{\varepsilon}

\title[{The linearized translator equation and applications}]{The Linearized translator equation and applications}

\author{Kyeongsu Choi, Robert Haslhofer, Or Hershkovits}

\begin{document}

\begin{abstract}
In this paper, we consider the linearized translator equation $L_\phi u=f$, around entire convex translators $M=\textrm{graph}(\phi)\subset\mathbb{R}^4$, i.e. in the first dimension where the Bernstein property fails. Here, $L_\phi u=\mathrm{div} (a_\phi  D u)+ b_\phi\cdot Du$ is a mean curvature type elliptic operator, whose coefficients degenerate as the slope tends to infinity. We derive two fundamental barrier estimates, specifically an upper-lower estimate and an inner-outer estimate, which allow to propagate $L^\infty$-control between different regions. Packaging these and further estimates together we then develop a Fredholm theory for $L_\phi$ between carefully designed weighted function spaces. Combined with Lyapunov-Schmidt reduction we infer that the space $\mathcal{S}$ of noncollapsed translators in $\mathbb{R}^4$ is a finite dimensional analytic variety and that the tip-curvature map $\kappa:\mathcal{S}\to\mathbb{R}$ is analytic. Together with the main result from  \cite{CHH_translators} this allows us to complete the classification of noncollapsed translators in $\mathbb{R}^4$.
In particular, we conclude that the one-parameter family of translators constructed by Hoffman-Ilmanen-Martin-White is uniquely determined by the smallest principal curvature at the tip.
\end{abstract}

\maketitle

\tableofcontents

\section{Introduction}

In this paper, we are concerned with the graphical translator equation
\begin{equation}\label{trans_eq_intro}
\mathrm{div}\left(\frac{D \phi}{\sqrt{1+|D \phi|^2}}\right)-\frac{1}{\sqrt{1+|D \phi|^2}}=0,
\end{equation}
and in particular the corresponding (inhomogeneous) linearized graphical translator equation
\begin{equation}\label{lin_trans_eq_intro}
\mathrm{div}(a_\phi  Du)+ b_\phi\cdot Du=f,
\end{equation}
where
\begin{equation}
a_\phi=\frac{\delta}{\sqrt{1+|D \phi|^2}}  -\frac{D\phi\otimes D\phi}{{(1+|D \phi|^2)^{3/2}}} , \qquad b_\phi=\frac{D \phi}{(1+|D \phi|^2)^{3/2}}.
\end{equation}
We recall that translators model slowly forming singularities under mean curvature flow, see e.g. \cite{Hamilton_Harnack,HuiskenSinestrari_convexity,White_nature,HaslhoferKleiner_meanconvex}. In particular, it is known that all translators that arise as blowup limits of mean-convex mean curvature flow are given by convex entire graphical solutions of \eqref{trans_eq_intro}. More generally, as has been proved recently in $\mathbb{R}^3$ in \cite{CHH,CCS,BK_mult1}, this is expected for all blowup limits near generic singularities. We also recall from \cite{CHH_translators,BN_noncollapsing,BLL} that the analytic property of being a convex entire graph is equivalent to the geometric property that $M=\mathrm{graph}(\phi)\subset \mathbb{R}^N$ is noncollapsed, i.e. that every $p\in M$ admits interior and exterior tangent balls of radius at least $\alpha/H(p)$ for some constant $\alpha>0$.\\

The study of elliptic PDEs of mean curvature type, such as \eqref{trans_eq_intro} and \eqref{lin_trans_eq_intro}, of course has a long history dating back to the work from more than 100 years ago by Bernstein \cite{Bernstein1915}. In particular, thanks to the solution of the Bernstein problem by Simons \cite{Sim68} and Bombieri-DeGiorgi-Giusti \cite{BGG} it is know that the graphical minimal surface equation
\begin{equation}\label{theta_def_intro}
\mathrm{div}\left(\frac{D \phi}{\sqrt{1+|D \phi|^2}}\right)=0
\end{equation}
admits nontrivial (i.e. nonplanar) entire solutions $\mathrm{graph}(\phi)\subset\mathbb{R}^N$ if and only if $N\geq 9$. In stark contrast, for equation \eqref{trans_eq_intro} the Bernstein property already fails in $\mathbb{R}^4$. More precisely, in pioneering work \cite{Wang_convex}, Wang on the one hand proved that every noncollapsed translator in $\mathbb{R}^3$ is the unique (up to scaling and rigid motion) rotationally symmetric bowl from \cite{AltschulerWu}, and on the other hand for every $N\geq 4$    constructed nontrivial examples of noncollapsed translators in $\mathbb{R}^N$, i.e. examples that are neither rotationally symmetric nor split off a line. Later, a simpler proof of Wang's uniqueness result has been given by the second author in \cite{Haslhofer_bowl} and a more detailed construction of Wang's examples has been given by Hoffman-Ilmanen-Martin-White \cite{HIMW}. In particular, in $\mathbb{R}^4$ one obtains a one-parameter family of examples $\{M_{\kappa}\}_{\kappa\in (0,1/3)}$, called the oval-bowls, which are illustrated in Figure \ref{figure_oval_bowls}.\\

\begin{figure}
\scalebox{0.5}{
\begin{tikzpicture}[x=1cm,y=1cm] \clip(-4,12) rectangle (4,-1);
\draw [samples=100,rotate around={0:(0,0)},xshift=0cm,yshift=0cm,domain=-8:8)] plot (\x,{(\x)^2});
\draw [dashed] (0,2) [partial ellipse=0:180:2.75*0.5cm and 0.9*0.5cm];
\draw (0,2) [partial ellipse=180:360:2.75*0.5cm and 0.9*0.5cm];
\draw [dashed] (0,4.1) [partial ellipse=0:180:4*0.5cm and 0.95*0.5cm];
\draw (0,4.1) [partial ellipse=180:360:4*0.5cm and 0.95*0.5cm];
\draw [dashed] (0,6.4) [partial ellipse=0:180:5.05*0.5cm and 1*0.5cm];
\draw (0,6.4) [partial ellipse=180:360:5.05*0.5cm and 1*0.5cm];
\draw [dashed] (0,9.2) [partial ellipse=0:180:6.05*0.5cm and 1.05*0.5cm];
\draw (0,9.2) [partial ellipse=180:360:6.05*0.5cm and 1.05*0.5cm];
\draw [color=blue] (0,4.1) [partial ellipse=90:270:0.3*0.5cm and 0.95*0.5cm];
\draw [color=blue,dashed] (0,4.1) [partial ellipse=-90:90:0.3*0.5cm and 0.95*0.5cm];
\draw [color=blue] (-2*0.5,4.1) [partial ellipse=90:270:0.2*0.5cm and 0.77*0.5cm];
\draw [color=blue,dashed] (-2*0.5,4.1) [partial ellipse=-90:90:0.2*0.5cm and 0.77*0.5cm];
\draw [color=blue] (2*0.5,4.1) [partial ellipse=90:270:0.2*0.5cm and 0.77*0.5cm];
\draw [color=blue,dashed] (2*0.5,4.1) [partial ellipse=-90:90:0.2*0.5cm and 0.77*0.5cm];
\draw [color=blue] (0,6.4) [partial ellipse=90:270:0.35*0.5cm and 1*0.5cm];
\draw [color=blue,dashed] (0,6.4) [partial ellipse=-90:90:0.35*0.5cm and 1*0.5cm];
\draw [color=blue] (-2.7*0.5,6.4) [partial ellipse=90:270:0.25*0.5cm and 0.88*0.5cm];
\draw [color=blue,dashed] (-2.7*0.5,6.4) [partial ellipse=-90:90:0.25*0.5cm and 0.88*0.5cm];
\draw [color=blue] (2.7*0.5,6.4) [partial ellipse=90:270:0.25*0.5cm and 0.88*0.5cm];
\draw [color=blue,dashed] (2.7*0.5,6.4) [partial ellipse=-90:90:0.25*0.5cm and 0.88*0.5cm];
\draw [color=blue] (-1*0.5,2) [partial ellipse=90:270:0.28*0.5cm and 0.85*0.5cm];
\draw [color=blue,dashed] (-1*0.5,2) [partial ellipse=-90:90:0.28*0.5cm and 0.85*0.5cm];
\draw [color=blue] (1*0.5,2) [partial ellipse=90:270:0.28*0.5cm and 0.85*0.5cm];
\draw [color=blue,dashed] (1*0.5,2) [partial ellipse=-90:90:0.28*0.5cm and 0.85*0.5cm];
\draw [color=blue] (-1.2*0.5,9.2) [partial ellipse=90:270:0.28*0.5cm and 1*0.5cm];
\draw [color=blue,dashed] (-1.2*0.5,9.2) [partial ellipse=-90:90:0.28*0.5cm and 1*0.5cm];
\draw [color=blue] (1.2*0.5,9.2) [partial ellipse=90:270:0.28*0.5cm and 1*0.5cm];
\draw [color=blue,dashed] (1.2*0.5,9.2) [partial ellipse=-90:90:0.28*0.5cm and 1*0.5cm];
\draw [color=blue] (-3.3*0.5,9.2) [partial ellipse=90:270:0.25*0.5cm and 0.9*0.5cm];
\draw [color=blue,dashed] (-3.3*0.5,9.2) [partial ellipse=-90:90:0.25*0.5cm and 0.9*0.5cm];
\draw [color=blue] (3.3*0.5,9.2) [partial ellipse=90:270:0.25*0.5cm and 0.9*0.5cm];
\draw [color=blue,dashed] (3.3*0.5,9.2) [partial ellipse=-90:90:0.25*0.5cm and 0.9*0.5cm];
\end{tikzpicture}}
\caption{The oval-bowls $\{M_{\kappa}\}_{\kappa\in(0,1/3)}$ are noncollapsed translators in $\mathbb{R}^4$, whose level sets look like 2d-ovals in $\mathbb{R}^3$. The principal curvatures at the tip are $(\kappa,\tfrac{1-\kappa}{2},\tfrac{1-\kappa}{2})$.}\label{figure_oval_bowls}
\end{figure}
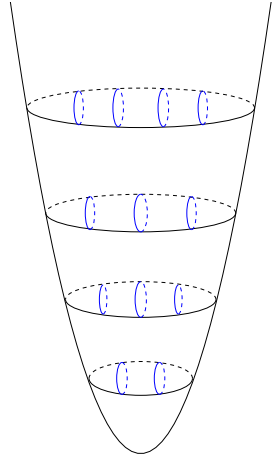

In this paper, we complete the classification of noncollapsed translators in $\mathbb{R}^4$, i.e. in the first dimension where the Bernstein property fails. A large portion of this classification has already been carried out in our prior paper \cite{CHH_translators}. However, in our prior paper a fundamental part, namely analyticity, has only been announced, but not proved. The purpose of the present paper is to establish analyticity.
We denote by $\mathcal{S}$ the space of all nontrivial suitably normalized noncollapsed translators in $\mathbb{R}^4$, namely the space of all strictly convex but not $\mathrm{SO}_3$-symmetric solutions $\phi:\mathbb{R}^3\to\mathbb{R}$ of the graphical translator equation \eqref{trans_eq_intro}, which are normalized such that $\phi(0)=0$ and $D\phi(0) = 0$ and such that their $\mathrm{SO}_2$-symmetry, c.f. \cite{Zhu}, is in the $x_2x_3$-variables. Equipping $\mathcal{S}$ with the smooth topology we consider the tip curvature map
\begin{equation}
\kappa:\mathcal{S}\to \mathbb{R},\quad \phi\mapsto \tfrac{1}{2} (\partial^2_{x_1}\phi) (0).
\end{equation}
The main result of the present paper is:

\begin{theorem}[analyticity]\label{theorem_analyticity_intro}
The space $\mathcal{S}$ is a finite-dimensional analytic variety over which the tip curvature map $\kappa:\mathcal{S}\to\mathbb{R}$ is analytic.
\end{theorem}

Together with the main result from our prior paper \cite{CHH_translators} we thus obtain:

\begin{corollary}[classification]\label{thm_main}
Every noncollapsed translator in $\mathbb{R}^4$, up to rigid motion and scaling, is
\begin{itemize}
\item either  $\mathbb{R}\times \mathrm{Bowl}_2$,
\item or the 3d round bowl $\mathrm{Bowl}_3$,
\item or belongs to the one-parameter family of oval-bowls $\{M_{\kappa}\}_{\kappa\in (0,1/3)}$.
\end{itemize}
In particular, the oval-bowls are uniquely parametrized by the smallest principal curvature at the tip.
\end{corollary}

As a consequence of uniqueness we also obtain:

\begin{corollary}[continuity]
The oval-bowls $\{M_{\kappa}\}_{\kappa\in (0,1/3)}$ depend continuously on $\kappa$.
\end{corollary}

To prove Theorem \ref{theorem_analyticity_intro} (analyticity) we develop a Fredholm theory for the linearized translator equation \eqref{lin_trans_eq_intro}. Since the ellipticity degenerates as the slope $|D\phi|$ tends to infinity, this requires sharp estimates in carefully designed weighted function spaces. This will be described in detail in the following subsections.\\

Let us conclude this overview with a comparison between our prior and present paper. Essentially, while our prior paper \cite{CHH_translators} was about the translator equation \eqref{trans_eq_intro}, the present one is about the linearized translator equation \eqref{lin_trans_eq_intro}. The linearized equation is less geometric. Hence, some key estimates that are obvious or well established in the context of \eqref{trans_eq_intro}, including in particular the avoidance principle \cite{Ilmanen_book} and the shrinker foliation from \cite{ADS1}, are not at all obvious in the context of \eqref{lin_trans_eq_intro}. To overcome this, we  have to establish novel PDE barrier estimates that act as replacements for these geometric estimates.

\bigskip

\subsection{The equation in different gauges} The linearized translator equation in graphical gauge is
\begin{equation}\label{eq_diff_gauges_graphical}
\mathrm{div}\left(\left(\frac{\delta}{\sqrt{1+|D \phi|^2}}  -\frac{D\phi\otimes D\phi}{{(1+|D \phi|^2)^{3/2}}}\right)  Du\right)+ \frac{D \phi}{(1+|D \phi|^2)^{3/2}}\cdot Du =f.
\end{equation}
The ellipticity degenerates as the slope $|D\phi|$ tends to $\infty$. To better capture this phenomenon it is useful to transform the equation into different gauges. To this end, first recall that a suitably normalized noncollapsed translator $M=\textrm{graph}(\phi)\subset \mathbb{R}^4$, thanks to the $\mathrm{SO}_2$-symmetry, can be described in terms of the cylindrical profile function $V=V(x,t)$ of the level sets $M\cap \{x_4=-t\}$, which is implicitly defined be the equation
\begin{equation}
\phi(x,V(x,t),0)=-t.
\end{equation}
We then consider the function
\begin{equation}
\bw(x,t):=-V_t(x,t)\, u(x,V(x,t),0),
\end{equation}
which we call the \emph{cylindrical variation} associated to the graphical variation $u$. This definition can be motivated geometrically by considering the one parameter family $M^\eps=\textrm{graph}(\phi-\eps u)$ and differentiating at $\eps=0$. The prefactor $-V_t>0$ geometrically comes from the horizontal projection. In a similar vein, regarding the right hand side of the equation we consider the function
\begin{equation}\label{bg_def_intro}
\bg(x,t) :=\sqrt{1+V_x^2(x,t)+V_t^2(x,t)}\, f(x,V(x,t),0),
\end{equation}
which we call the \emph{cylindrical inhomogeneity} associated to the inhomogeneity $f$. Most of the time we will actually work with the renormalized versions of these functions, specifically with
\begin{equation}
w(y,\tau):= e^{\frac{\tau}{2}}\bw\big(e^{-\frac{\tau}{2}}y,-e^{-\tau}\big)\quad\textrm{and}\quad g(y,\tau):= e^{-\frac{\tau}{2}}\bg\big(e^{-\frac{\tau}{2}}y,-e^{-\tau}\big).
\end{equation}
As we will see, in terms of $w$ and $g$ the linearized translator equation takes the form
\begin{equation}\label{eq_intro_w}
-w_\tau+\left(1-\frac{v_{y}^2}{1+ v_y^2}\right)w_{yy}-\left(\frac{y}{2}+\frac{2v_{y}v_{yy}}{(1+v_{y}^2)^2}\right) w_y+\left(1+\frac{2-v^2}{2v^2}\right)w+e^\tau \mathcal{F}w=g,
\end{equation}
where $v(y,\tau):=e^{\tau/2}V(e^{-\tau/2}y,-e^{-\tau})$ denotes the renormalized profile function. Here, one can decompose $\mathcal{F}=\alpha \partial^2_{\tau}+\beta \partial_{\tau y}+\widetilde{\mathcal{F}}$, where the linear operator $\widetilde{\mathcal{F}}$ only involves the derivatives $\partial_y$, $\partial_{y}^2$ and $\partial_\tau$, so in principle our equation of course still an elliptic PDE in $y$ and $\tau$. However, since the ellipticity, in particular the coefficient $e^{\tau}\alpha>0$, decays exponentially as $\tau\to -\infty$, the equation is more amenable to parabolic techniques, specifically techniques for the parabolic Ornstein-Uhlenbeck operator
\begin{equation}
-\partial_\tau +\mathfrak{L}=-\partial_\tau + \partial_y^2 -\tfrac{y}{2}\partial_y+1.
\end{equation}

In the tip region, namely for $v(\cdot,\tau)\leq \theta$, where $\theta>0$ is a small fixed constant, it is better to work with the inverse profile function. Specifically, denoting by $X(\cdot,t)$ the inverse function of $V(\cdot, t)$ we have
\begin{equation}
\phi(X(\upsilon,t),\upsilon,0)=-t.
\end{equation}
In a similar vein as above, we then consider the function
\begin{equation}
\bW(\upsilon,t):=-X_t(\upsilon,t)\, u(X(\upsilon,t),\upsilon,0),
\end{equation}
which we call the \emph{tip variation} associated to $u$, and the function
\begin{equation}
\bG(\upsilon,t):=\sqrt{1+X_{\upsilon}^2(\upsilon,t)+X_t^2(\upsilon,t)}\, f(X(\upsilon,t),\upsilon,0),
\end{equation}
which we call the \emph{tip inhomogeneity} associated to $f$. Most of the time we will again work with the renormalized versions of these functions, specifically with 
\begin{equation}
W(v,\tau):= e^{\frac{\tau}{2}}\bW\big(e^{-\frac{\tau}{2}}v,-e^{-\tau}\big)\quad\textrm{and}\quad G(v,\tau):= e^{-\frac{\tau}{2}}\bG\big(e^{-\frac{\tau}{2}}v,-e^{-\tau}\big).
\end{equation}
As we will see, in tip gauge the linearized translator equation takes the form
\begin{equation}\label{evolve_inhom_tip_intro}
-W_\tau+\frac{W_{vv}}{1+Y_{v}^2}+\left(\frac{1}{v}-\frac{v}{2}-2\frac{Y_{vv}Y_{v}}{(1+Y_{v}^2)^2}\right)W_v+\frac{1}{2}W+e^{\tau} {\mathcal{F}}W=G\, ,
\end{equation}
where $Y(\cdot,\tau)$ denotes the inverse function of $v(\cdot,\tau)$. In particular, due to the exponentially decaying ellipticity coefficient coming from $e^{\tau}{\mathcal{F}}W$, the equation is again more amenable to parabolic techniques.\\

The most basic solution of the (homogenous) linearized translator equation is the constant function $u=1$, which geometrically arises from shifting the translator in $x_4$-direction. For this trivial solution we have $w\sim e^{\tau}$ in the parabolic region, and $W\sim |\tau|^{1/2}e^{\tau}$ in the tip region. Another important solution is the function $u$ that arises from varying the parameter in the Hoffman-Ilmanen-Martin-White family. This is less explicit, but as a consequence of the results of the present paper one obtains $u=x_1^2-\tfrac{1}{2}x_2^2-\tfrac{1}{2}x_3^2+o(|x|^2)$ near the origin, $|w|\leq C|\tau|^{-2}$ in the parabolic region, and $|W|\leq C |\tau|^{-1/2}$ in the tip region.

\bigskip

\subsection{Estimates for the linearized translator equation}
Let $M=\mathrm{graph}(\phi)\subset \mathbb{R}^4$ be a noncollapsed translator as above. For any $h>0$, the hypersurface $M\cap \{x_4<h\}$ can be expressed as graph over a domain $\Omega_h\subset\mathbb{R}^3$. Denote by $C^{k-2,\alpha}(\Omega_h/S^1)$ the space of all $f\in C^{k-2,\alpha}(\Omega_h)$ that are $S^1$-symmetric in the $x_2x_3$-variables. Given any $h<\infty$ and $f\in C^{k-2,\alpha}(\Omega_h/S^1)$, we consider the Dirichlet problem
\begin{equation}\label{Dirichlet_intro}
    \begin{cases}
      L_{\phi}u=f & \text{on  $\Omega_h$}\\
      u=0 & \text{on $\partial \Omega_h$},\\
    \end{cases}       
\end{equation}
where $L_\phi$ denotes the linear operator from \eqref{eq_diff_gauges_graphical}.
By standard theory, for any $h<\infty$ this problem has a unique solution $u\in C^{k,\alpha}(\Omega_h/S^1)$. Our main aim, motivated by developing a Fredholm theory for $h\to \infty$, is to establish estimates that are uniform in $h$.\\

Most importantly, we have two barrier estimates that allow us to relate the size of the solution in different regions. In the following theorems, $u$ denotes a solution of the Dirichlet problem \eqref{Dirichlet_intro} with inhomogeneity $f$, and $w,W$  and $g,G$ denote the associated variations and inhomogeneities in cylindrical and tip gauge, respectively. All constants $C<\infty$ are uniform, i.e. independent of $h$.\\

Our first main barrier estimate is an upper-lower estimate, which allows us to propagate smallness of the solution at any given height $h'$ to smallness at lower heights:

\begin{theorem}[upper-lower estimate]\label{upper_lower_intro}
There exists $h_0<\infty$, such that for every $h'\in [h_0,h]$ we have
\begin{align}
\sup_{ \Omega_{h'}\setminus\Omega_{h_0}}|u| \leq \sup_{ \partial\Omega_{h'}}|u|+Ch'\sup_{ \Omega_{h'}}\frac{|f|}{H}.
\end{align}
\end{theorem}
In particular, Theorem \ref{upper_lower_intro} (upper-lower estimate) serves as a substitute for the avoidance principle under mean curvature flow, which has been a key ingredient in all prior papers on translators, see e.g. \cite{Wang_convex,Haslhofer_bowl,Hershkovits_translators,BC,BC2,SpruckXiao,HIMW,Zhu,CHH_translators} among many others.
\\

Second, we have an inner-outer estimate, which allows us to propagate smallness in the parabolic region $|y|\leq \ell$ (here $\ell<\infty$ is a fixed large constant) to smallness in the intermediate and tip region. Fix an exponent $\mu\geq 0$ (we will later choose $\mu=1$, but the case $\mu=0$ is interesting as well).  Loosely speaking, our estimate says if $|w|\leq A/|\tau|^{1+\mu}$ in the parabolic region, then $|w|\leq CA(\sqrt{2}-v)/|\tau|^\mu$ in the intermediate region and $|W|\leq CA|\tau|^{1/2-\mu}$ in the tip region. More precisely, increasing $h_0$ we can assume that the asymptotics from \cite{CHH_translators} hold at all $\tau\leq\tau_0:=\log(h_0)$, and then the statement is at follows:

\begin{theorem}[inner-outer estimate]\label{inner_outer_intro}
Suppose that $A<\infty$ is a constant such that
\begin{equation}\label{lower_bd_basic_sup_intro}
\sup_{\tau\in [-\log(h)+1,\tau_0]}  |\tau|^{1+\mu}|w(\ell,\tau)|+ \sup_{\tau\in [-\log (h),-\log(h)+1]}  |\tau|^{1/2+\mu}|w(\ell,\tau)| +
 \sup_{x\in \partial \Omega_{h_0}}|u(x)|\leq A,
\end{equation}
and suppose that for all $\tau\in [-\log (h),\tau_0]$ we have
\begin{equation}\label{g_growth_basic_sup_intro}
\sup_{y \in \left[\ell,Y(\theta,\tau)\right]}    (\sqrt{2}-v)^{-2}|g(y,\tau)| +
\sup_{v\leq \theta} \left(|\tau|^{1/2}+\frac{1}{v^3}\min\left(1,v^2|\tau|/\ell^2\right)\right)^{-1} |G(v,\tau)| \leq A|\tau|^{-\mu}.
\end{equation}
Then, for all $\tau\in [-\log (h)^{1/2},\tau_0]$ we get
\begin{equation}
\sup_{y \in \left[\ell,Y(\theta,\tau)\right]}  (\sqrt{2}-v)^{-1}|w(y,\tau)| +
\sup_{v\leq \theta} |\tau|^{-1/2} |W(v,\tau)| \leq CA|\tau|^{-\mu}.
\end{equation}
\end{theorem}

In particular, Theorem \ref{inner_outer_intro} (inner-outer estimate) serves as a substitute for the shrinker foliation from \cite{ADS1}, which has been crucial for propagating information from the parabolic region to the intermediate and tip region, see e.g. \cite{ADS2,BC,BC2,CHH,CHHW,CHH_blowdown,CHH_translators,DH_shape,DuZhu,CDDHS,CH_R4,CDZ}.\\

In addition to these two barrier estimates, the other key estimates are the energy estimates and the Schauder estimates, which will be discussed in detail in Section \ref{energy_est_sec} and Section \ref{Holder_sec}, respectively.\\

Regarding the energy estimates, for now let us just mention that, under certain orthogonality conditions, combining the energy estimates in the cylindrical and tip region we obtain a decay estimate of the form
\begin{multline}
\||\tau|^{1+\mu}\fp_0w_{\cC}\|_{\fH,\infty}+\||\tau|^{2+\mu} (w_\cC-\fp_0(w_{\cC}))\|_{\cD,\infty} +\| |\tau|^{2+\mu} W_\cT \|_{2,\infty}\\
 \leq   C\left( \|w\|_{C^2_{\exp}(\cC)}+\|W\|_{C^2_{\exp}(\cT)}+\| |\tau|^{2+\mu}g_{\cC}\|_{\mathcal{D}^{\ast},\infty}+\| |\tau|^{1+\mu} G_\cT \|_{2,\infty}+\sup_{\Omega_h} \frac{|f|}{H_{\phi}}\right),
\end{multline}
where the various Gaussian norms are defined as usual. Except for the $\tau$-weights and the new terms coming from the inhomogeneity, this estimate is similar to the decay estimate for translators from \cite{CHH_translators}.\\

Regarding the Schauder estimates, for now let us just mention that because of the degenerating ellipticity one has to work, roughly speaking, in neighborhoods that are much longer in $x_4$-direction compared to the other directions. The trick to do this efficiently is to introduce the functions
\begin{equation}\label{unren_tilde_quant_intro}
\tilde{\bw}(x,s,t)={\bw}(x,s+t), \qquad {\tilde{\bg}}(x,s,t)={\bw}(x,s+t),
\end{equation}
This description is of course redundant. However, viewing $s$ as a spatial variable and $t$ as a time variable is the key to establish sharp Schauder estimates in our degenerating setting. In particular, this yields
\begin{equation}\label{Holder_W_away_intro}
\|\tilde{\bw}\|_{C^{k,\alpha}_H(P_{r/2}(x,t))}\leq C \left( \|\tilde{\bw}\|_{C^{0}_H(P_r(x,t))}+ \|\tilde{\bg}\|_{C^{k-2,\alpha,(2)}_H(P_r(x,t))}\right),
\end{equation}
where the norms are parabolic H\"older norms that use the mean curvature as a weight, and $r\sim H^{-1}$. Of course, since $\tilde{\bw}_s=\tilde{\bw}_t$ these parabolic bounds are actually some elliptic bounds (rather nonstandard, specifically anisotropic and with different behavior at small and large scales) in disguise. 
Furthermore, we have an $L^2\to L^\infty$ estimate,  and of course we have corresponding estimates for the tip variation as well.

\medskip

\subsection{Fredholm theory and nonlinear theory} We then package together our estimates to establish a Fredholm theory and nonlinear theory. To this end, we introduce suitable weighted H\"older norms, which capture the size in the cylindrical region and the soliton region defined by 
\begin{equation}
\mathcal{C}=\left\{ (x_1,t)\; \bigg{|}\; -t\geq h_0,\;V(x_1,t) \geq \ell\sqrt{\tfrac{|t|}{\log |t|}} \right\} ,\quad
\mathcal{S}=\left\{ (x_2,t)\; \bigg{|}\; -t\geq h_0,\; x_2 \leq \ell\sqrt{\tfrac{|t|}{\log|t|}} \right\},
\end{equation}
as well as the cap region $\{-t\leq 2h_0 \}$. Specifically, recalling that we denote by $\bw$ and $\bW$ the cylindrical and tip variation associated to $u$, and choosing $r\sim H^{-1}$, we introduce the domain H\"older norm
\begin{align}\label{Hold_def_intro}
\|u\|_{C^{k,\alpha}_{\star}(\mathbb{R}^3/S^1)}:=\|u\|_{C^{k,\alpha}(\Omega_{2h_0})}+\sup_{(x_1,t)\in\mathcal{C}}\frac{1}{\rho_\star(x_1,t)}\|\tilde{\bw}\|_{C_H^{k,\alpha}\left(P_{r}\left(x_1,t\right)\right)}+\sup_{(x_2,t)\in\mathcal{S}}\|\tilde{\bW}\|_{C_H^{k,\alpha}\left(Q_r\left(x_2,t\right)\right)}\, ,
\end{align}
where motivated by the inner-outer estimate we work with the weight function
\begin{equation}\label{domain_weight_intro}
\rho_\star(x,t):=\left\{\begin{array}{ll}
      \frac{1}{\log|t|}  \left( \sqrt{2}+\frac{10}{\log|t|}-\frac{V(x,t)}{\sqrt{|t|}}\right) & \text{if } V(x,t)\geq \theta \sqrt{|t|}\\
        \frac{1}{\log|t|} \frac{|t|}{V(x,t)^2} &  \text{if } V(x,t)< \theta \sqrt{|t|}.
        \end{array}\right.
\end{equation}
Similarly, recalling that we denote by $\bg$ and $\bG$ the cylindrical tip inhomogeneity associated to $f$, and choosing $r\sim H^{-1}$ as above, we introduce the target H\"older norm
\begin{align}\label{im_Hold_def_intro}
\|f\|_{C^{k-2,\alpha}_{\bullet}(\mathbb{R}^3/S^1)}:=\|f\|_{C^{k-2,\alpha}(\Omega_{2h_0})}+\sup_{(x_1,t)\in\mathcal{C}}\frac{1}{\rho_\bullet(x_1,t)}\|\tilde{\bg}\|_{C_H^{k-2,\alpha,(2)}\left(P_{r}\left(x_1,t\right)\right)}+\sup_{(x_2,t)\in\mathcal{S}}\|\tilde{\bG}\|_{C_H^{k-2,\alpha,(2)}\left(Q_r\left(x_2,t\right)\right)}\, ,
\end{align}
where again motivated by the inner-outer estimate we work with the weight function
\begin{equation}\label{target_weight_intro}
\rho_\bullet(x,t):=\left\{\begin{array}{ll}
         \frac{1}{\log|t|}\Big(\sqrt{2}+\frac{10}{\log|t|}-\frac{V(x,t)}{\sqrt{|t|}}\Big)^{2}  & \text{if } V(x,t)\geq \theta \sqrt{|t|}\\
        \frac{1}{\log|t|}+\frac{1}{(\log|t|)^{3/2}}\frac{|t|^{3/2}}{V(x,t)^3} &  \text{if } V(x,t)< \theta \sqrt{|t|}.
        \end{array}\right.
\end{equation}

Now, fixing a real number $\alpha\in (0,1)$, for any integer $k\geq 4$ we consider the Banach space norms
\begin{equation}
\|u\|_{\mathbb{X}^{k,\alpha}(\mathbb{R}^3/S^1)}:=\|u\|_{C^{k,\alpha}_{\star}(\mathbb{R}^3/S^1)}+ \|L_\phi u\|_{C^{k-2,\alpha}_{\bullet}(\mathbb{R}^3/S^1)},\qquad
\|f\|_{\mathbb{Y}^{k-2,\alpha}(\mathbb{R}^3/S^1)}:=\|f\|_{C^{k-2,\alpha}_{\bullet}(\mathbb{R}^3/S^1)},
\end{equation}
and show that the linearized translator operator $L_\phi$ is Fredholm:

\begin{theorem}[Fredholm property]\label{Fredholm_theorem_intro}
The map $L_\phi:\mathbb{X}^{k,\alpha}(\mathbb{R}^3/S^1)\rightarrow \mathbb{Y}^{k-2,\alpha}(\mathbb{R}^3/S^1)$ is Fredholm.
\end{theorem}

To prove this, we establish uniform Fredholm estimates for solutions of the Dirichlet problem \eqref{Dirichlet_intro}. Specifically, we show that for any sufficiently large $h<\infty$ and any solution $u$ of the Dirichlet problem that satisfies a suitable orthogonality condition with $3$ problematic eigenfunctions, we have
\begin{equation}
\|u\|_{\mathbb{X}^{k,\alpha}(\Omega_h/S^1)} \leq C\|f\|_{\mathbb{Y}^{k-2,\alpha}(\Omega_h/S^1)}.
\end{equation}
Using this estimate it is easy to conclude that $L_\phi$ is Fredholm with cokernel of dimension at most $3$.\\

Finally, we establish the nonlinear mapping properties of the translator operator
\begin{equation}
\Theta[\phi]=\mathrm{div}\left(\frac{D \phi}{\sqrt{1+|D \phi|^2}}\right)-\frac{1}{\sqrt{1+|D \phi|^2}}.
\end{equation}

\begin{theorem}[nonlinear mapping properties]\label{quad_err_put_together_intro}
There exists $\eps=\eps(\phi)>0$ such that the map\footnote{Note that we lose two derivatives in the nonlinear theory, but this can be easily dealt with using Frechet spaces.}
\begin{equation}
B_{\mathbb{X}^{k+2,\alpha}(\mathbb{R}^3/S^1)}(0,\eps)\to \mathbb{Y}^{k-2,\alpha}(\mathbb{R}^3/S^1),\qquad u\mapsto \Theta[\phi+u]
\end{equation}
is analytic, and its derivative is given by $L_{\phi+u}$.
\end{theorem}
To prove this, working with the complexification of the spaces and maps from above, given any base point $\phi_0=\phi+u_0$ with $ \| u_0 \|_{\mathbb{X}^{k+2,\alpha}(\mathbb{R}^3/S^1,\mathbb{C})}\leq \eps$,
 we consider the quadratic quantity
\begin{equation}
Q_{\phi_0}[u]:= \Theta[\phi_0+u]-\Theta[\phi_0]-L_{\phi_0} u,
\end{equation}
and show that 
\begin{align}
\| Q_{\phi_0}[ u ]\|_{\mathbb{Y}^{k-2,\alpha}(\mathbb{R}^3/S^1,\mathbb{C})}\leq C \| u \|_{\mathbb{X}^{k+2,\alpha}(\mathbb{R}^3/S^1,\mathbb{C})}^2,
\end{align}
provided that $\| u \|_{\mathbb{X}^{k+2,\alpha}(\mathbb{R}^3/S^1,\mathbb{C})}\leq \eps$.\\

Via Lyaponov-Schmidt reduction we can combine the above results to prove Theorem \ref{theorem_analyticity_intro} (analyticity).

\bigskip

\noindent\textbf{Acknowledgments.}
KC has been supported by the KIAS Individual Grant MG078902, an Asian Young Scientist Fellowship, and the National Research Foundation (NRF) grants  RS-2023-00219980 and RS-2024-00345403 funded by the Korea government (MSIT). RH has been supported by the NSERC Discovery Grant RGPIN-2023-04419. OH has been supported by ISF grant 437/20. This project has received funding from the European Research Council (ERC) under the European Union's Horizon 2020 research and innovation program, grant agreement No 101116390.

\bigskip

\section{Notation and preliminaries}

Recall that a mean-convex translator $M\subset\mathbb{R}^{4}$ is called noncollapsed if every $p\in M$ admits interior and exterior tangent balls of radius $\alpha/H(p)$ for some constant $\alpha>0$ (in fact one can take $\alpha=1$).  Hence, assuming without loss of generality that the translator moves with unit speed in positive $x_4$-direction, by \cite{HaslhoferKleiner_meanconvex} there exists a smooth convex function $\phi:\mathbb{R}^3\rightarrow \mathbb{R}$ such that $M=\mathrm{graph}(\phi)$ and
\begin{equation}\label{tran_eq_prelim}
\mathrm{div}\left(\frac{D \phi}{\sqrt{1+|D \phi|^2}}\right)=\frac{1}{\sqrt{1+|D \phi|^2}}.
\end{equation}

Conversely, given a convex entire function $\phi:\mathbb{R}^3\rightarrow \mathbb{R}$ satisfying \eqref{tran_eq_prelim}, the hypersurface $M=\mathrm{graph}(\phi)$ is a translator, which by \cite{BLL} is $\alpha$-noncollaped from some $\alpha>0$. 

If $M$ is $\mathrm{SO}_3$-symmetric, then up to rigid motion and scaling $M=\mathrm{Bowl_3}$. If $M$ is not strictly convex, then by \cite{Haslhofer_bowl} up to rigid motion and scaling $M=\mathbb{R}\times \mathrm{Bowl}_2$. We can thus assume from now on that  $\phi$ is nontrivial, namely strictly convex and not $\mathrm{SO}_3$-symmetric. Then, by \cite{CM_uniqueness} (see also \cite{White_nature,HaslhoferKleiner_meanconvex}) in suitable coordinates we have
\begin{equation}
\lim_{h\rightarrow \infty} \frac{M-he_4}{\sqrt{h}}=\mathbb{R}\times S^1(\sqrt{2})\times \mathbb{R}.
\end{equation}

By \cite{CHH_blowdown} the function $\phi$ has a unique minimum. After shifting coordinates, we can assume without loss of generality that $\phi(x)\geq 0$ with equality if and only if $x=0$.
By \cite{Zhu} (see also \cite[Theorem 2.5]{CHH_translators}) the hypersurface
$M$ is invariant under the action of $S^1=\{1\}\times \mathrm{SO}_2\times \{1\}$ on $\mathbb{R}^4$ by rotations.
Hence, we can write
\begin{equation}\label{V_def}
M\cap \{x_4=h\}= \big\{(x_1,V(x_1,-h)\cos\vartheta,V(x_1,-h)\sin \vartheta,h)\;|\;x_1 \in [-d^{-}(h),d^{+}(h)],\vartheta\in [0,2\pi]\big\}.
\end{equation}
for some function $V$, called the \emph{unrenormalized cylindrical profile function}. The function
\begin{equation}\label{v_def}
v(y,\tau):=e^{\tau/2}V(e^{-\tau/2}y,-e^{-\tau}),
\end{equation}
is called the \emph{renormalized cylindrical profile function}. Moreover, in the tip regions we define $Y(\cdot,\tau)$ as the inverse function of $v(\cdot,\tau)$, and let
\begin{equation}\label{def_Z}
Z(\rho,\tau)= \sqrt{|\tau|/2}\left(Y\big(\rho\sqrt{2/|\tau|},\tau\big)-Y(0,\tau)\right).
\end{equation}
The following theorem summarizes the precise qualitative behaviour of $M$ at infinity:
\begin{theorem}[sharp asymptotics {\cite[Theorem 3.11 and Corollary 5.8]{CHH_translators}}]\label{thm_unique_asympt_recall}
For every $\eps>0$ there exists $\tau_0=\tau_0(M,\eps)>-\infty$ such that for every $\tau\leq \tau_0$ the following precise asymptotic hold:
\begin{enumerate}
\item Parabolic region: The renormalized profile function satisfies
\begin{equation}
\left| v(y, \tau)-\sqrt{2}\left(1-\frac{y^2-2}{4 |\tau|}\right) \right| \leq\frac{\eps}{|\tau|} \qquad (|y |\leq \eps^{-1}).
\end{equation}
\item Intermediate region: The function $\bar{v}(z,\tau):=v(|\tau|^{1/2}z,\tau)$ satisfies
\begin{equation}\label{barzbd}
|\bar{v}(z,\tau)-\sqrt{2-z^2}|\leq \eps,
\end{equation}
on $[-\sqrt{2}+\eps,\sqrt{2}-\eps]$.
\item Soliton regions: We have the estimate
\begin{equation}
\| Z(\cdot,\tau)-Z_0(\cdot)\|_{C^{100}(B(0,\eps^{-1}))}\leq \eps,
\end{equation}
where $Z_0(\rho)$ is the profile function of the $2d$-bowl with speed $1$.
\end{enumerate}
Moreover, for $\tau\leq\tau_0$ in the collar region $\ell |\tau|^{-1/2}\leq  v(\cdot,\tau) \leq 2 \theta$ we have
\begin{equation}
\big|y (v^2)_y+4\big|\leq \eps\, ,
\end{equation}
provided $\theta=\theta(\eps)>0$ is sufficiently small and $\ell=\ell(\eps)< \infty$ is sufficiently large.
\end{theorem}

In particular, the sharp asymptotics imply that whenever $\phi(x)$ is sufficiently large then
\begin{equation}\label{H_two_sided}
\frac{1}{2\sqrt{\phi(x)}}\leq H_{\phi}(x) \leq 2\sqrt{\frac{\log{\phi(x)}}{\phi(x)}}.
\end{equation}
Moreover, let us point out that \cite[Proof of Theorem 3.11]{CHH_translators} for $\tau\leq\tau_0$ and $y\geq 2$ yields
\begin{equation}\label{profile_growth}
\frac{y^2}{C|\tau|}\leq \sqrt{2}-v\leq \frac{Cy^2}{|\tau|}.
\end{equation}
Consequently, for $\tau\leq\tau_0$ and $v(y,\tau)\geq \theta/2$ we obtain
\begin{equation}\label{profile_derivative_growth}
|v_y| \leq \frac{C}{|\tau|}(1+|y|),
\end{equation}
and thus in particular
\begin{equation}\label{impr_cyl_est}
|v_y| + |v_{yy}|+|v_{yyy}|\leq \frac{C}{\sqrt{|\tau|}}.
\end{equation}
Alternatively, \eqref{profile_growth} and \eqref{profile_derivative_growth} also follow directly from the global gradient estimate from \cite{CHH_profile}.
Also recall that by \cite[Lemma 5.6]{CHH_translators} for $\tau\leq\tau_0$ and $v(y,\tau)\geq \ell|\tau|^{-1/2}$ we have the cylindrical estimates
\begin{equation}\label{eq_cyl_est}
|v_y| + v|v_{yy}|+v^2|v_{yyy}|+|v(v_\tau+\tfrac{y}{2}v_y-\tfrac12 v)+1|\leq \eps,\qquad v^2|v_{y\tau}|+v^3|v_{\tau\tau}|\leq C,
\end{equation}
and by \cite[Lemma 5.18 and Proposition 5.20]{CHH_translators} we have the tip estimates
\begin{equation}\label{eq_tip_est}
\frac{1}{4}|\tau|^{1/2}\leq \left|\frac{Y_v}{v}\right|\leq |\tau|^{1/2}, \qquad |Y_\tau|\leq \eps \left|\frac{Y_v}{v}\right|,\qquad |Y_{vv}|+|Y_{v\tau}|+|Y_{\tau\tau}|\leq C|\tau|^{5/2}.
\end{equation}

Throughout this paper, we fix a small constant $\theta>0$, a large constant $\ell<\infty$, and a very negative constant $\tau_0>-\infty$. By convention, these constants can be adjusted at finitely many instances.

\bigskip

\section{Transformation to different gauges}

Recall that the graphical translator operator is given by the formula
\begin{equation}\label{theta_def}
\Theta[\phi]=\mathrm{div}\left(\frac{D \phi}{\sqrt{1+|D \phi|^2}}\right)-\frac{1}{\sqrt{1+|D \phi|^2}}.
\end{equation}
Fix a noncollapsed translator $M=\mathrm{graph}(\phi)\subset \mathbb{R}^4$ as above, and consider the linearization
\begin{equation}\label{L_def}
Lu:=\frac{d}{d\eps}|_{\eps=0} \Theta[\phi+\eps u].
\end{equation}
Explicitly, the $L$-operator in graphical gauge is given by
\begin{equation}\label{ell_equation}
Lu=\mathrm{div}(a_\phi  Du)+ b_\phi\cdot Du,
\end{equation}
where
\begin{equation}
a_\phi=\frac{\delta}{\sqrt{1+|D \phi|^2}}  -\frac{D\phi\otimes D\phi}{{(1+|D \phi|^2)^{3/2}}} , \qquad b_\phi=\frac{D \phi}{(1+|D \phi|^2)^{3/2}}.
\end{equation}

\medskip

To motivate the following computations, consider the one-parameter family of hypersurfaces $M^\eps=\mathrm{graph}(\phi^\eps)$, where $\phi^\eps=\phi-\eps u$. Denote by $v^\eps$ the renormalized cylindrical profile function of $M^\eps$, and set\footnote{The signs are compatible with the geometric fact that if the graph moves downwards then the level sets move outwards.}
\begin{equation}
w:=\frac{d}{d\eps}\Big|_{\eps=0}v^\eps.
\end{equation}
Recall that the unrenormalized cylindrical profile function $V^\eps$ of $M^\eps$ is defined by
\begin{equation}
M^\eps\cap \{x_4=-t\} = \big\{ (x_1,V^\eps(x_1,t)\cos\vartheta,V^\eps(x_1,t)\sin\vartheta,-t) \, |\, V^\eps(x_1,t)\geq 0,\vartheta\in[0,2\pi]\big\}.
\end{equation}
By symmetry, we can choose $\vartheta=0$. Since $M^\eps=\mathrm{graph}(\phi^\eps)$, we then have $x_4=\phi^\eps(x_1,x_2,0)$ and thus
\begin{equation}\label{V_def}
V^\eps(x_1,-\phi^\eps(x_1,x_2,0))=x_2.
\end{equation}
In terms of the renormalized cylindrical profile function, c.f. \eqref{v_def}, this becomes
\begin{equation}\label{v_u_rel}
v^\eps\left(\frac{x_1}{\sqrt{\phi^\eps(x_1,x_2,0)}},-\log\phi^\eps(x_1,x_2,0)\right)=\frac{x_2}{\sqrt{\phi^\eps(x_1,x_2,0)}}.
\end{equation}
Differentiating this identity with respect to $\eps$ an evaluating at $0$ yields
\begin{equation}
w+v_y \frac{x_1 u}{2 \phi^{3/2}} +v_\tau \frac{u}{\phi}=\frac{x_2u}{2 \phi^{3/2} },
\end{equation}
where $\phi\equiv\phi(x_1,x_2,0)$.
Observing also that in terms of the variables $y= {x_1}/{\sqrt{\phi}}$ and  $\tau=-\log \phi$, equation \eqref{v_u_rel} evaluated at $\eps=0$ simply takes the form
$v(y,\tau)=x_2/\sqrt{\phi}$,
we thus infer that
\begin{equation}\label{wfromu}
w(y,\tau)=e^{\tau} \left(\frac{v}{2}-\frac{y}{2}v_y -v_{\tau}\right)\!\!(y,\tau)\,u(e^{-\frac{\tau}{2}}y,e^{-\frac{\tau}{2}}v(y,\tau),0).
\end{equation} 

\bigskip

\subsection{Equation in cylindrical gauge}

We call $w$ defined by \eqref{wfromu} the \emph{renormalized cylindrical variation} associated to the graphical variation $u$. In a similar vain, to the graphical inhomogeneity $f$ we associate the \emph{renormalized cylindrical inhomogeneity} 
\begin{equation}\label{g_cyl_trans}
g(y,\tau)=e^{-\frac{\tau}{2}} \sqrt{1+v_y^2(y,\tau)+e^{\tau}\left(v_{\tau}+\frac{y}{2}v_y-\frac{v}{2}\right)^2\!\!(y,\tau)}\, f(e^{-\frac{\tau}{2}}y,e^{-\frac{\tau}{2}}v(y,\tau),0).
\end{equation}

\begin{proposition}[renormalized cylindrical variation]\label{lin_cyl_eq_prop}
Suppose $Lu=f$. Then the renormalized cylindrical variation $w$ defined by \eqref{wfromu} satisfies
\begin{equation}\label{eq_w.evolution_lin}
-w_\tau+\mathfrak{L}w+{\mathcal{E}}w+e^{\tau}\mathcal{F}w=g,
\end{equation}
where $g$ is the renormalized cylindrical inhomogeneity defined in \eqref{g_cyl_trans}. Here, 
\begin{equation}\label{eq_def.E[w]_lin}
\mathfrak{L}=\partial_y^2-\frac{y}{2}\partial_y+1,\qquad\qquad \mathcal{E}=-\frac{v_{y}^2}{1+ v_y^2}\partial^2_{y}-\frac{2v_{y}v_{yy}}{(1+v_{y}^2)^2}\partial_y+\frac{2-v^2}{2v^2},
\end{equation}
and $\mathcal{F}$ is a second order linear differential operator that will be specified in the proof below.
\end{proposition}

\begin{proof}
Given the translator $M=\mathrm{graph}(\phi)$, the functions $u$ and $f$, and a point $x$, choose a one-parameter family of $S^1$-symmetric convex functions $\phi^\eps$ with compact support and $\phi^0=\phi$ such that
\begin{equation}
\frac{d}{d\eps}\Big|_{\eps=0}\phi^\eps=-u
\end{equation}
in a neighborhood (of the orbit) of the point $x$ under consideration. Denote by $V^\eps$ and $v^\eps$ the unrenormalized and renormalized profile functions of $M^\eps=\mathrm{graph}(\phi^\eps)$, and set
\begin{equation}
{\psi^\eps}:=-\Theta[{\phi^\eps}].
\end{equation}
\begin{claim}[renormalized profile function]\label{claim_ren_prof}
The renormalized profile function ${v}^\eps$ satisfies
\begin{multline}\label{cyl_gauge_trans}
-{v}^\eps_\tau+\frac{{v}^\eps_{yy}}{1+\left(v^\eps_y\right)^2}-\frac{y}{2}{v}^\eps_y+\frac{{v^\eps}}{2}-\frac{1}{{v^\eps}}+ e^{\tau}\cN[{v}^\eps]\\
=e^{-\frac{\tau}{2}}\sqrt{1+(v^\eps_y)^2+e^{\tau}\left({v}^\eps_{\tau}+\frac{y}{2}{v}^\eps_y-\frac{{v}^\eps}{2}\right)^2}{\psi}^\eps(e^{-\frac{\tau}{2}}y,e^{-\frac{\tau}{2}}{v^\eps},0),
\end{multline}
where $\mathcal{N}$ is the operator defined in \cite[Proposition 5.3]{CHH_translators}.
\end{claim}
\begin{proof}[Proof of the claim]
Parametrizing the hypersurface  ${M}^\eps=\mathrm{graph}(\phi^\eps)$ by
\begin{equation}
X^\eps(x,t,\vartheta)=(x,{V}^\eps(x,t)\cos\vartheta,{V}^\eps(x,t)\sin\vartheta,-t),
\end{equation}
and setting $e_r=\cos\vartheta e_2+\sin\vartheta e_3$, similarly as in \cite[Proof of Proposition 5.3]{CHH_translators}, the upper pointing unit normal is
\begin{equation}
N^\eps=-\frac{V_t^\eps e_4 - V_x^\eps e_1+e_r}{\sqrt{1+({V}^\eps_x)^2+({V}^\eps_t)^2}},
\end{equation}
and the mean curvature is given by
\begin{equation}
H^\eps=\left(\frac{(1+(V^\eps_t)^2)V^\eps_{xx}+(1+(V^\eps_x)^2)V_{tt}-2V^\eps_xV^\eps_tV^\eps_{xt}}{1+(V^\eps_x)^2+(V^\eps_t)^2}-\frac{1}{V^\eps}\right)\langle e_r,N^\eps\rangle.
\end{equation}
Since ${\psi^\eps}=\langle e_4,N^\eps\rangle-H^\eps$ by definition of the translator operator, this yields
\begin{multline}\label{V_eq_der}
-V^\eps_t+\frac{(1+(V^\eps_t)^2)V^\eps_{xx}+(1+(V^\eps_x)^2)V^\eps_{tt}-2V^\eps_xV^\eps_tV^\eps_{xt}}{1+(V^\eps_x)^2+(V^\eps_t)^2}-\frac{1}{V^\eps}\\
=\sqrt{1+(V^\eps_x)^2+(V^\eps_t)^2}\psi^\eps(x,V^\eps(x,t),0).
\end{multline}
Finally, differentiating the defining equation of the renormalized profile function \eqref{v_def} we see that
\begin{equation}
V^\eps_x=v^\eps_y,\;\;\;\;\;\;V^\eps_t=e^{\frac{\tau}{2}}\left(v^\eps_\tau+\frac{y}{2}v^\eps_y-\frac{v^\eps}{2}\right).
\end{equation}
Hence, transforming to the renormalized variables, similarly as in \cite[Proof of Proposition 5.3]{CHH_translators}, the claim follows. 
\end{proof}

Continuing the proof of the proposition, we consider the difference
\begin{equation}
w^\eps:=v^\eps-v.
\end{equation}
Then, using the claim and arguing similarly as in \cite[Proposition 5.9]{CHH_translators} we see that
\begin{equation}\label{w_s_tau_eq}
-w^{\eps}_\tau+\mathfrak{L}w^\eps+\mathcal{E}^\eps w^\eps+e^{\tau}\mathcal{F}^\eps w^\eps=e^{-\frac{\tau}{2}}\sqrt{1+\left(v^\eps_y\right)^2+e^{\tau}\left(v^\eps_{\tau}+\frac{y}{2}v^\eps_y-\frac{{v^\eps}}{2}\right)^2}{\psi^\eps}(e^{-\frac{\tau}{2}}y,e^{-\frac{\tau}{2}}{v^\eps},0),
\end{equation}
where 
\begin{equation}\label{eq_def.E[w]}
\mathcal{E}^\eps w^\eps=-\frac{(v^\eps_{y})^2}{1+ (v^\eps_y)^2}w^\eps_{yy}-\frac{(v^\eps_{y}+v_{y})v_{yy}}{(1+(v^\eps_{y})^2)(1+v_{y}^2)}w^\eps_y+\frac{2-v^\eps v}{2v^\eps v}w^\eps,
\end{equation}
and 
\begin{align}
\mathcal{F}^\eps w^\eps=&\frac{ \mathcal{P}[v^\eps,v^\eps,w^\eps]}{\mathcal{Q}[v^\eps,v^\eps]}+\mathcal{R}[v^\eps,v]\left(w^\eps_\tau-\tfrac{w^\eps}{2}\right)+\mathcal{S}[v^\eps,v]w^\eps_y,
\end{align}
and where $\mathcal{P},\mathcal{Q},\mathcal{R},\mathcal{S}$ are the second order differential expressions specified in the cited proof.

Now, computing in a suitable neighborhood of the point under consideration, by the discussion at the beginning of this subsection the renormalized cylindrical variation $w$ defined by \eqref{wfromu} satisfies
\begin{equation}
\frac{d}{d\eps}\Big|_{\eps=0}w^\eps=w.
\end{equation}
Moreover, by definition of the linearized translator operator we have
\begin{equation}
\frac{d}{d\eps}\Big|_{\eps=0}\psi^\eps=f.
\end{equation}
Hence, differentiating \eqref{w_s_tau_eq} we conclude that
\begin{align}
-w_{\tau}+ \mathfrak{L}w+\mathcal{E} w+e^{\tau}\mathcal{F} w=g,
\end{align}
where $g$ is given by \eqref{g_cyl_trans}, where $\mathcal{L}$ and $\mathcal{E}$ are given by \eqref{eq_def.E[w]_lin}, and where\footnote{The coefficients of $w_{\tau \tau}$ and $w_{\tau y}$ in the expression for $\mathcal{F}$ are $\alpha= \frac{(1+v_y^2)}{1+v_y^2+e^{\tau}\left(\frac{y}{2}v_y+v_{\tau}-\frac{v}{2}\right)^2}$ and $\beta=  \frac{-2\left( v_{y}(v_{\tau} -\tfrac{v}{2})-\tfrac{y}{2}\right)}{1+v_y^2+e^{\tau}\left(\frac{y}{2}v_y+v_{\tau}-\frac{v}{2}\right)^2}$.}
\begin{align}\label{F_def}
\mathcal{F}w=&\frac{ \mathcal{P}[v,v,w]}{\mathcal{Q}[v,v]}+\mathcal{R}[v,v]\left(w_\tau-\tfrac{w}{2}\right)+\mathcal{S}[v,v]w_y.
\end{align}
This proves the proposition.
\end{proof}

\subsection{Equation in tip gauge}

Recall that the tip profile function $Y(\cdot,\tau)$ is defined as inverse of the function $v(\cdot,\tau)$, where we tacitly assume that we work with the right tip where $Y>0$ (the argument for the left tip is the same).
We call
\begin{equation}\label{Wfromu}
W(v,\tau):=e^{\tau}\left(\frac{Y}{2}-\frac{v}{2}Y_v  -Y_{\tau}\right)\!\!(v,\tau)\, u(e^{-\frac{\tau}{2}}Y(v,\tau),e^{-\frac{\tau}{2}}v,0)
\end{equation}
the \emph{tip variation} associated to $u$, and
\begin{equation}\label{g_tip_trans}
G(v,\tau):=e^{-\frac{\tau}{2}}\sqrt{1+Y_v^2(v,\tau)+e^{\tau}\left(Y_{\tau}+\frac{vY_v}{2}-\frac{Y}{2}\right)^2\!\!(v,\tau)}\, f(e^{-\frac{\tau}{2}}Y(v,\tau),e^{-\frac{\tau}{2}}v,0)
\end{equation}
the \emph{tip inhomogeneity} associated to $f$.

\begin{proposition}[tip variation]\label{lin_tip_eq_prop}
Suppose $Lu=f$. Then the tip variation $W$ defined by \eqref{Wfromu} satisfies
\begin{equation}\label{evolve_inhom_tip}
-W_\tau+\frac{W_{vv}}{1+Y_{v}^2}+\left(\frac{1}{v}-\frac{v}{2}-2\frac{Y_{vv}Y_{v}}{(1+Y_{v}^2)^2}\right)W_v+\frac{1}{2}W+e^{\tau}{\mathcal{F}}W=G\, ,
\end{equation}
where $G$ is the tip inhomogeneity defined by \eqref{g_tip_trans}, and where $\mathcal{F}$ is a second order linear differential operator that will be specified in the proof below.
\end{proposition}

\begin{proof}
Working with a suitable one-parameter family $M^\eps=\mathrm{graph}(\phi^\eps)$ as above, we consider $Y^\eps$ defined as the inverse function of $v^\eps$. Then, dealing with the inhomogeneity as before we see that
\begin{multline}\label{eq_Y.evolution}
-Y^\eps_\tau+\frac{Y^\eps_{vv}}{1+(Y^\eps_v)^2}+\frac{1}{v}Y^\eps_v +\frac{1}{2}(Y^\eps-vY^\eps_v)+e^{\tau}\mathcal{M}[Y^\eps]\\
=e^{-\frac{\tau}{2}}\sqrt{1+(Y^\eps_v)^2+e^{\tau}\left(Y^\eps_{\tau}+\frac{vY^\eps_v}{2}-\frac{Y^\eps}{2}\right)^2}\psi^\eps(e^{-\frac{\tau}{2}}Y^\eps,e^{-\frac{\tau}{2}}v,0),
\end{multline}
where $\mathcal{M}$ is the expression from \cite[Proposition 5.4]{CHH_translators}.
Differentiating this, we conclude that
\begin{equation}
-W_\tau+\frac{W_{vv}}{1+Y_{v}^2}+\left(\frac{1}{v}-\frac{v}{2}-2\frac{Y_{vv}Y_{v}}{(1+Y_{v}^2)^2}\right)W_v+\frac{1}{2}W+e^\tau \mathcal{F}W=G,
\end{equation}
where $W$ and $G$ are given by \eqref{Wfromu} and \eqref{g_tip_trans}, and where
\begin{equation}\label{F_DEF}
\mathcal{F}W=\frac{\mathcal{P}[Y,Y,W]}{\mathcal{Q}[Y,Y]}+\mathcal{R}[Y,Y]\left(\frac{W}{2}-W_\tau\right)+\mathcal{S}[Y,Y]W_v,
\end{equation}
with $\mathcal{P},\mathcal{Q},\mathcal{R},\mathcal{S}$ denoting the quantities from \cite[Proof of Proposition 5.11]{CHH_translators}.
\end{proof}

Finally, let us record the following simple transformation rule:
\begin{corollary}[transformation rule]\label{cor_trans_rule} We have
\begin{equation}\label{W_w_transition}
W(v,\tau)=-Y_v(v,\tau) w(Y(v,\tau),\tau),
\end{equation}
and
\begin{equation}\label{W_w_transition}
G(v,\tau)=-Y_v(v,\tau) g(Y(v,\tau),\tau).
\end{equation}
\end{corollary}

\begin{proof}
Using the same setting as above the first formula follows by differentiating the relation
\begin{equation}
 Y^\eps(v^\eps(y,\tau),\tau)=y.
\end{equation}
To proceed, observe that differentiating the identity $y=Y(v(y,\tau),\tau)$ gives
\begin{equation}
0=Y_\tau +Y_v v_\tau,  \qquad  1=Y_v v_y.
\end{equation}
Using this, the second formula follows by comparing the expressions in \eqref{g_cyl_trans} and \eqref{g_tip_trans}.
\end{proof}

\bigskip

\section{Barrier estimates}\label{sec_barrier_estimates}

Let $M=\mathrm{graph}(\phi)\subset \mathbb{R}^4$ be a noncollapsed translator as above. For any $h>0$, the hypersurface $M\cap \{x_4<h\}$ can be expressed as graph over a domain $\Omega_h\subset\mathbb{R}^3$. Denote by $C^{k-2,\alpha}(\Omega_h/S^1)$ the space of all $f\in C^{k-2,\alpha}(\Omega_h)$ that are $S^1$-symmetric in the $x_2x_3$-variables. Given $h<\infty$ and $f\in C^{k-2,\alpha}(\Omega_h/S^1)$, by standard elliptic theory, the Dirichlet problem
\begin{equation}\label{bdval_prob}
    \begin{cases}
      Lu=f & \text{on  $\Omega_h$}\\
      u=0 & \text{on $\partial \Omega_h$}.\\
    \end{cases}       
\end{equation}
has a unique solution $u\in C^{k,\alpha}(\Omega_h/S^1)$. Here, $L$ denotes the operator from equation \eqref{ell_equation}.

\subsection{The upper-lower estimate}

In this subsection, we construct a subsolution for the $L$-operator, which will allow us to relate the values of $u$ at different heights.

Note that since $\Theta[\phi]=0$, the mean curvature is given by the formula
\begin{equation}
H_\phi=\mathrm{div}\left(\frac{D \phi}{\sqrt{1+|D \phi|^2}}\right)=\frac{1}{\sqrt{1+|D \phi|^2}}.
\end{equation}
Moreover, recall that on a graph the metric and the second fundamental form are given by
\begin{equation}
g_\phi=\delta+D \phi\otimes D\phi,\qquad A_\phi = \frac{\mathrm{Hess}\phi}{\sqrt{1+|D\phi|^2}}.
\end{equation}
\begin{lemma}[$L$-operator]\label{lemma_Lopid}
If $M=\mathrm{graph}(\phi)\subset\mathbb{R}^4$ is a graphical translator, then 
\begin{align}\label{first_L_id}
L\phi = H_\phi -2A_\phi(e_4^\top,e_4^\top),
\end{align}
and
\begin{align}
L\log H_\phi\geq -|A_\phi|^2_{g_\phi}H_\phi+\frac{A_\phi(e_4^\top,e_4^\top)^2} {(1-H^2_\phi)H_\phi}.
\end{align}
\end{lemma}
\begin{proof}
By the above formulas we have
\begin{align}
L\phi &= H_\phi -\mathrm{div}\left(\frac{|D\phi|^2D \phi}{{(1+|D \phi|^2)^{3/2}}}\right) +\frac{|D\phi|^2}{(1+|D \phi|^2)^{3/2}}.
\end{align}
Setting $f=|D\phi|^2/(1+|D\phi|^2)$ and using the product rule in the form
\begin{equation}
\mathrm{div}\left(f\frac{D \phi}{\sqrt{1+|D \phi|^2}}\right)= \frac{Df\cdot D \phi}{\sqrt{1+|D \phi|^2}}+fH_\phi,
\end{equation}
we infer that
\begin{align}
L\phi = H_\phi -2\frac{\mathrm{Hess}\phi(D\phi,D\phi)}{(1+|D \phi|^2)^{5/2}}.
\end{align}
Since
\begin{align}
A_\phi(e_4^\top,e_4^\top)=\frac{\mathrm{Hess}\phi(D\phi,D\phi)}{(1+|D \phi|^2)^{5/2}},
\end{align}
this proves \eqref{first_L_id}. Next, using $\log H_\phi=-\tfrac{1}{2}\log(1+|D\phi|^2)$ we observe that
\begin{equation}
a_\phi D\!\log H_\phi=-\frac{\mathrm{Hess}\phi D\phi} {(1+|D \phi|^2)^{3/2}}  +\frac{\textrm{Hess}\phi(D\phi,D\phi)D\phi}{(1+|D \phi|^2)^{5/2}}.
\end{equation}
and
\begin{equation}
b_\phi \cdot D\!\log H_\phi = -\frac{\textrm{Hess}\phi(D\phi,D\phi)}{(1+|D \phi|^2)^{5/2}}.
\end{equation}
We continue by computing
\begin{equation}
-\textrm{div}\left(\frac{\textrm{Hess}\phi D\phi} {(1+|D \phi|^2)^{3/2}}\right) =-\frac{|\textrm{Hess}\phi|^2+D\Delta\phi\cdot D\phi }{(1+|D \phi|^2)^{3/2}}+3\frac{|\textrm{Hess}\phi D\phi|^2} {(1+|D \phi|^2)^{5/2}},
\end{equation}
and
\begin{equation}
\textrm{div}\left(\frac{\textrm{Hess}\phi(D\phi,D\phi)D\phi}{(1+|D \phi|^2)^{5/2}}\right)+b \cdot D\!\log H_\phi = D\left(\frac{\textrm{Hess}\phi(D\phi,D\phi)}{(1+|D \phi|^2)^{2}}\right)\cdot \frac{D\phi}{\sqrt{1+|D\phi|^2}}.
\end{equation}
Moreover, by the translator equation we can substitute
\begin{equation}
\frac{\textrm{Hess}\phi(D\phi,D\phi)}{(1+|D \phi|^2)^{2}}=\frac{\Delta\phi-1}{1+|D \phi|^2}.
\end{equation}
This yields 
\begin{align}
L\log H_\phi=-\frac{|\mathrm{Hess}\phi|^2 }{(1+|D \phi|^2)^{3/2}}+3\frac{|\mathrm{Hess}\phi D\phi|^2} {(1+|D \phi|^2)^{5/2}}-2\frac{(\mathrm{Hess}\phi(D\phi,D\phi))^2 }{(1+|D \phi|^2)^{7/2}}.
\end{align}
Now, recalling that the shape operator is given by
\begin{equation}
g_\phi^{-1}A_\phi=\frac{\mathrm{Hess}\phi}{\sqrt{1+|D\phi|^2}}-\frac{D\phi\otimes\mathrm{Hess}\phi D\phi}{(1+|D\phi|^2)^{3/2}},
\end{equation}
we compute
\begin{equation}
|A_\phi|^2_{g_\phi}=\mathrm{tr}\left((g_\phi^{-1}A_\phi)^2\right)=\frac{|\mathrm{Hess}\phi|^2 }{1+|D \phi|^2}-2\frac{|\mathrm{Hess}\phi D\phi|^2}{(1+|D \phi|^2)^{2}}+\frac{(\mathrm{Hess}\phi(D\phi,D\phi))^2 }{(1+|D \phi|^2)^{3}}.
\end{equation}
Moreover, by the Cauchy-Schwarz inequality we have
\begin{equation}
|\mathrm{Hess}\phi(D\phi,D\phi)|\leq |\mathrm{Hess}\phi D\phi| |D\phi|,
\end{equation}
hence
\begin{equation}
\frac{|\mathrm{Hess}\phi D\phi|^2}{(1+|D \phi|^2)^{5/2}}-\frac{(\mathrm{Hess}\phi(D\phi,D\phi))^2 }{(1+|D \phi|^2)^{7/2}}\geq \frac{(\mathrm{Hess}\phi(D\phi,D\phi))^2 }{(1+|D \phi|^2)^{7/2}}\frac{1}{|D\phi|^2}.
\end{equation}
Finally, observe that
\begin{equation}
1-H^2_\phi=\frac{|D\phi|^2}{1+|D \phi|^2}.
\end{equation}
Combining the above facts yields the assertion of the lemma.
\end{proof}
\begin{proposition}[global subsolution]\label{sub_sol_prop}
For every  noncollapsed translator $M=\mathrm{graph}(\phi)\subset\mathbb{R}^4$ there exists a constant $\varepsilon=\varepsilon(M)>0$ such that 
\begin{equation}\label{B_sup}
L[\phi+\log H_{\phi}]\geq \varepsilon H_{\phi}.
\end{equation}
\end{proposition}
\begin{proof}
Applying Lemma \ref{lemma_Lopid} ($L$-operator) we see that
\begin{align}\label{B_sup2}
\frac{L[\phi + \log H_\phi]}{H_\phi} \geq H_\phi^2-|A_\phi|^2_{g_\phi} +(1-H_\phi^2)\left(1- \frac{A_\phi(e_4^\top,e_4^\top)}{H_\phi(1-H_\phi^2) }\right)^2.
\end{align}
By convexity, we have $|A_\phi|_{g_\phi}\leq H_\phi$. Now, recall that by differentiating the translator equation one obtains the identity $\nabla H=-A(e_4^\top,\cdot)$. Together with the  local curvature estimate from \cite{HaslhoferKleiner_meanconvex} this implies
\begin{equation}
|A_\phi(e_4^\top,e_4^\top)|\leq C H_\phi^2.
\end{equation}
Hence, the right hand side of \eqref{B_sup2} is bigger than $1/4$ whenever $H\leq H_0:=\min\{C^{-1},1\}/4$. On the other hand, to deal with the region $\{H\geq H_0\}$ we can assume without loss of generality that the translator is strictly convex (if this is not the case, a similar argument applies after splitting off a line). Then, by \cite{CHH_blowdown} we have $H\to 0$ as $|x|\to \infty$. In particular, the region $\{H\geq H_0\}$ is compact, so in this region $H_\phi^2-|A_\phi|^2_{g_\phi}$ is bounded below by some $\eps=\eps(M)>0$. This concludes the proof of the proposition.
\end{proof}

Using our subsolution, we can now prove our main result of this subsection:

\begin{theorem}[upper-lower estimate]\label{max_princ_inhom}
There exist $C=C(M)<\infty$ and $h_0=h_0(M)<\infty$ depending only on $M=\textrm{graph}(\phi)$, such that for every $h'\in [h_0,h]$ the solution $u$ of the Dirichlet problem \eqref{bdval_prob} satisfies
\begin{align}\label{comp_levels_max}
\sup_{x\in \Omega_{h'}\setminus\Omega_{h_0}}|u(x)| \leq \sup_{y\in \partial\Omega_{h'}}|u(y)|+Ch'\sup_{y\in \Omega_{h'}}\frac{|f(y)|}{H_{\phi}(y)}.
\end{align}

\end{theorem}
\begin{proof} For ease of notation, let us abbreviate
\begin{equation}
K_{h'}:=\sup_{y\in \Omega_{h'}}\frac{|f(y)|}{H_{\phi}(y)}.
\end{equation}
By Proposition \ref{sub_sol_prop} (global subsolution) for every $x\in \Omega_{h'}$ we have
\begin{equation}
L\left[\frac{K_{h'}}{\eps}(\phi+\log H_{\phi})\pm u\right]\!(x) \geq K_{h'}H_{\phi}(x) \pm f(x)\geq 0.
\end{equation}  
Hence, for every $x\in \Omega_{h'}$ we get
\begin{equation}\label{max_prin_basic_eq}
\frac{K_{h'}}{\eps}(\phi(x)+\log H_{\phi}(x))\pm u(x) \leq \sup_{y\in \partial \Omega_{h'}}\left(\frac{K_{h'}}{\eps}(\phi+\log H_{\phi})\pm u\right)(y).
\end{equation}
Now, by the mean curvature asymptotics from \eqref{H_two_sided} for $\phi(x)\geq h_0(M)$ we have
\begin{equation}\label{bar_bound}
0< \phi(x)+\log H_{\phi}(x) \leq \phi(x).
\end{equation}
 Hence, for every $ x\in \Omega_{h'}\setminus \Omega_{h_0}$ we get
\begin{equation}
|u(x)| \leq \frac{K_{h'}}{\eps}\sup_{y\in \partial{\Omega_{h'}}}\phi(y)+\sup_{y\in \partial{\Omega_{h'}}} |u(y)|.
\end{equation} 
Since $\phi(y)=h'$ on $\partial \Omega_{h'}$, this proves the theorem.
\end{proof}

As a corollary of the proof we also obtain:

\begin{corollary}[level set estimate]\label{cor_bound_est}
For every $h'\in [h_0,h]$ we have
\begin{equation}\label{raw_boundary_est}
\sup_{ x\in \partial \Omega_{h'}}|u(x)| \leq C\max\{h-h',\log h\}\sup_{ y\in \Omega_h}\frac{|f(y)|}{H_{\phi}(y)}.
\end{equation}
\end{corollary}

\begin{proof}
Using again \eqref{H_two_sided} note that if $\phi(x)=h'$ and $\phi(y)=h$ then 
\begin{equation}
(\phi+\log H_{\phi})(y)-(\phi+\log H_{\phi})(x) \leq 2\max\{h-h',\log h\}.
\end{equation}
Since $u(y)=0$ by the Dirichlet boundary condition, the estimate \eqref{max_prin_basic_eq} thus yields the assertion.
\end{proof}

\bigskip

\subsection{The inner-outer estimate} 

The goal of this subsection is to prove an inner-outer estimate that serves as a substitute for the shrinker-foliation barrier estimates from \cite{ADS1}.

Motivated by Proposition \ref{lin_cyl_eq_prop} (renormalized cylindrical variation) let us consider the operator
\begin{align}
L_{\cyl}w:=\frac{w_{yy}}{1+v_y^2}-\left(\frac{y}{2}+\frac{2v_{y}v_{yy}}{(1+v_{y}^2)^2}\right)w_y+\left(\frac{1}{2}+\frac{1}{v^2}\right)w+e^{\tau}\mathcal{F}w-w_{\tau},
\end{align}
and motivated by Proposition \ref{lin_tip_eq_prop} (renormalized tip variation) let us consider the operator
\begin{equation}
L_{\tip}W:=\frac{W_{vv}}{1+Y_{v}^2}+\left(\frac{1}{v}-\frac{v}{2}-2\frac{Y_{vv}Y_{v}}{(1+Y_{v}^2)^2}\right)W_v+\frac{1}{2}W+e^\tau \mathcal{F}W-W_{\tau}.
\end{equation}
Assuming $Lu=f$, if $w$ and $W$ are obtained from $u$ by \eqref{wfromu} and \eqref{Wfromu}, and if $g$ and $G$ are obtained from $f$  by \eqref{eq_w.evolution_lin} and \eqref{evolve_inhom_tip}, respectively, then by the cited propositions we have the equations
\begin{equation}
L_{\cyl}w=g,\qquad\qquad L_{\tip}W=G.
\end{equation}

We begin by constructing a suitable supersolution. Set
\begin{equation}
b_1=v^{-1}-2^{-1/2}
\end{equation}
and
\begin{equation}
b_2=|\tau|^{1/2}|v_y|-\Gamma v+100|v_y|\min\left\{ \frac{1}{v},\frac{|\tau|^{1/2}}{\ell}\left(1+\frac{Y_v}{10\ell}1_{\{ v\leq 2\ell |\tau|^{-1/2}\}} \right)  \right\},
\end{equation}
where $\Gamma<\infty$ is a numerical constant, which will be chosen below.
Moreover, fix a monotone smooth function $\chi:\mathbb{R}\rightarrow \mathbb{R}_{+}$ with $\chi(v)=0$ for $v\leq {\theta}/2$ and $\chi(v)=1$ for $v\geq \theta$, and set
\begin{equation}\label{b_def}
b=(\chi\circ v)b_1+\Lambda(1-\chi\circ v)b_2,
\end{equation}
where $\Lambda=\Lambda(\theta)<\infty$ will be chosen below.
Finally, motivated by Corollary \ref{cor_trans_rule} (transformation rule) set
\begin{equation}
B(v,\tau)=-Y_v(v,\tau) b(Y(v,\tau),\tau).
\end{equation}

\begin{proposition}[supersolution]\label{basic_super_sol} 
If $\tau \leq \tau_0$, then for $y\in [\ell,Y(\theta,\tau)]$ we have
\begin{equation}\label{basic_supuer_sol_cyl_ineq}
L_{\cyl} b \leq -\frac{1}{4} b^2,
\end{equation}
and for $v\leq 2\theta$ we have 
\begin{equation}\label{basic_supuer_sol_tip_ineq}
L_{\tip}B \leq - |\tau|^{1/2}-\frac{1}{v^3}\min(1,v^2|\tau|/\ell^2).
\end{equation}
\end{proposition}

\begin{proof}
For the cylindrical region, recall that by Claim \ref{claim_ren_prof} (renormalized profile function) and the translator equation we have
\begin{align}\label{cyl_gauge_trans_rec}
v_\tau=\frac{v_{yy}}{1+v_y^2}-\frac{y}{2}v_y+\frac{v}{2}-\frac{1}{v}+ e^{\tau}\cN[v].
\end{align}
This implies
\begin{equation}
L_{\cyl}v=-\frac{2v_y^2v_{yy}}{(1+v_y^2)^2}+\frac{2}{v}+e^{\tau}(\mathcal{F}v-\cN[v]).
\end{equation}
To proceed, note that for any $w>0$ one has
\begin{equation}\label{l_cyl_modified_eq}
L_{\cyl}w^{-1}=-\frac{1}{w^2}L_{\cyl}w+\frac{2}{w^3}\frac{w_y^2}{1+v_y^2}+\frac{2}{w}\left(\frac{1}{2}+\frac{1}{v^2}\right)+e^{\tau}\left(\frac{\mathcal{F}w}{w^2}+\mathcal{F}w^{-1}\right).
\end{equation} 
This yields
\begin{align}
L_{\cyl}v^{-1}= \frac{1}{v}+\frac{2v_y^2}{v^2(1+v_y^2)}\left(\frac{v_{yy}}{1+v_y^2}+\frac{1}{v}\right)+e^{\tau}\left(\mathcal{F}v^{-1}+\frac{\mathcal{N}[v]}{v^2}\right).
\end{align}
Also observe that
\begin{equation}
L_{\cyl}1=\frac{1}{2}+\frac{1}{v^2}+e^{\tau}\mathcal{F}1.
\end{equation}
Combining these two formulas we therefore obtain
\begin{equation}
L_{\cyl}\left(v^{-1}-2^{-1/2}\right)=-2^{-1/2}(v^{-1}-2^{-1/2})^2+E,
\end{equation}
where
\begin{equation}
E=\frac{2v_y^2}{v^2(1+v_y^2)}\left(\frac{v_{yy}}{1+v_y^2}+\frac{1}{v}\right)+e^{\tau}\left(\mathcal{F}v^{-1}+\frac{\mathcal{N}[v]}{v^2}-2^{-1/2}\mathcal{F}1\right).
\end{equation}
Now, for $\tau\leq\tau_0$ and $y\in[\ell,Y(\theta/2,\tau)]$, by Theorem \ref{thm_unique_asympt_recall} (sharp asymptotics) and the cylindrical estimates from \eqref{eq_cyl_est} we can safely estimate 
\begin{equation}
|E| \leq 4\frac{v_y^2}{v^3},
\end{equation}
and by the profile estimates from \eqref{profile_growth} and \eqref{profile_derivative_growth} we have
\begin{equation}
v_y^2\leq \frac{\theta}{100}\left(\sqrt{2}-v\right)^2.
\end{equation}
Combining the above estimates we infer that 
\begin{equation}\label{Lcylb1}
L_{\cyl}b_1 \leq -\frac{1}{4}b_1^2
\end{equation}
for $\tau\leq\tau_0$ and $y\in[\ell,Y(\theta/2,\tau)]$.\\

In the tip region we work with the function $B_2=-Y_v b_2$, namely
\begin{equation}
B_2= |\tau|^{1/2}-\Gamma v|Y_v|+100\min\left\{ \frac{1}{v},\frac{|\tau|^{1/2}}{\ell}\left(1-\frac{|Y_v|}{10\ell}1_{\{ v|\tau|^{1/2}\leq 2\ell \}}\right)  \right\}.
\end{equation}
We will first derive an inequality for the more basic function
\begin{equation}
B_2^\circ= |\tau|^{1/2}-\Gamma v|Y_v|.
\end{equation}
Recall that by equation \eqref{eq_Y.evolution} we have
\begin{equation}\label{Y_ev_eq_calc}
Y_\tau=\frac{Y_{vv}}{1+Y_v^2}+\left(\frac{1}{v}-\frac{v}{2}\right)Y_v +\frac{1}{2}Y+e^{\tau}\mathcal{M}[Y].
\end{equation}
Differentiating this we get
\begin{equation}
Y_{v\tau}=\frac{Y_{vvv}}{1+Y_v^2}-\frac{2Y_vY_{vv}^2}{(1+Y_v^2)^2}+\left(\frac{1}{v}-\frac{v}{2}\right)Y_{vv}-\frac{1}{v^2}Y_v+e^{\tau}\left(\mathcal{M}[Y]\right)_v.
\end{equation}
This yields
\begin{equation}\label{eq_Ltip_Yv}
L_{\tip}Y_v=\left( \frac{1}{v^2}+\frac{1}{2} \right) Y_v+e^{\tau}\left(\mathcal{F}Y_v-(\mathcal{M}[Y])_v\right).
\end{equation}
Moreover, a direct computation shows that
\begin{equation}
L_{\tip}v=\frac{1}{v}-\frac{2Y_{vv}Y_v}{(1+Y_v^2)^2}+e^{\tau}\mathcal{F}v.
\end{equation}
To proceed, observe that for any $U$ and $V$ one has the product formula
\begin{equation}
L_{\tip}(UV)=UL_{\tip}V+VL_{\tip}U+\frac{2U_vV_v}{1+Y_v^2}-\frac{1}{2}UV+e^{\tau}\left(\mathcal{F}(UV)-U\mathcal{F}V-V\mathcal{F}U\right).
\end{equation}
This yields
\begin{align}
L_{\tip}(vY_v)=\left(\frac{2}{v}+\frac{v}{2}\right) Y_v-\frac{2Y_v^2Y_{vv}}{(1+Y_v^2)^2}-e^{\tau}\left(\mathcal{F}(vY_v)-v(\mathcal{M}[Y])_v\right).
\end{align}
Also note that
\begin{equation}
L_{\tip}|\tau|^{1/2}=\frac{1}{2}\left(1-\frac{1}{|\tau|}\right)|\tau|^{1/2}+e^{\tau}\mathcal{F}|\tau|^{1/2}.
\end{equation}
Combining these two formulas we therefore obtain
\begin{equation}
L_{\tip}B_2^\circ=\Gamma \left(\left(2+\frac{v^2}{2}\right)\frac{Y_v}{v}-\frac{2Y_v^2Y_{vv}}{(1+Y_v^2)^2}\right)+\frac{1}{2}\left(1-\frac{1}{|\tau|}\right) |\tau|^{1/2}+E,
\end{equation}
where
\begin{equation}
E=e^{\tau}\left( \mathcal{F}|\tau|^{1/2}+\Gamma \left(v(\mathcal{M}[Y])_v -\mathcal{F}(vY_v)\right) \right).
\end{equation}
Next, we note that there is a numerical constant $\gamma<1$, such that for $\tau\leq \tau_0$ and $v\leq 2\theta$ we have
\begin{equation}\label{curv_ratio_lessthan1}
\left|\frac{Y_v^2Y_{vv}}{(1+Y_v^2)^2}\right| \leq \gamma\frac{|Y_v|}{v}.
\end{equation}
Indeed, in the proof of \cite[Lemma 5.19]{CHH_translators} (cylindrical estimate) it has been observed that
\begin{equation}
\left|\frac{vY_vY_{vv}}{(1+Y_v^2)^2}\right|=\frac{Y_v^2}{1+Y_v^2}\frac{\lambda_1}{\lambda_2},
\end{equation}
where $\lambda_1$ and $\lambda_2$ are the principal curvatures, and that $\lambda_1/\lambda_2\ll 1$ in the collar region. Together with the fact that $\lambda_1/\lambda_2\leq 1$ on the bowl soliton and Theorem \ref{thm_unique_asympt_recall} (sharp asymptotics) this implies \eqref{curv_ratio_lessthan1}.
Moreover, by the tip estimates from \eqref{eq_tip_est} for $\tau\leq \tau_0$ and $v\leq 2\theta$ we have
\begin{equation}
 \frac{Y_v}{v} \leq -\frac{1}{4}|\tau|^{1/2} \qquad \textrm{and}\qquad  |E|\leq\frac{1}{100}\frac{|Y_v|}{v}.
\end{equation}
We now choose $\Gamma=1000/(1-\gamma)$. Moreover, possibly after decreasing $\theta$ we can assume that
\begin{equation}\label{corr_term_small}
\Gamma v |Y_v|\leq \tfrac{1}{2}|\tau|^{1/2}.
\end{equation}
Combining the above estimates, with this choice of constants, for all $\tau\leq \tau_0$ and $v\leq 2\theta$ we get
\begin{equation}\label{Lcylb2}
L_{\tip}B_2^\circ \leq -2|\tau|^{1/2}.
\end{equation}

Furthermore, a direct computation shows that
\begin{equation}
L_{\tip}v^{-1}=\frac{1}{v^3}\frac{2}{1+Y_v^2}-\frac{1}{v^2}\left(\frac{1}{v}-\frac{v}{2}-2\frac{Y_{vv}Y_v}{(1+Y_v^2)^2}\right)+\frac{1}{2v}+e^{\tau}\mathcal{F}\frac{1}{v}.
\end{equation}
Observing also that in the collar region we have
\begin{equation}
\frac{1}{1+Y_v^2}\ll 1\qquad\mathrm{and}\qquad
\left| \frac{vY_{vv}Y_v}{(1+Y_v^2)^2}\right | \leq \frac{10}{\ell^2},
\end{equation}
this yields
\begin{equation}
L_{\tip}v^{-1}\leq -\tfrac{1}{2}v^{-3}
\end{equation}
for $v \geq \ell |\tau|^{-1/2}$. On the other hand, thanks to \eqref{eq_Ltip_Yv}, for $v \leq 2\ell |\tau|^{-1/2}$ we can estimate
\begin{equation}
L_{\tip}\left(10\ell |\tau|^{1/2}-|Y_v||\tau|^{1/2}\right)\leq -\tfrac{1}{5}|\tau|/v.
\end{equation}
Remembering also that the minimum of two supersolutions is a supersolution, we thus infer that
\begin{equation}
L_{\tip}\left(B_2-B_2^\circ\right) \leq -2v^{-3}\min(1,v^2|\tau|/\ell^2).
\end{equation}

Finally, to deal with the transition region we observe that
\begin{equation}
L_{\cyl}b=\chi L_{\cyl}b_1+\Lambda(1-\chi)L_{\cyl}b_2+(\chi'\circ v)\left(\frac{1}{v}-\frac{v}{2}\right)(b_1-\Lambda b_2)+E,
\end{equation}
where
\begin{multline}
E= \left[(\chi''\circ v)\frac{v_y^2}{1+v_y^2}-(\chi'\circ v)\left(\frac{2v_y^2v_{yy}}{(1+v_y^2)^2}+e^{\tau}\cN[v] \right)\right](b_1-\Lambda b_2)\\
+2(\chi'\circ v)\frac{v_y}{1+v_y^2}(b_1-\Lambda b_2)_y
+e^{\tau}\left(\mathcal{F}b-(\chi\circ v) \mathcal{F}b_1-\Lambda (1-(\chi\circ v))\mathcal{F}b_2\right).
\end{multline}
We will now estimate these terms in the transition region $\theta/2\leq v\leq 2\theta$ for $\tau\leq \tau_0$. First of all, by our estimates \eqref{Lcylb1} and \eqref{Lcylb2}, remembering also Corollary \ref{cor_trans_rule} (transformation rule), we have
\begin{equation}
\chi L_{\cyl}b_1+\Lambda(1-\chi)L_{\cyl}b_2 \leq -2b.
\end{equation}
Next, thanks to \eqref{corr_term_small} we can fix $\Lambda=\Lambda(\theta)<\infty$ large enough to ensure that $b_1\leq \Lambda b_2$ in the transition region. Together with $\chi'\geq 0$ this implies
\begin{equation}
(\chi'\circ v)\left(\frac{1}{v}-\frac{v}{2}\right)(b_1-\Lambda b_2)\leq 0.
\end{equation}
Furthermore, using the derivative estimates from \eqref{impr_cyl_est}
we can estimate
\begin{equation}
|E|\leq C|\tau|^{-1/2}.
\end{equation}
Summing up, for $\theta/2\leq v\leq 2\theta$ for $\tau\leq \tau_0$ we thus obtain
\begin{align} 
L_{\cyl}b\leq -b 
\end{align}
Remembering Corollary \ref{cor_trans_rule} (transformation rule) this finishes the proof of the proposition.
\end{proof}

To deal with the Dirichlet boundary, we fix a smooth convex function  $\kappa:\mathbb{R}\to\mathbb{R}_+$ such that $\kappa(\tau)=|{\tau}|/{\sqrt{\log h}}$ for $\tau\leq-\log(h)+1$ and $\kappa(\tau)=2$ for $\tau\geq -\sqrt{{\log(h)}}$.  

\begin{corollary}[weighted supersolution]\label{weighted_super_sol} 
If $\tau \leq \tau_0$, then for $y\in [\ell,Y(\theta,\tau)]$ we have
\begin{equation}\label{weighted_supuer_sol_cyl_ineq}
L_{\cyl}(\kappa|\tau|^{-\mu} b) \leq -\tfrac{1}{8}  \kappa |\tau|^{-\mu} b^2,
\end{equation}
and $ v\leq 2\theta$ we have
\begin{equation}\label{weighted_supuer_sol_tip_ineq}
L_{\tip}(\kappa|\tau|^{-\mu} B) \leq -\tfrac{1}{2}\kappa |\tau|^{-\mu}\left(|\tau|^{1/2}+v^{-3}\min(1,v^2|\tau|/\ell^2)\right).
\end{equation}
\end{corollary}

\begin{proof} 
By convexity our weight function satisfies
\begin{equation}\label{pder_rel}
|\kappa'| \leq \frac{\kappa}{|\tau|}.
\end{equation}
Hence, for $\tau\leq \tau_0$ and $y\in[\ell,Y(\theta,\tau)]$ we infer that
\begin{equation}
L_{\cyl}(\kappa b)=\kappa L_{\cyl}b-\kappa' b +e^{\tau}(\mathcal{F}(\kappa b)-\kappa\mathcal{F}b)\leq -\frac{1}{4}\kappa b^2+\frac{\kappa}{|\tau|} b+e^{\frac{99\tau}{100}},
\end{equation}
where we used Proposition \ref{basic_super_sol} (supersolution) and the cylindrical estimates from \eqref{eq_cyl_est}. Since $b\geq 100/|\tau|$ for $y\geq\ell$
 by Theorem \ref{thm_unique_asympt_recall} (sharp asymptotics), this yields the first estimate. Arguing similarly, now using the tip estimates from \eqref{eq_tip_est}, we obtain the second estimate as well.
\end{proof}

Using our (weighted) supersolution we can now prove the main result of this subsection:

\begin{theorem}[inner-outer estimate]\label{linear_sup_sol_est}
Let $u$ be a solution of the Dirichlet problem \eqref{bdval_prob} with inhomogeneity $f$, and denote by $w$ and $W$ the associated variations and by  $g$ and $G$ the associated  inhomogeneities in cylindrical and tip gauge, respectively. Suppose that $A<\infty$ is a constant such that
\begin{equation}\label{lower_bd_basic_sup}
\sup_{\tau\in [-\log(h)+1,\tau_0]}  |\tau|^{1+\mu}|w(\ell,\tau)|+ \sup_{\tau\in [-\log (h),-\log(h)+1]}  |\tau|^{1/2+\mu}|w(\ell,\tau)| +
 \sup_{x\in \partial \Omega_{h_0}}|u(x)|\leq A,
\end{equation}
and suppose that for all $\tau\in [-\log (h),\tau_0]$ we have
\begin{equation}\label{g_growth_basic_sup}
\sup_{y \in \left[\ell,Y(\theta,\tau)\right]}  |\tau|^{\mu} (\sqrt{2}-v)^{-2}|g(y,\tau)| +
\sup_{v\leq \theta} |\tau|^{\mu}\left( |\tau|^{1/2}+v^{-3}\min(1,v^2|\tau|/\ell^2) \right)^{-1} |G(v,\tau)| \leq A.
\end{equation}
Then, for all $\tau\in [-\sqrt{\log (h)},\tau_0]$ we get
\begin{equation}
\sup_{y \in \left[\ell,Y(\theta,\tau)\right]} |\tau|^{\mu} (\sqrt{2}-v)^{-1}|w(y,\tau)| +
\sup_{v\leq \theta} |\tau|^{\mu-1/2} |W(v,\tau)| \leq CA.
\end{equation}
\end{theorem}

\begin{proof}
Recall that our equation in the respective gauges takes the form $L_{\cyl}w=g$ and $L_{\tip}W=G$. Hence, using Corollary \ref{weighted_super_sol} (weighted supersolution) and the assumptions of the theorem we infer that 
\begin{equation}
L_{\cyl}(\lambda A \kappa|\tau|^{-\mu} b \pm w) \leq 0\qquad\quad \forall (y,\tau)\in   [\ell,Y(\theta,\tau)]\times [-\log(h),\tau_0],
\end{equation}
and
\begin{equation}
L_{\tip}(\lambda A \kappa|\tau|^{-\mu} B \pm W) \leq  0\qquad\quad \forall (v,\tau)\in   [0,\theta]\times [-\log(h),\tau_0],
\end{equation}
provided the numerical factor $\lambda<\infty$ is sufficiently large. Now, consider the corresponding unrescaled domain 
\begin{equation}
D:=\left(\Omega_h-\Omega_{h_0}\right)\cap \Big\{x_1 \geq \ell\sqrt{\phi(x_1,x_2,x_3)}\Big\}.
\end{equation}
Let $u_A:D\to \mathbb{R}$ be the graphical function whose corresponding cylindrical and tip variations, obtained by the transformation rules \eqref{wfromu} and \eqref{Wfromu}, are given by $\lambda A\kappa|\tau|^{-\mu} b$ and $\lambda A\kappa|\tau|^{-\mu} B$, respectively. Rearranging \eqref{cyl_gauge_trans_rec} we see that the prefactor in \eqref{wfromu} satisfies
\begin{equation}\label{mul_vu}
\frac{{v}}{2}-\frac{y}{2}{v}_y-{v}_\tau=\frac{1}{{v}}-\frac{{v}_{yy}}{1+v_y^2}- e^{\tau}\cN[{v}] >0
\end{equation}
in the region under consideration, where we used Theorem \ref{thm_unique_asympt_recall} (sharp asymptotics) and the cylindrical estimates from \eqref{eq_cyl_est}, and similarly  we see that the prefactor in \eqref{Wfromu} satisfies
\begin{equation}\label{mul_Yu}
\frac{Y}{2}-\frac{v}{2}Y_v-Y_\tau>0
\end{equation}
in the region under consideration, where we used the tip estimates from \eqref{eq_tip_est}.
So in graphical gauge the above estimates take the form
\begin{equation}
L(u_A\pm u)\leq 0 \qquad \textrm{in}\;D. 
\end{equation}
Recall that $u=0$ on $\partial\Omega_{h}$, and note that by assumption \eqref{lower_bd_basic_sup} we have $u_A\pm u\geq 0$ on  $\partial\Omega_{h_0}$, provided we fix $\lambda=\lambda(h_0)<\infty$ sufficiently large.
Moreover, using again assumption \eqref{lower_bd_basic_sup}, but now also using Theorem \ref{thm_unique_asympt_recall} (sharp asymptotics)  and the fact that $\kappa(\tau)\geq \tfrac12 |\tau|^{1/2}$ for $\tau\in[-\log(h), -\log(h) +1]$, we see that $u_A\pm u\geq 0$ for $x_1=\ell\sqrt{\phi(x_1,x_2,x_3)}$ as well. This shows that
 \begin{equation}
u_A\pm u \geq 0 \qquad \textrm{on}\;\partial D.
\end{equation} 
Hence, applying the maximum principle for $L$ on $D$, we conclude that
\begin{equation}
u_A \pm u \geq 0\qquad \textrm{in}\; D.
\end{equation}
Remembering the transformation rules \eqref{wfromu} and \eqref{Wfromu}, and the fact that $\kappa(\tau)=2$ for $\tau \geq -\sqrt{\log(h)}$, this proves the theorem.
\end{proof}

\bigskip

\section{Energy estimates}\label{energy_est_sec}

In this section, we derive energy estimates for the linearized translator equation in various gauges. Throughout this section, we assume that $u$ is a solution of the Dirichlet problem \eqref{Dirichlet_intro} with inhomogeneity $f$. Moreover, we assume that $h\gg h_0 \gg 1$ is sufficiently large, and abbreviate $\tau_0:=-\log(h_0)$.\\

Recall that the Ornstein-Uhlenbeck operator 
\begin{equation}\label{frakLdef}
\mathfrak{L}=\partial^2_{y}-\frac{y}{2}\partial_y+1
\end{equation}
is self-adjoint operator on the Hilbert space $\fH:=L^2(\mathbb{R},e^{-y^2/4} dy)$. Decomposing $\fH$ according to the positive, neutral and negative eigenvalues of $\mathfrak{L}$, we write
\begin{equation}
\fH=\fH_+\oplus \fH_0 \oplus \fH_-.
\end{equation} 
Here, $\fH_+$ is spanned by the unstable eigenfunctions  $\psi_1=1$ and $\psi_2=y$, and $\fH_0$ is spanned by the neutral eigenfunction $\psi_0=y^2-2$. We write $\mathfrak{p}_\pm$ and $\mathfrak{p}_0$ for the orthogonal projections on $\fH_\pm$ and $\fH_0$.

\subsection{Energy estimate for the cylindrical variation}\label{cyl_en_sec}

In this subsection, we prove an energy estimate in the cylindrical region. Our estimate is related to the one appearing in \cite[Section 5.4]{CHH_translators}, with the important difference that we include some $\tau$-weights. Another new step is to control the boundary terms.\\

Recall that in addition to the Gaussian $L^2$-norm $\|-\|_{\fH}$ one also needs the Gaussian $H^1$-norm
\begin{equation}
\|p\|_{\mathcal{D}} := \left( \int (p^2 +p_y^2) e^{-y^2/4} dy \right)^{1/2},
\end{equation}
and its dual norm
\begin{equation}
\|p \|_{\mathcal{D}^\ast} := \sup_{\|q\|_{\mathcal{D}}\leq 1  } \langle p,q \rangle\, .
\end{equation}

It will be convenient to start the energy estimates at
\begin{equation}
\tau_{h_\mathrm{in}}:=-\log(h-h^{\gamma_k}),
\end{equation}
where $\gamma_k=1-\frac{1}{100k}$ and $k\geq 4$ is a given integer. This is on the one hand close enough to the Dirichlet boundary at $\tau_h=-\log(h)$ to inherit smallness from the vanishing boundary data, and on the other hand far enough away from the boundary to allow higher order estimates to kick in.

For functions $p:[\tau_{h_\mathrm{in}},\tau_0]\rightarrow \mathcal{X}$, where $\mathcal{X}=\fH,\mathcal{D}$ or $\mathcal{D}^\ast$, we set
\begin{equation}
\|p \|_{\mathcal{X},\infty}(\tau):=\sup_{\tau'\in [\tau_{h_\mathrm{in}}+1,\tau] }\left( \int_{\tau'-1}^{\tau'} \| p(\cdot,\sigma)\|^2_{\mathcal{X}} \, d\sigma \right)^{1/2},
\end{equation}
and we often simply abbreviate
\begin{equation}
\| p \|_{\mathcal{X},\infty}:=\| p\|_{\mathcal{X},\infty}(\tau_0).
\end{equation}

Recall that we denote by $w$ the cylindrical variation associated to $u$, and by $g$ the cylindrical cylindrical inhomogeneity associated to $f$.
Fixing a suitable cutoff function $\varphi_\cC:\mathbb{R}^+\to [0,1]$ such that
$\varphi_\cC(v)=0$ if $v\leq \tfrac58 \theta$ and $\varphi_\cC(v)=1$ if $v\geq  \tfrac78 \theta$,
we define their truncated versions by
\begin{equation}
w_{\cC}(y,\tau):=\varphi_{\cC}(v(y,\tau))w(y,\tau),\qquad
g_{\cC}(y,\tau):=\varphi_{\cC}(v(y,\tau))g(y,\tau). 
\end{equation}
To state our energy estimate for $w_\cC$, let us use the notation
\begin{align}
\| p \|_{C^2_{\exp}(\mathcal{C})}(\tau):=\sup_{\tau_{h_\mathrm{in}} \leq \tau'\leq \tau} \left(  e^{\frac{49\tau'}{100}}\sup_{y\in C_{\tau'} } \big(|p|+|p_y|+|p_{\tau}|+ |p_{yy}|+|p_{y\tau}|+|p_{\tau\tau}|\big)(y,\tau')\right),
\end{align}
where $C_{\tau}=\{y : v(y,\tau)\geq \tfrac58 \theta\}$, and let us simply abbreviate
\begin{equation}
\| p\|_{C^2_{\exp}(\mathcal{C})}:=\| p\|_{C^2_{\exp}(\mathcal{C})}(\tau_0)\, .
\end{equation}

\begin{proposition}[energy estimate for the cylindrical variation]\label{energy_cyl_inhom} If $\fp_ {+}(w_{\cC}(\tau_0))=0$, then
\begin{multline}
\| |\tau|^{2+\mu}(w_\cC-\fp_0w_\cC) \|_{\mathcal{D},\infty}\leq C\left( \||\tau|^{1+\mu} w_{\cC}\|_{\mathcal{D},\infty}+\||\tau|^{2+\mu} w\, 1_{\{\theta/2\leq v\leq\theta\}}\|_{\mathfrak{H},\infty}\right)\\
+C \| w\|_{C^2_{\exp}(\mathcal{C})}+C\||\tau|^{2+\mu} g_{\cC}\|_{\mathcal{D}^{\ast},\infty}+C\sup_{\Omega_h}\frac{|f|}{H_{\phi}},
\end{multline}
where $C=C(\phi)<\infty$ is independent of $h$.
\end{proposition}

\begin{proof}
Recall first from Proposition \ref{lin_cyl_eq_prop} (renormalized cylindrical variation) that the function $w$ satisfies
\begin{equation}
w_\tau=\mathfrak{L}w
+{\mathcal{E}}w+e^{\tau}\mathcal{F}w-g,
\end{equation}
To capture the extra terms from the cutoff, similarly as in \cite[Equation (6.11)]{ADS2}, we set
\begin{equation}
\overline{\mathcal{E}}[w,\varphi_{\cC}(v)]:= (\partial_\tau-\mathfrak{L})(w \varphi_{\cC}(v) )-\varphi_{\cC}(v)(\partial_\tau-\mathfrak{L})w +\varphi_{\cC}(v) \mathcal{E}w-\mathcal{E}(w\varphi_{\cC}(v)).
\end{equation}
Then, we have
\begin{align}\label{ev_eq_trunc_w}
(\partial_\tau -\mathfrak{L}) w_\cC =\mathcal{E}w_\cC+\overline{\mathcal{E}}[w,\varphi_\cC(v)]-g_{\cC}+e^{\tau}\varphi_\cC(v){\mathcal{F}}w.
\end{align}
Hence, the function $p:=|\tau|^{2+\mu} w_\cC$ evolves by
\begin{equation}\label{trunc_eq}
(\partial_\tau -\mathfrak{L}) p=q,
\end{equation}
where
\begin{equation}
q=-(2+\mu)|\tau|^{1+\mu} w_\cC+{\mathcal{E}}(|\tau|^{2+\mu} w_\cC)+|\tau|^{2+\mu}\overline{\mathcal{E}}[w,\varphi_\cC(v)]
+|\tau|^{2+\mu} e^{\tau}\varphi_\cC(v){\mathcal{F}}w-|\tau|^{2+\mu} g_{\cC}.
\end{equation}
Now, setting $\hat{p}=p^++p^-$ with $p^{\pm}=\fp_{\pm}(p)$, we have the general energy inequality
\begin{equation}
\sup_{\tau\in[\tau_{h_\mathrm{in}},\tau_0]}\|\hat{p}(\tau)\|^2_{\fH}+\frac{1}{C}\|\hat{p}\|^2_{\mathcal{D},\infty} \leq C\left( \|{q}\|^2_{\mathcal{D}^{\ast},\infty}
 +\|p^+(\tau_0)\|^2_{\fH}+\|p^-(\tau_{h_\mathrm{in}})\|^2_{\fH}\right) ,
\end{equation}
where, in contrast to \cite[Lemma 5.8]{ADS2}, we have the additional boundary term $\|p^-(\tau_{h_\mathrm{in}})\|^2_{\fH}$.
Using also the assumption  $\fp_ {+}(w_{\cC}(\tau_0))=0$, we thus get
\begin{align}\label{to_derive_en}
\||\tau|^{2+\mu}(w_{\cC}-\fp_0w_{\cC})\|_{\mathcal{D},\infty} \leq C\left( \| {q}\|_{\mathcal{D}^\ast,\infty}+C\|p^-(\tau_{h_\mathrm{in}})\|_{\mathfrak{H}}\right).
\end{align}
To control the terms on the right hand side, first observe we have the trivial inequality
\begin{equation}
\||\tau|^{1+\mu} w_{\cC}\|_{\mathcal{D}^\ast,\infty}\leq \| |\tau|^{1+\mu} w_{\cC}\|_{\mathcal{D},\infty}.
\end{equation}
Next, arguing similarly as in \cite[Proof of Proposition 5.12]{CHH_translators} we see that
\begin{equation}\label{to_derive_en_easy}
\left\| |\tau|^{2+\mu}\overline{\mathcal{E}}[w,\varphi_\cC(v)]+|\tau|^{2+\mu} e^{\tau}\varphi_\cC(v)\mathcal{F}w\right\|_{\mathcal{D}^{\ast},\infty}  \\
 \leq C\| |\tau|^{2+\mu} w 1_{\{\theta/2\leq v\leq\theta\}}\|_{\mathfrak{H},\infty}+C\| w\|_{C^2_{\exp}(\mathcal{C})}.
\end{equation}
Hence, our main remaining task is to estimate the  $\| \,\, \|_{\mathcal{D}^\ast,\infty}$ norm of
\begin{equation}
{\mathcal{E}}p=-\frac{v_{y}^2}{1+ v_{y}^2}p_{yy}-2\frac{v_{y}v_{yy}}{(1+v_{y}^2)^2}p_y+\frac{2-v^2}{2v^2}p.
\end{equation}
To this end, recall that by the weighted Sobolev inequality (see e.g. \cite[Lemma 4.12]{ADS1}) one has
\begin{equation}
\| (1+|y|) f \|_{\mathfrak{H}}\leq C\|  f \|_{\mathcal{D}},\qquad \| (1+|y|) f \|_{\mathcal{D}^\ast}\leq C\|  f \|_{\mathfrak{H}}.
\end{equation}
Also, if $g\in\mathcal{D}$ and $h\in W^{1,\infty}$ then by the product rule
\begin{equation}
\|  hg \|_{\mathcal{D}}\leq 2  \|  h \|_{W^{1,\infty}}  \|  g \|_{\mathcal{D}},
\end{equation}
and hence by duality
\begin{equation}
\|  hf \|_{\mathcal{D}^\ast}\leq 2  \|  h \|_{W^{1,\infty}}  \|  f \|_{\mathcal{D}^\ast}.
\end{equation}
Now, using the derivative estimates from \eqref{impr_cyl_est} we see that
\begin{equation}
\left\|\frac{v_{y}^2}{1+ v_{y}^2}p_{yy}\right\|_{\mathcal{D}^\ast}+\left\|\frac{v_{y}v_{yy}}{(1+v_{y}^2)^2}p_y\right\|_{\mathcal{D}^\ast}\leq \frac{C}{|\tau|} \| p \|_{\mathcal{D}}.
\end{equation}
Moreover, thanks to the profile estimates from \eqref{profile_growth} and \eqref{profile_derivative_growth} we have
\begin{equation}
\sup_{v\geq \theta/2} \left| \frac{2-v^2}{v^2(1+y^2)}\right|+\sup_{v\geq \theta/2}\left|\partial_y\left( \frac{2-v^2}{v^2(1+y^2)}\right)\right|\leq  \frac{C}{|\tau|}.
\end{equation}
This yields
\begin{equation}\label{Ddast_third}
\left\|\frac{2-v^2}{2v^2}p\right\|_{\mathcal{D}^{\ast}} \leq C \left\|\frac{2-v^2}{2v^2(1+y^2)}p\right\|_\mathcal{D}\leq  \frac{C}{|\tau|} \| p \|_{\mathcal{D}}.
\end{equation}
Finally, by the transformation rule \eqref{wfromu}, equation \eqref{cyl_gauge_trans_rec} and the cylindrical estimates from \eqref{eq_cyl_est} we have
\begin{equation}
\sup|w_\cC(\cdot,\tau_{h_\mathrm{in}})| \leq Ce^{\tau_{h_\mathrm{in}}}\sup_{\partial\Omega_{h_{\mathrm{in}}}}|u|.
\end{equation}
Together with Corollary \ref{cor_bound_est} (level set estimate) this yields
\begin{equation}\label{to_derive_en_bd}
\|p^-(\tau_{h_\mathrm{in}})\|_{\mathfrak{H}} \leq C\sup_{\Omega_h} \frac{|f|}{H_{\phi}}.
\end{equation}
This concludes the proof of the proposition.
\end{proof}

\subsection{Energy estimate for the tip variation}
In this subsection, we prove an energy estimate in the tip region, similar to the one appearing in \cite[Section 5.5]{CHH_translators}. Here, we again need to include $\tau$-factors, and we again need to deal with the inhomogeneity and the boundary term.\\

Recall that we denote by $W$ the tip variation associated to $u$, and by $G$ the tip inhomogeneity associated to $f$.
Fixing a suitable cutoff function $\varphi_\cT(v)$ such that $\varphi_\cT(v)=1$ if $v\leq \theta$ and $\varphi_\cT(v)=0$ if $v\geq  2 \theta$, we set
\begin{equation}
W_{\mathcal{T}}(v,\tau):=\varphi_\cT(v)W(v,\tau),\qquad G_{\mathcal{T}}(v,\tau):=\varphi_\cT(v)G(v,\tau).
\end{equation}
Moreover, similarly as in  \cite{ADS2,CHH_translators}, fixing a suitable cutoff function $\zeta(v)$  such that $\zeta(v)=0$ for $v \leq \theta/4$ and $\zeta(v)=1$ for $v \geq \theta/2$,
we consider the weight function
\begin{equation}\label{def_weight_fn}
\bar{\mu}(v,\tau):=-\frac14 Y^2(\theta,\tau)+\int^\theta_v \left[ \zeta(\tilde{v}) \left(\frac{Y^2}{4}\right)_{\tilde{v}}-(1-\zeta(\tilde{v}))\frac{1+Y_{\tilde{v}}^2}{\tilde{v}} \right]\, d\tilde{v}\, ,
\end{equation}
and define
\begin{equation}\label{tip2_norm_def}
\|F(\cdot,\tau)\|_{2}:= \left[ \int^{2\theta}_0F^2(v,\tau)\, e^{\bar{\mu}(v,\tau)}dv\right]^{1/2},
\end{equation}
and
\begin{equation}
\|F\|_{2,\infty}(\tau):=\sup_{\tau_{h_\mathrm{in}}+1\leq \tau' \leq \tau}\frac{1}{|\tau'|^{1/4}}\left[\int^{\tau'}_{\tau'-1}\int^{2\theta}_0 F^2(v,\sigma)e^{\bar{\mu}(v,\sigma)}\, dvd\sigma\right]^{1/2}\, .
\end{equation}
Furthermore, we set
\begin{equation}
\|F\|_{C^2|T_{\tau}}:=  \sup_{v\leq 2\theta} \big(|F|+|F_v|+|F_{\tau}|+ |F_{vv}|+|F_{v\tau}|+|F_{\tau\tau}|\big)(v,\tau),
\end{equation}
and
\begin{align}
\| F\|_{C^2_{\exp}(\mathcal{T})}(\tau):=\sup_{\tau_{h_\mathrm{in}} \leq \tau'\leq \tau} e^{\frac{99}{100}\tau'}\|F\|_{C^2|T_{\tau'}}.
\end{align}
Finally, as usual we abbreviate
 \begin{equation}
\|F\|_{2,\infty}= \|F\|_{2,\infty}(\tau_0),\qquad \| F\|_{C^2_{\exp}(\mathcal{T})}:=\| F\|_{C^2_{\exp}(\mathcal{T})}(\tau_0)\, .
\end{equation}

\begin{proposition}[energy estimate for the tip variation]\label{prop_int_tip_inhom}
There exists a constant $C=C(\phi)<\infty$, such that
\begin{equation}
\| |\tau|^{2+\mu} W_{\mathcal{T}}\|_{2,\infty} \leq C\left(\| |\tau|^{1+\mu} W 1_{[\theta,2\theta]}\|_{2,\infty}+\|W\|_{C^2_{\exp}(\cT)}+\| |\tau|^{1+\mu} G_\cT \|_{2,\infty}+\sup_{\Omega_h} \frac{|f|}{H_{\phi}}\right).
\end{equation}
\end{proposition}

\begin{proof}
Recall from Proposition \ref{lin_tip_eq_prop} (tip variation) that we have the evolution equation
\begin{equation}\label{evolve_inhom_tip_recall}
W_\tau=\frac{W_{vv}}{1+Y_{v}^2}+\left(\frac{1}{v}-\frac{v}{2}-2\frac{Y_{vv}Y_{v}}{(1+Y_{v}^2)^2}\right)W_v+\frac{1}{2}W+e^\tau \mathcal{F}W -G\, .
\end{equation}
Thus, similarly as in \cite[Lemma 5.23]{CHH_translators} we have the energy inquality
\begin{multline}\label{enineq}
\frac12\frac{d}{d\tau} \| W_\cT\|_2^2
\leq -\frac{1}{8}\int \frac{( W_\cT)_v^2}{1+Y_{v}^2} e^{\bar{\mu}} dv 
+\int H W_\cT^2 e^{\bar{\mu}} dv+\frac{C(\theta)}{|\tau|}\int^{2\theta}_\theta W^2 e^{\bar{\mu}} dv\\
+\left(e^\tau \|W\|_{C^2|T_{\tau}}+\|G_{\cT}\|_{2}\right)  \| W_\cT\|_2 ,
\end{multline}
where
\begin{equation}\label{tilde_G_def}
H=\frac12(1+Y_v^2)\left(\frac{1}{v}-\frac{v}{2}-\frac{{\bar{\mu}}_v}{1+Y_v^2}\right)^2+\frac12+{\bar{\mu}}_{\tau}.
\end{equation}
Indeed, when computing $\frac{1}{2}\frac{d}{d\tau} \| W_\cT\|_2^2$ the only new term is $-\int G_{\cT} W_{\cT}e^{\bar{\mu}} dv$, for which the Cauchy-Schwarz inequality immediately gives
\begin{equation}
\left|\int G_{\cT} W_{\cT}e^{\bar{\mu}} dv\right| \leq \|G_{\cT}\|_{2}\|W_{\cT}\|_{2}.
\end{equation}
Next, recall that $W_v$ vanishes for $v=0$ thanks to the $S^1$-symmetry. Thus, we can apply \cite[Corollary 5.22]{CHH_translators} (weighted Poincare inequality), which gives
\begin{equation}
\|W_\cT\|_2^2\leq \frac{C_0}{|\tau|} \int_0^{2\theta}\frac{( W_\cT)_v^2}{1+Y_{v}^2}e^{\bar{\mu}} dv\, , 
\end{equation}
where $C_0<\infty$ is a constant. Hence, similarly as in  \cite[Proof of Proposition 5.17]{CHH_translators} we infer that
\begin{equation}
\frac12\frac{d}{d\tau} \| W_\cT\|_2^2
\leq -\eta |\tau|\|W_\cT\|_2^2 +\frac{C}{|\tau|}\left(\|W1_{[\theta,2\theta]}\|_2^2 +e^{2\tau}  \|W \|_{C^2|T_\tau}^2+\|G_{\cT}\|^2_{2}\right),
\end{equation}
where $\eta=\eta(C_0)>0$. This implies
\begin{multline}\label{diff_eq_1b}
\frac{d}{d\tau}\left(|\tau|^{7/2+2\mu} \| W_\cT\|_2^2\right) \\
\leq -\eta|\tau|^{9/2+2\mu}\|W_\cT\|_2^2 +C|\tau|^{5/2+2\mu}\left(\|W1_{[\theta,2\theta]}\|_2^2+\|G_{\cT}\|^2_{2}\right)+ C\| W \|_{C^2_{\exp}(\cT)}^2(\tau)\, .
\end{multline}
To proceed, we set
\begin{align}
\mathrm{a}(\tau):=|\tau|^{7/2+2\mu} \| W_\cT\|_2^2,\qquad
  \mathrm{b}(\tau):=|\tau|^{3/2+2\mu}\left(\|W1_{[\theta,2\theta]}\|_2^2+\|G_{\cT}\|_2^2\right),
\end{align}
and consider
\begin{align}
\mathrm{A}(\tau):=\int_{\tau-1}^{\tau} \mathrm{a}(\tau')d\tau',\qquad \mathrm{B}(\tau):=\int_{\tau-1}^{\tau}  \mathrm{b}(\tau')d\tau'.
\end{align}
Then, we have
\begin{equation}\label{diff_eq_2}
\frac{d}{d\tau} \left[ e^{-\frac{\eta\tau^2}{2}}\mathrm{A}(\tau)\right] \leq C|\tau|e^{-\frac{\eta\tau^2}{2}}\left( \mathrm{B}(\tau)+\|W \|_{C^2_{\exp}(\cT)}^2(\tau)\right).
\end{equation}
Integrating, for any $\tau\in [\tau_{h_\mathrm{in}}+1, \tau_0]$ we infer that
\begin{equation}
\mathrm{A}(\tau)\leq\mathrm{A}(\tau_{h_\mathrm{in}}+1)+C\left( \sup_{\tau_{h_\mathrm{in}}+1\leq \tau'\leq \tau} B(\tau')+\|W \|_{C^2_{\exp}(\cT)}^2(\tau)\right).
\end{equation}
Finally, by the transformation rule \eqref{Wfromu} and the tip estimates from \eqref{eq_tip_est} we have
\begin{equation}
\sup|W_\cT(\cdot,\tau)| \leq C|\tau|^{1/2}e^{\tau}\sup_{\log\phi=-\tau}|u|.
\end{equation}
Also recall that by \cite[Proof of Proposition 5.21 and equation (558)]{CHH_translators} the weight ${\bar{\mu}}$ satisfies the coarse estimates
\begin{equation}
\sup_{v\leq 2\theta} |{\bar{\mu}}_\tau|\leq C |\tau|^3,\qquad \sup_{v\leq 2\theta} e^{{\bar{\mu}}}\leq e^{\frac{\tau}{4}}.
\end{equation}
Together with Corollary \ref{cor_bound_est} (level set estimate) we thus conclude that
\begin{equation}\label{to_derive_en_bd}
\mathrm{A}(\tau_{h_\mathrm{in}}+1)^{1/2}  \leq C\sup_{\Omega_h} \frac{|f|}{H_{\phi}}.
\end{equation}
This finishes the proof of the proposition.
\end{proof}

\bigskip

\subsection{Decay estimate} In this subsection, similarly to \cite[Section 5.6]{CHH_translators} (see also  \cite[Section 8]{ADS2}), we combine the estimates from the previous subsections to derive a decay estimate.

\begin{lemma}[coercivity estimate]\label{prop_coercivity_inhom}
Suppose $\fp_ {+}(w_{\cC}(\tau_0))=0$. 
Then, for some $C=C(\phi)<\infty$ we have
\begin{multline}
\| |\tau|^{2+\mu}(w_\cC-\mathfrak{p}_0 w_\cC) \|_{\cD,\infty}+\| |\tau|^{2+\mu} W_\cT \|_{2,\infty}  \\
\leq  C\left( \| |\tau|^{1+\mu} \mathfrak{p}_0 w_\cC  \|_{\mathfrak{H},\infty}+\|w\|_{C^2_{\exp}(\cC)}+\|W\|_{C^2_{\exp}(\cT)}+\| |\tau|^{2+\mu} g_{\cC}\|_{\cD^{\ast},\infty}+\| |\tau|^{1+\mu} G_\cT \|_{2,\infty}+\sup_{\Omega_h} \frac{|f|}{H_\phi} \right)\, .
\end{multline}
\end{lemma}

\begin{proof}
By Theorem  \ref{thm_unique_asympt_recall} (sharp asymptotics) we have
\begin{equation}
\frac{C(\theta)^{-1}}{\sqrt{|\tau|}} \leq |v_y(Y(v,\tau),\tau)| \leq   \frac{C(\theta)}{\sqrt{|\tau|}}\qquad \textrm{for}\,\, v\in[\theta,2\theta].
\end{equation} 
Together with Corollary \ref{cor_trans_rule} (transformation rule) this yields
\begin{equation}\label{w_W_compar}
\frac{C(\theta)^{-1}}{\sqrt{|\tau|}}|W(v,\tau)| \leq |w(Y(v,\tau),\tau)| \leq  \frac{C(\theta)}{\sqrt{|\tau|}} |W(v,\tau)|\qquad \textrm{for}\,\, v\in[\theta,2\theta].
\end{equation}
Hence, arguing similarly as in \cite[Proof of Lemma 8.1]{ADS2}, for $p=0,1,2$ we get
\begin{equation}\label{eq_equiv_of_norms}
C(\theta)^{-1}\| |\tau|^{p+\mu} W 1_{[\theta,2\theta]} \|_{2,\infty}\leq \| |\tau|^{p+\mu} w\, 1_{v(\cdot,\tau)\in[\theta,2\theta]} \|_{\fH,\infty}\leq  C(\theta)\| |\tau|^{p+\mu} W 1_{[\theta,2\theta]} \|_{2,\infty}\, .
\end{equation}
Applying Proposition \ref{prop_int_tip_inhom} (energy estimate in tip region) this yields
\begin{equation}\label{eq_WT-WD_compatible}
\| |\tau|^{2+\mu} W_{\mathcal{T}}\|_{2,\infty} \leq C\left(\| |\tau|^{1+\mu} w_\cC \|_{\mathcal{D},\infty}+\|W\|_{C^2_{\exp}(\cT)}+\| |\tau|^{1+\mu} G_\cT \|_{2,\infty}+\sup_{\Omega_h} \frac{|f|}{H_{\phi}}\right),
\end{equation}
Similarly, replacing $\theta$ by $\theta/2$, and applying Proposition \ref{energy_cyl_inhom} (energy estimate in the cylindrical region) we infer that
\begin{multline}\label{coerce_inter}
\| |\tau|^{2+\mu}(w_\cC-\fp_0w_\cC) \|_{\mathcal{D},\infty}\\
\leq C\left( \| |\tau|^{1+\mu} w_{\cC}\|_{\mathcal{D},\infty}+\| |\tau|^{2+\mu} W_\cT \|_{2,\infty}
+ \| w\|_{C^2_{\exp}(\mathcal{C})}+\|\tau^2 g_{\cC}\|_{\mathcal{D}^{\ast},\infty}+\sup_{\Omega_h}\frac{|f|}{H_{\phi}}\right).
\end{multline}
Observing also that
\begin{equation}
\| |\tau|^{1+\mu} w_{\cC}\|_{\cD,\infty} \leq \frac{1}{4C^2} \| |\tau|^{2+\mu} (w_\cC-\mathfrak{p}_0 w_\cC) \|_{\mathfrak{H},\infty}+ \| |\tau|^{1+\mu} \mathfrak{p}_0 w_\cC  \|_{\cD,\infty},
\end{equation}
and using absorption, this proves the lemma.
\end{proof}

\begin{theorem}[decay estimate]\label{main_decay_lin_inhom}
If $\fp_ {+}(w_{\cC}(\tau_0))=0$ and $\fp_ {0}(w_{\cC}(\tau_0))=0$, then
\begin{multline}
\| |\tau|^{1+\mu}\fp_0w_{\cC}\|_{\fH,\infty}+\| |\tau|^{2+\mu} (w_\cC-\fp_0(w_{\cC}))\|_{\cD,\infty} +\| |\tau|^{2+\mu} W_\cT \|_{2,\infty}\\
 \leq   C\left( \|w\|_{C^2_{\exp}(\cC)}+\|W\|_{C^2_{\exp}(\cT)}+\||\tau|^{2+\mu} g_{\cC}\|_{\mathcal{D}^{\ast},\infty}+\| |\tau|^{1+\mu} G_\cT \|_{2,\infty}+\sup_{\Omega_h} \frac{|f|}{H_{\phi}}\right),
\end{multline}
where $C=C(\phi)<\infty$ is independent of $h$.
\end{theorem}

\begin{proof}
Setting $\psi_0:=(y^2-2)/\|y^2-2\|_{\fH}$, consider the spectral coefficient
\begin{equation}
a(\tau):=\langle w_\cC(\tau),\psi_0\rangle_{\mathfrak{H}}.
\end{equation}
Since $a(\tau_0)=0$ by assumption, we infer that
\begin{equation}\label{eq_a}
a(\tau)=-\frac{1}{|\tau|^{2+\mu}} \int_{\tau}^{\tau_0}\left(F(\sigma)+N(\sigma)\right)|\sigma|^{2+\mu} d\sigma,
\end{equation}
where similarly as in  \cite[Proof of Proposition 5.25]{CHH_translators}  we have
\begin{equation}\label{eq_error_F}
F(\tau):=\left\langle\mathcal{E}[w_\cC]-\frac{(2+\mu)a(\tau)}{8|\tau|}\psi_0^2 +   \overline{\mathcal{E}}[w,\varphi_\cC]+ e^{\tau}\varphi_\cC(v){\mathcal{F}}[w], \psi_0\right\rangle_{\mathfrak{H}},
\end{equation}
and we have the new term
\begin{equation}
N(\tau):=-\left\langle g_{\cC}(\tau),\psi_0 \right\rangle_{\fH}.
\end{equation}
Note that
\begin{equation}
\left|\int_{\tau-1}^\tau  N(\sigma) |\sigma|^{2+\mu}\, d\sigma\right| \leq C \|\tau^{2+\mu}g_{\cC}\|_{\mathcal{D}^{\ast},\infty}.
\end{equation}
Moreover, arguing similarly as in \cite[Proof of Claim 5.27]{CHH_translators}, but now using Lemma \ref{prop_coercivity_inhom} (coercivity estimate) in lieu of \cite[Lemma 5.26]{CHH_translators}, for $\tau \in [\tau_{h_\mathrm{in}}+1,\tau_0]$ we infer that
\begin{equation}
\left|\int_{\tau-1}^\tau F(\sigma) |\sigma|^{2+\mu}\, d\sigma\right| 
\leq \frac{1}{2} A + CB,
\end{equation}
where
\begin{equation}
A:= \sup_{\tau_{h_\mathrm{in}}+1 \leq \tau'\leq \tau_0}\left( \int_{\tau'-1}^{\tau'} \left(a(\sigma)|\sigma|^{1+\mu}\right)^2 \, d\sigma \right)^{1/2}\, ,
\end{equation}
and
\begin{equation}
B:=\|w\|_{C^2_{\exp}(\cC)}+\|W\|_{C^2_{\exp}(\cT)}+\| |\tau|^{2+\mu} g_{\cC}\|_{\mathcal{D}^{\ast},\infty}+\| |\tau|^{1+\mu} G_\cT \|_{2,\infty}+\sup_{\Omega_h} \frac{|f|}{H_{\phi}}.
\end{equation}
Hence, for any $\tau \in [\tau_{h_\mathrm{in}}+1,\tau_0]$  we can estimate
\begin{equation}
|\tau|^{1+\mu} |a(\tau)|\leq\frac{1}{|\tau|}\sum^{\lceil\tau_0 \rceil}_{j=\lfloor\tau\rfloor} \left| \int^j_{j-1}(F(\sigma)+N(\sigma))|\sigma|^{2+\mu} d\sigma\right|
 \leq \frac{1}{2} A + CB.
\end{equation}
This implies
\begin{equation}
A\leq  CB,
\end{equation}
and together with Lemma \ref{prop_coercivity_inhom} (coercivity estimate) establishes the assertion of the theorem.
\end{proof}

As a corollary, we get a decay estimate for entire solutions of the homogenous problem $Lu=0$:

\begin{corollary}[decay estimate for entire homogenous solutions]\label{main_decay_cor}
Let $u\in C^{k,\alpha}_{\mathrm{loc}}(\mathbb{R}^3/S^1)$ be a solution of $Lu=0$.
Suppose that $\fp_ {+}(w_{\cC}(\tau_0))=0$ and $\fp_ {0}(w_{\cC}(\tau_0))=0$, and suppose in addition that $\| w_\cC \|_{\cD,\infty}<\infty$ and $\limsup_{\tau\to-\infty}\|W_{\cT}\|_{2}<\infty$. 
Then, for some $C=C(\phi)<\infty$ we have
\begin{equation}
\| w_\cC\|_{\cD,\infty} +\| W_\cT \|_{2,\infty}
 \leq   C\left( \|w\|_{C^2_{\exp}(\cC)}+\|W\|_{C^2_{\exp}(\cT)}\right).
\end{equation}
\end{corollary}

\begin{proof}
This follows by inspecting the above proof. Indeed, all terms involving the inhomogeneity can be simply dropped since $f=0$ by assumption, and all the $\tau$-weights can be dropped as well (this is similar to the simpler setting from \cite{CHH_translators}, which did not have any $\tau$-weights either). Finally, thanks to the finiteness assumption the steps in the above proofs that use absorption are indeed justified.\end{proof}

\bigskip

\section{Interior estimates}\label{Holder_sec}
The purpose of this section is twofold. On the one hand, we prove estimates that lead to a Schauder theory for $L$ in appropriate H\"older spaces. On the other hand, our estimates also allow us to control the weighted parabolic $C^2$-norm of $w$ and $W$ in terms of their $L^2$-norm and the $C^{\alpha}$-norm of the inhomogeneity.\\

In the cylindrical region, denoting by $V$ the unrenormalized cylindrical profile function, we call
\begin{equation}\label{bw_def}
\bw(x,t):=-V_t(x,t)\,\, u(x,V(x,t),0)
\end{equation}
the \emph{cylindrical variation} associated to $u$, and
\begin{equation}\label{bg_def}
\bg(x,t) :=\sqrt{1+V_x^2(x,t)+V_t^2(x,t)}\,\, f(x,V(x,t),0)
\end{equation}
the \emph{cylindrical inhomogeneity} associated to $f$. To reformulate the equations in a suitable parabolic form it is useful to define
\begin{equation}\label{unren_tilde_quant}
\tilde{\bw}(x,s,t)={\bw}(x,s+t), \qquad {\tilde{\bg}}(x,s,t)={\bw}(x,s+t).
\end{equation}
As discussed in the introduction, this redundant description enables us to view $s$ as a spatial variable and $t$ as a time variable, which is the key to establish sharp Schauder estimates in our degenerating setting.

\begin{proposition}[cylindrical variation]\label{prop_cyl_var}
Suppose $Lu=f$. Then, defining $\tilde{\bw}$ and ${\tilde{\bg}}$ by \eqref{unren_tilde_quant}, we have
\begin{multline}\label{eq_V^D_evolution}
\tilde{\bw}_t= \tfrac{1+\tilde{V}_{x}^2}{1+\tilde{V}_{x}^2+\tilde{V}_{s}^2}\tilde{\bw}_{ss}+\tfrac{1+\tilde{V}_{s}^2}{1+\tilde{V}_{x}^2+\tilde{V}_{s}^2}\tilde{\bw}_{xx}-\tfrac{2\tilde{V}_{x}\tilde{V}_{s}}{1+\tilde{V}_{x}^2+\tilde{V}_{s}^2}\tilde{\bw}_{xs}\\
+ \tfrac{2\tilde{V}_{x}(\tilde{V}_{ss}-\tilde{V}_{t}-1/\tilde{V})-2\tilde{V}_{s}\tilde{V}_{xs}}{1+\tilde{V}_{x}^2+\tilde{V}_{s}^2} \tilde{\bw}_x+ \tfrac{ 2\tilde{V}_{s}(\tilde{V}_{xx}-\tilde{V}_{t}-1/\tilde{V})-2\tilde{V}_{x} \tilde{V}_{xs}}{1+\tilde{V}_{x}^2+\tilde{V}_{s}^2} \tilde{\bw}_s+\tfrac{1}{\tilde{V}^2}\tilde{\bw}-\tilde{\bg},
\end{multline}
where $\tilde{V}(x,s,t)=V(x,s+t)$.
\end{proposition}

\begin{proof}
As before, we work with a suitable one-parameter family $\phi^\eps$ such that
\begin{equation}
\frac{d}{d\eps}\Big|_{\eps=0}\phi^\eps=-u
\end{equation}
near the point under consideration. Then, differentiating the identity
\begin{equation}
V^\eps(x_1,-\phi^\eps(x_1,x_2,0))=x_2
\end{equation}
we infer that
\begin{equation}
\frac{d}{d\eps}|_{\eps=0}V^\eps(x,t)=-V_t(x,t)\, u(x,V(x,t),0).
\end{equation}
Hence, the assertion follows by differentiating both sides of equation \eqref{V_eq_der} and evaluating at $\eps=0$.
\end{proof}

To capture the position of the (right) tip we define a positive function $X$ by
\begin{equation}\label{X_gauge}
(X(x_2,x_3,t),x_2,x_3,-t)\in \mathrm{graph}(\phi),
\end{equation} 
Similarly as before, we set $\tilde{X}(x_2,x_3,x_4,t)=X(x_2,x_3,x_4+t)$, and work with the functions
\begin{align}\label{def_tildeWG}
\tilde{{\bW}}(x_2,x_3,x_4,t):=&-\tilde{X}_t(x_2,x_3,x_4,t)\,\, u(\tilde{X}(x_2,x_3,x_4,t),x_2,x_3),\\
\tilde{{\bG}}(x_2,x_3,x_4,t):=&\sqrt{1+|D\tilde{X}(x_2,x_3,x_4,t)|^2}\,\, f(\tilde{X}(x_2,x_3,x_4,t),x_2,x_3).\nonumber
\end{align}

\begin{proposition}[tip variation]\label{prop_tip_var}
Suppose $Lu=f$. Then, defining ${\tilde{\bW}}$ and ${\tilde{\bG}}$ by \eqref{def_tildeWG}, we have
\begin{equation}\label{X_D_evolve}
{\tilde{\bW}}_t=\left(\delta_{ij}-\frac{\tilde{X}_{i}\tilde{X}_{j}}{1+|D\tilde{X}|^2}\right){\tilde{\bW}}_{ij}+2\frac{\tilde{X}_{i}\tilde{X}_{ij}{\tilde{\bW}}_j}{1+|D\tilde{X}|^2}-2\frac{\tilde{X}_{i}\tilde{X}_{j}\tilde{X}_{ij}\tilde{X}_{k}{\tilde{\bW}}_k}{(1+|D\tilde{X}|^2)^2}-\tilde{\bG}.
\end{equation} 
\end{proposition}

\begin{proof}
Working with a suitable one-parameter family as above, differentiating the identity
\begin{equation}
X^{\eps}(x_2,x_3,-\phi^{\eps}(x_1,x_2,x_3))=x_1
\end{equation}
we infer that
\begin{equation}
\frac{d}{d\eps}|_{\eps=0}X^{\eps}(x_2,x_3,t)=-X_t(x_2,x_3,t)\, u(X(x_2,x_3,t),x_2,x_3).
\end{equation}
Now, in terms of $\tilde{X}^{\eps}(x_2,x_3,x_4,t):=X^\eps(x_2,x_3,x_4+t)$ the inhomogeneous translator equation reads
\begin{equation}\label{X_eq}
\tilde{X}^\eps_t-\left(\delta_{ij}-\frac{\tilde{X}^\eps_i\tilde{X}^\eps_j}{1+|D \tilde{X}^\eps|^2}\right)\tilde{X}^{\eps}_{ij}=\sqrt{1+|D \tilde{X}^{\eps}|^2}\Theta[\phi^\eps](\tilde{X}^{\eps},x_2,x_3).
\end{equation}
Hence, the assertion follows by differentiating both sides and evaluating at $\eps=0$.
\end{proof}

Finally, we collect the transformation rules
\begin{equation}\label{transwbf}
w(y,\tau)=e^{\frac{\tau}{2}}\tilde{\bw}(e^{-\frac{\tau}{2}}y,0,-e^{-\tau}),  \qquad\qquad  g(y,\tau)=e^{-\frac{\tau}{2}}\tilde{\bg}(e^{-\frac{\tau}{2}}y,0,-e^{-\tau}),
\end{equation}
and
\begin{equation}\label{trangbf}
W(v,\tau)=e^{\frac{\tau}{2}}{\tilde{\bW}}(e^{-\frac{\tau}{2}}v,0,0,-e^{-\tau}),  \qquad\qquad  G(v,\tau)=e^{-\frac{\tau}{2}}\tilde{\bG}(e^{-\frac{\tau}{2}}v,0,0,-e^{-\tau}), 
\end{equation}
which immediately follow from a similar formula holding on the variation level.\\

\subsection{Interior estimates for the cylindrical variation}\label{cyl_holder_sec}
In this subsection, we prove interior estimates for the cylindrical variation by adapting the arguments from \cite[Section 5.7]{CHH_translators} to our setting. As usual, here we have to deal in addition with the inhomogeneity and the boundary. Another novelty is that we use the mean curvature as a weight function in order to obtain a sharper Schauder estimate in the collar region.\\
 
Let us introduce our weighted parabolic H\"older norms. For space-time points $X=(x,s,t)$ and $X'=(x',s',t')$ we work with the parabolic distance
\begin{equation}
d(X,X')=\sqrt{|x-x'|^2+|s-s'|^2+|t-t'|}.
\end{equation}
Moreover, we set
\begin{equation}
H(x,s,t):=H_{\phi}(x,\tilde{V}(x,s,t),0).
\end{equation}
Now, given $\alpha\in (0,1)$, nonnegative integers $k,l$, and a region $U$ over which a function $\bbf$ is defined, we set
\begin{equation}
[\bbf]_{H;U}^{k,(l)}=\sup_{(x,s,t)\in U}\sup_{i+j+2m= k} H(x,s,t)^{1-k-l} \left| \partial_x^i \partial_s^j\partial_t^m\bbf(x,s,t)\right|\, ,
\end{equation}
and
\begin{equation}
[\bbf]_{H;U}^{k,\alpha,(l)}=\sup_{X,X'\in U}\;\sup_{i+j+2m= k}\left| H^{-1}(X)+H^{-1}(X')\right|^{k+l+\alpha-1}\frac{ |\partial_x^i \partial_s^j\partial_t^m\bbf(X)- \partial_x^i \partial_s^j\partial_t^m\bbf(X')|}{|d(X,X')|^\alpha}\, .
\end{equation}
Then, we can define weighted parabolic H\"older norms norms by
\begin{align}
 \|\bbf\|_{C^{k,\alpha,(l)}_H	(U)}=\|\bbf\|_{C^{k,(l)}_H(U)}+  [\bbf]_{H;U}^{k,\alpha,(l)}\, ,\qquad \textrm{where} \qquad \|\bbf\|_{C^{k,(l)}_H(U)}=\sum_{m=0}^k [\bbf]_{H;U}^{m,(l)}\, .
\end{align}
Finally, when the offset $l$ equals $0$ we simply drop it from the notation.\footnote{In practice, due to the scaling of the operator $L$, we will choose $l=0$ for the domain norms and $l=2$ for the image norms.}\\ 

Note that for functions $\bbf=\bbf(x,s,t)$ that actually only depend on $x$ and $s+t$ our parabolic H\"older norms are in fact (rather nonstandard) elliptic H\"older norms in disguise.\\

Our interior estimates take place naturally in parabolic cubes
\begin{equation}
P_r(x,s,t)=\{(x',s',t'): |x'-x|\leq r,|s'-s| \leq r , t-r^2\leq t'\leq t \}.
\end{equation}
Moreover, we often abbreviate
\begin{equation}\label{Q_cyl_def}
P_r(x,t):=P_r(x,0,t).
\end{equation}

\begin{proposition}[interior estimates for cylindrical variation]\label{prop:interior_estimates_cylinder}
Suppose that $V(x,t)\geq \ell\sqrt{|t|/\log|t|}$ holds at some given time $t\leq-h_0$. If $r\leq \tfrac{1}{10} H^{-1}(x,0,t)$ is such that  $P_r(x,t)\cap \{t\leq -h\}=\emptyset$ then
\begin{equation}\label{C_0_est_W}
\sup_{P_{r/2}(x,t)}|\tilde{\bw}|\leq   \frac{C}{r^{2}}\left( \int_{P_r(x,t)} \tilde{\bw}^2 \, dx'\, ds'\, dt'\right)^{\frac{1}{2}}{ +C r^{ \frac{2}{3}} \left( \int_{P_r(x,t)} \tilde{\bg}^3\, dx'\, ds'\, dt'\right)^{\frac{1}{3}}}.
\end{equation}
Moreover, if we assume $t\geq -h/e$ then setting $r:=\tfrac{1}{10} H^{-1}(x,0,t)$ we have
\begin{equation}\label{Holder_W_away}
\|\tilde{\bw}\|_{C^{k,\alpha}_H(P_{r/2}(x,t))}\leq C \left( \|\tilde{\bw}\|_{C^{0}_H(P_r(x,t))}+ \|\tilde{\bg}\|_{C^{k-2,\alpha,(2)}_H(P_r(x,t))}\right).
\end{equation}
Furthermore, if we only assume $t\geq -h_{\mathrm{in}}$ then setting $r:=|t|^{\frac{1}{2}(1-\frac{1}{100k})}$ we still get 
\begin{equation}\label{Holder_W_uptob}
\|\tilde{\bw}\|_{C^{k,\alpha}_H(P_{r/2}(x,t))}\leq C|t|^{\frac{1}{10}} \left( \|\tilde{\bw}\|_{C^{0}_H(P_r(x,t))}+ \|\tilde{\bg}\|_{C^{k-2,\alpha,(2)}_H(P_r(x,t))}\right)\, .
\end{equation}
\end{proposition}

\begin{proof}Consider the rescaling
\begin{equation}
\hat{\tilde{\bw}}(\hat x_1,\hat x_2,\hat t):= \frac{1}{r}\tilde{\bw}(x+r\hat x_1,r\hat{x_2},t+r^2\hat{t}),\qquad \hat{\tilde{\bg}}(\hat x_1,\hat x_2,\hat t):=r \tilde{\bg}(x+r\hat x_1,r\hat{x_2},t+r^2\hat{t}).
\end{equation}
Proposition \ref{prop_cyl_var} (cylindrical variation) and the cylindrical estimates from \eqref{eq_cyl_est} imply
\begin{equation}\label{VD_eq}
\tfrac{\partial}{\partial \hat t}   \hat{\tilde{\bw}} (\hat x,\hat t)=a_{ij}(\hat x,\hat t)\tfrac{\partial^2}{\partial \hat x_i \partial \hat x_j}\hat{\tilde{\bw}} (\hat x,\hat t)+b_{i}(\hat x,\hat t)\tfrac{\partial}{\partial \hat x_i}\hat{\tilde{\bw}} (\hat x,\hat t)+c(\hat x,\hat t)\hat{\tilde{\bw}} (\hat x,\hat t){-\hat{\tilde{\bg}}(\hat x,\hat t)},
\end{equation}
where, provided $\tau_0$ is sufficiently negative, the coefficients satisfy
\begin{equation}\label{coef_bounds_VD}
\|a_{ij}\|_{C^{k,\alpha}(P_1(0))}+\|b_i\|_{C^{k,\alpha}(P_1(0))}+\|c\|_{C^{k,\alpha}(P_1(0))}\leq C,\;\;\;\;\;\;\;a_{ij}\xi^i\xi^j \geq C^{-1}|\xi|^2\, .
\end{equation} 
Therefore, standard interior $L^\infty$-estimates \cite[Theorem 7.36]{Lieberman} yield
\begin{equation}
\sup_{P_{1/2}(0)}|\hat{\tilde{\bw}}|\leq C \left( \int_{P_{1}(0)}\hat{\tilde{\bw}}^2\, d\hat{x}\, d\hat{s}\, d\hat{t}\right)^{\frac{1}{2}}{+C \left( \int_{P_{1}(0)} \hat{\tilde{\bg}}^3\, d\hat{x}\, d\hat{s}\, d\hat{t}\right)^{\frac{1}{3}}}\, .
\end{equation}
Scaling back by $r$, and observing that the condition $P_r(x,t)\cap \{t\leq -h\}=\emptyset$ ensures that the equation is defined in the parabolic cubes under consideration,  this proves \eqref{C_0_est_W}.\\
Similarly, standard interior Schauder estimates \cite[Theorem 4.9]{Lieberman}  yield
\begin{equation}\label{trivial_Schauder}
\|\hat{\tilde{\bw}}\|_{C^{k,\alpha}(P_{1/2}(0))}\leq C \|\hat{\tilde{\bw}}\|_{C^{0}(P_{1}(0))} +{C \|\hat{\tilde{\bg}}\|_{C^{k-2,\alpha}(P_{1}(0))} }.
\end{equation}
If $t\geq -h/e$, then we can again choose $r=\tfrac{1}{10}H^{-1}(x,0,t)$, and scale back, which proves \eqref{Holder_W_away}.\\
Finally, if we only assume $t\geq -h_{\mathrm{in}}$, then we work with the smaller radius $r:=|t|^{\frac{1}{2}(1-\frac{1}{100k})}$ to ensure the equation is defined in the parabolic cubes under consideration. Observing also that thanks to the mean curvature asymptotics from \eqref{H_two_sided} we have
\begin{equation}\label{H_vs_r}
r\leq H^{-1}(x,0,t)\leq r|t|^{\frac{1}{100k}},
\end{equation}
Hence, scaling \eqref{trivial_Schauder} by $r$ we get the remaining estimate, which concludes the proof.
\end{proof}

\begin{corollary}[$L^\infty$-estimate for cylindrical variation]\label{cor:C0_cylinder}
Suppose $v(y,\tau)\geq \theta/2$ at some $\tau\in[\tau_{h_\mathrm{in}}, \tau_0]$. If $\tau\geq \tau_{h_\mathrm{in}} + 1$ then
\begin{align}\label{C_0_most}
|w(y,\tau)| \leq \frac{C}{|\tau|^{1+\mu}} e^{\frac{(|y|+2)^2}{8}} \| |\tau|^{1+\mu} w_\cC\|_{\fH,\infty}(\tau+1)+C  \|\tilde{\bg}\|_{C^{0,(2)}_H  \left(P_{\exp(-\tau/2)/10} (e^{-\tau/2}y,-e^{-\tau}) \right)},
\end{align}
and if $\tau\leq \tau_{h_\mathrm{in}} + 1$ then
\begin{align}\label{C_0_boundary}
|w(y,\tau)| \leq \frac{C}{|\tau|^{1/2+\mu}} e^{\frac{(|y|+2)^2}{8}} \| |\tau|^{1+\mu} w_\cC\|_{\fH,\infty}(\tau+1)+\frac{C}{|\tau|^{1/2+\mu}}\sup_{\Omega_h}\frac{|f|}{H_{\phi}}.
\end{align}
\end{corollary}

\begin{proof}
Recall that by the transformation rule \eqref{wfromu}, taking also into account equation \eqref{cyl_gauge_trans_rec} and the cylindrical estimates from \eqref{eq_cyl_est}, we have
\begin{equation}
|w(y,\tau)| \leq Ce^{\tau}\sup_{\log\phi = -\tau}|u|.
\end{equation}
Together with Corollary \ref{cor_bound_est} (level set estimate) this yields
\begin{equation}\label{to_derive_en_bd}
\tau\leq -\log\left(h-\frac{h}{\log^{1/2+\mu}h}\right)\qquad \Rightarrow\qquad |w(y,\tau)|\leq \frac{C}{|\tau|^{1/2+\mu}}\sup_{\Omega_h}\frac{|f|}{H_{\phi}}.
\end{equation}
Assume from now on that $\tau\geq -\log\left(h-h/\log^{1/2+\mu}h\right)$. To treat both cases simultaneously, we use the notation $\beta=1/2$ if $\tau\leq \tau_{h_\mathrm{in}}+1$  and $\beta=0$ if $\tau> \tau_{h_\mathrm{in}}+1$. Denote by $(x,t)$ the point in the original flow corresponding to $(y,\tau)$. Choose $r=\frac{1}{10}\sqrt{|t|/|\tau|^{\beta}}$, and observe that $P_r(x,t)\cap \{t\leq -h_{\mathrm{in}}\}=\emptyset$.\\
Now, using the transformation rule \eqref{transwbf}, and assuming that $\tau_0$ is sufficiently negative, we see that
\begin{equation}
\int_{t-r^2}^{t}\int_{-r}^{r}\int_{x-r}^{x+r}\tilde{\bw}^2(x',s',t')\, dx'\, ds'\, dt'\\
 \leq C |t|^{5/2}r \int_{\max\{\tau-1,\tau_{h_\mathrm{in}}\}}^{\tau+1}\int_{y-2}^{y+2} w^2(y',\tau')\, dy'\, d\tau'. 
\end{equation}
This implies
\begin{align}
\int_{P_{r}(x,t)} \tilde{\bw}^2 \, dx'\, ds'\, dt' \leq  C \frac{|t|^3}{|\tau|^{2(1+\mu)+\beta/2}}e^{\frac{(|y|+2)^2}{4}}\| |\tau|^{1+\mu} w_{\cC}\|^2_{\fH,\infty}(\tau+1).
\end{align}
Moreover, using the fact that $H$ is comparable to $|t|^{-1/2}$ in the region of integration, we see that
\begin{align}
\int_{P_{r}(x,t)} \tilde{\bg}^3 \, dx'\, ds'\, dt'\leq C\frac{|t|^{1/2}}{|\tau|^{2\beta}} \left(\sup_{P_{r}(x,t)} \frac{|\tilde{\bg}|}{H_{\phi}}\right)^3.
\end{align}
Hence, remembering the transformation rule \eqref{transwbf}, and applying Proposition \ref{prop:interior_estimates_cylinder} (interior estimates for cylindrical variation) we conclude that
\begin{equation}\label{basic_C_0_est}
|t|^{1/2}|w(y,\tau)| \leq \frac{C}{|t|}|\tau|^{\beta}\frac{|t|^{3/2}}{|\tau|^{1+\mu+\beta/4}}e^{\frac{(y+2)^2}{8}}\| |\tau|^{1+\mu} w_{\cC}\|_{\fH,\infty}(\tau+1)+C\frac{|t|^{1/3}}{|\tau|^{\beta/3}}\frac{|t|^{1/6}}{|\tau|^{2\beta/3}}\left(\sup_{P_r(x,t)} \frac{|\tilde{\bg}|}{H_{\phi}}\right).
\end{equation}
This implies the assertion.
\end{proof}

\begin{corollary}[{$C^k$-estimate for cylindrical variation}]\label{C^2_cyl_cor}
If $\tau\in [\tau_{h_\mathrm{in}},\tau_0-1]$ and  $v(y,\tau)\geq \ell/\sqrt{|\tau|}$, then\footnote{Since $\tilde{\bw}_s=\tilde{\bw}_t$ one also gets some bounds for $\sum_{0\leq i+j\leq k}  |\partial_y^i \partial_\tau^j w|(y,\tau)$, which however come with a worse weight factor.}
\begin{multline}
\sum_{0\leq i+2j\leq k}  |\partial_y^i \partial_\tau^j w|(y,\tau)\\
\leq e^{-\frac{\tau}{8}}\left(\sup_{|\tau'-\tau|\leq e^{\frac{\tau}{100k}},  v(y',\tau')\geq \frac{\ell}{2\sqrt{|\tau|}}}| w(y',\tau')|+ \|\tilde{\bg}\|_{C_H^{k,\alpha,(2)}\left(P_{\exp\left((1-\frac{1}{100k})\frac{\tau}{2}\right)}(e^{-\tau/2}y,-e^{-\tau})\right)}\right).
\end{multline}
\end{corollary}
\begin{proof}
Setting $(x,t):=(e^{-\tau/2}y,-e^{-\tau})$ and $r:=|t|^{\frac{1}{2}(1-\frac{1}{100k})}$, observe that all points in $Q_{r}(x,t)$ correspond to rescaled points $(y',\tau')$ with $|\tau'-\tau|\leq e^{\frac{\tau}{100k}}$ and  $v(y',\tau')\geq \frac{\ell}{2\sqrt{|\tau|}}$. Hence, applying Proposition \ref{prop:interior_estimates_cylinder} (interior estimates for cylindrical variation) we get 
\begin{equation}\label{Sc_resc_app}
\|\tilde{\bw}\|_{C^{k,\alpha}_H(P_{r/2}(x,t))}\leq Ce^{-\tau/10} \left( \sup_{|\tau'-\tau|\leq e^{\frac{\tau}{100k}},  v(y',\tau')\geq \frac{\ell}{2\sqrt{|\tau|}}}| w(y',\tau')|+ \|\tilde{\bg}\|_{C^{k-2,\alpha,(2)}_H(P_r(x,t))}\right)\, .
\end{equation}
Moreover, differentiating the transformation rule \eqref{transwbf} yields\begin{align}\label{tran_der_bf1}
&  \tilde{\bw}_x=w_y , && |t|^{\frac{1}{2}}  \tilde{\bw}_{xx} =w_{yy}, &&|t|^{\frac{1}{2}}\tilde{\bw}_t= w_\tau+\frac{y}{2}w_y-\frac{w}{2},
\end{align}
and
\begin{equation}\label{trans_der_bf2}
|t|  \tilde{\bw}_{xt}= w_{y\tau}+\frac{y}{2}w_{yy},\qquad |t|^{\frac{3}{2}}\tilde{\bw}_{tt}= w_{\tau\tau}+yw_{\tau y}+\frac{y^2}{4}w_{yy}+\frac{y}{4}w_y-\frac{w}{4},
\end{equation}
and similarly for the higher derivatives.
Hence, using \eqref{H_vs_r} and $|y| \leq 2\sqrt{|\tau|}$, the result follows from \eqref{Sc_resc_app} and the definition of the norm.
\end{proof}

\bigskip

\subsection{Interior estimates for the tip variation}\label{tip_holder_sec}
In this subsection, we prove interior estimates for the tip variation by adapting the arguments from \cite[Section 5.8]{CHH_translators}, with similar modifications as in the previous subsection.

Setting
\begin{equation}
H(x_2,x_3,x_4,t)=H_{\phi}(\tilde{X}(x_2,x_3,x_4,t),x_2,x_3),
\end{equation}
we define
\begin{equation}
[\bF]_{H;U}^{k,(l)}=\sup_{X\in U}\sup_{i+j+\ell+2m= k} |H(X)|^{1-k-l} \left| \partial_{x_2}^i \partial_{x_3}^j  \partial_{x_4}^{\ell} \partial_t^m\bF(X)\right|\, ,
\end{equation}
and
\begin{equation}
[\bF]_{H;U}^{k,\alpha,(l)}=\sup_{X,X'\in U }\;\sup_{i+j+\ell+2m= k}\left| H^{-1}(X)+H^{-1}(X')\right|^{k+l+\alpha-1}\frac{ |\partial_{x_2}^i \partial_{x_3}^j  \partial_{x_4}^{\ell} \partial_t^m\bF(X)- \partial_{x_2}^i \partial_{x_3}^j  \partial_{x_4}^{\ell} \partial_t^m\bF(X')|}{|d(X,X')|^\alpha}\, ,
\end{equation}
where
\begin{equation}
d(X,X')=\sqrt{|x_2-x_2'|^2+|x_3-x_3'|^2+|x_4-x_4'|^2+|t-t'|}.
\end{equation}
Then, we can define weighted parabolic H\"older norms by
\begin{align}
 \|\bF\|_{C^{k,\alpha,(l)}_H	(U)}=\|\bF\|_{C^{k,(l)}_H(U)}+  [\bF]_{H;U}^{k,\alpha,(l)}\, \qquad \textrm{where} \qquad \|\bF\|_{C^{k,(l)}_H(U)}=\sum_{m=0}^k [\bF]_{H;U}^{m,(l)}\, .
\end{align}
As before, when $l=0$ we omit if from the notation, and observe that for functions $\bF$ that actually only depend on $x_2$, $x_3$ and $x_4+t$ our parabolic H\"older norms are elliptic H\"older norms in disguise.
 Finally, we work with parabolic cubes
\begin{equation}
Q_r(X)=\left\{X': t-r^2\leq t'\leq t,  |\langle x'-x , e_i\rangle|\leq r \;\; \text{for each} \;i=2,3,4  \right\},
\end{equation}
and abbreviate
\begin{equation}\label{Q_tip_def}
Q_r(x,t):=Q_r(x,0,0,t).
\end{equation}

\begin{proposition}[interior estimates for tip variation]\label{prop_interior_estimates_solition} Suppose $x\leq \ell\sqrt{|t|/\log|t|}$ holds at some given time $t\leq-h_0$. If $t\geq -h/e$, then setting $r=\sqrt{|t|/\log|t|}$ we have
\begin{equation}
\sup_{Q_{r/2}(x,t)}|{\tilde{\bW}}|\leq   \frac{C}{r^{\frac{5}{2}}}\left( \int_{Q_r(x,t)} {\tilde{\bW}}^2 \, dx_2'\, dx_3'\, dx_4'\, dt'\right)^{\frac{1}{2}}+Cr^{\frac{3}{4}}\left( \int_{Q_r(x,t)} \tilde{\bG}^4 \, dx_2'\, dx_3'\, dx_4'\, dt'\right)^{\frac{1}{4}}\, ,
\end{equation}
and 
\begin{equation}
\|{\tilde{\bW}}\|_{C^{k,\alpha}_H(Q_{r/2}(x,t))}\leq C\left( \|{\tilde{\bW}}\|_{C^{0}_H(Q_r(x,t))}+\|\tilde{\bG}\|_{C^{k-2,\alpha,(2)}_H(Q_r(x,t))}\right).
\end{equation} 
Furthermore, if we only assume $t\geq -h_{\mathrm{in}}$, then setting $r:=|t|^{\frac{1}{2}(1-\frac{1}{100k})}$ we still have
\begin{equation}
\|{\tilde{\bW}}\|_{C^{k,\alpha}_H(Q_{r/2}(x,t))}\leq C|t|^{1/10}\left( \|{\tilde{\bW}}\|_{C^{0}_H(Q_r(x,t))}+\|\tilde{\bG}\|_{C^{k-2,\alpha,(2)}_H(Q_r(x,t))}\right).
\end{equation} 
\end{proposition}

\begin{proof}
Consider the rescaling
\begin{equation}
\hat{\tilde{\bW}}(\hat x, \hat t)= \frac{1}{r}{\tilde{\bW}}(x+r \hat x,t+r^2 \hat t),\;\;\;\;\;\hat{\tilde{\bG}}(\hat{x},\hat{t})=r \tilde{\bG}(x+r \hat{x},t+r^2 \hat{t}).
\end{equation}
Then, Proposition \ref{prop_tip_var} (tip variation) and Theorem \ref{thm_unique_asympt_recall} (sharp asymptotics) imply that
\begin{equation}\label{eq_tildW}
 \hat{\tilde{\bW}}_{\hat t} =a_{ij} \hat{\tilde{\bW}}_{ij}  +b_{i} \hat{\tilde{\bW}}_i+\hat{\tilde{\bG}},
\end{equation}
where, assuming $\tau_0\ll 0$, in the space-time region under consideration the coefficients satisfy
\begin{equation}
\|a_{ij}\|_{C^{k,\alpha}(Q_1(0))}+\|b_i\|_{C^{k,\alpha}(Q_1(0))}\leq C,\qquad   a_{ij}\xi_i\xi_j \geq C^{-1}|\xi|^2.
\end{equation} 
Therefore, standard interior $L^\infty$-estimates \cite[Theorem 7.36]{Lieberman} yield
\begin{equation}
\sup_{Q_{1/2}(0)}|\hat{\tilde{\bW}}|\leq C \left( \int_{Q_{1}(0)}\hat{\tilde{\bW}}^2\, d\hat{x}\, d\hat{t}\right)^{\frac{1}{2}}{+C \left( \int_{Q_{1}(0)} \hat{\tilde{\bG}}^4\, d\hat{x}\, d\hat{t}\right)^{\frac{1}{4}}}\, ,
\end{equation}
and standard interior Schauder estimates \cite[Theorem 4.9]{Lieberman} yield
\begin{equation}
\|\hat{\tilde{\bW}}\|_{C^{k,\alpha}(Q_{1/2}(0))}\leq C \left( \|\hat{\tilde{\bW}}\|_{C^{0}(Q_{1}(0))}+\|\hat{\tilde{\bG}}\|_{C^{k-2,\alpha}(Q_{1}(0))}\right) \, .
\end{equation}
Scaling back to the original variables, this proves the proposition.
\end{proof}

\begin{corollary}[{$L^\infty$-estimate for tip variation}]\label{cor:C0_tip}
If $\tau\leq \tau_0-1$, then
\begin{equation}
\sup_{\tau' \leq \tau}e^{\frac{26}{100}\tau'}\sup_{ v\leq \frac{9}{10} \theta} |W(v,\tau')| \leq \|W_\cT\|_{2,\infty}(\tau+1)+\sup_{\Omega_h}\frac{|f|}{H_{\phi}}\, .
\end{equation}
\end{corollary}

\begin{proof}
If $\tau\leq -\log(h/e)$, then by Theorem \ref{max_princ_inhom} (barrier estimate) and equation \eqref{Wfromu} we have
\begin{equation}
|W(y,\tau)|\leq C|\tau|^{5}\sup_{\Omega_h}\frac{|f|}{H_{\phi}}.
\end{equation}
Assume from now that $\tau\geq -\log(h/e)$. Note that
\begin{equation}\label{rel_XD_W}
\tilde{\bW}(x_2,x_3,x_4,t)=e^{-\frac{\tau}{2}}W(v,\tau)
\end{equation}
where
\begin{equation}
 \tau=-\log(-x_4-t) \qquad\textrm{ and }\qquad v=e^{\frac{\tau}{2}}\left(x_2^2+x_3^2\right)^{\frac{1}{2}}\, .
\end{equation}
Suppose first $v\leq \ell/\sqrt{|\tau|}$. Arguing as in \cite[Proof of Proposition 5.30]{CHH_translators}, we see that setting $x:=e^{-\frac{\tau}{2}}v$ and $R:=|\tau|^{-1/2}e^{-\tau/2}$ we have
\begin{equation}
\frac{1}{R}   \int_{Q_{R}(x,t)} \tilde{\bW}^2  \, dx_2'\, dx_3'\, dx_4'\, dt'\leq C|\tau|^{\frac{1}{2}}e^{-\frac{351}{100}\tau}\|W_\cT\|_{2,\infty}^2(\tau+1).
\end{equation}
Moreover, observe that by Theorem \ref{thm_unique_asympt_recall} (sharp asymptotics) in the soliton region we have
\begin{equation}
|D\tilde{X}(x_2,x_3,x_4,t)|\leq C|\tau|^{1/2}.
\end{equation}
Hence, from \eqref{def_tildeWG} we infer that
\begin{equation}
|\tilde{\bG}(x_2,x_3,x_4,t)|\leq C|\tau|^{1/2}|f(\tilde{X}(x_2,x_3,x_4,t),x_2,x_3)|
\end{equation}
at all points at the soliton region. This implies
\begin{equation}\label{G_L_sup}
\frac{1}{R}   \int_{Q_{R}(x,t)} \tilde{\bG}^4  \, dx_2'\, dx_3'\, dx_4'\, dt' \leq CR^4|\tau|^2\sup_{Q_R(x,t)}|f|^4 \leq C |\tau|^2 \sup_{\Omega_h} \left(\frac{|f|}{H_{\phi}}\right)^4, 
\end{equation}
where the last inequality follows from $C^{-1} \leq |RH_{\phi}|\leq C$ on $Q_R(x,t)$. Together with Proposition \ref{prop_interior_estimates_solition} (interior estimates in soliton region) this yields the asserted estimate in the soliton region $v\leq \ell/\sqrt{|\tau|}$.\\

Suppose now $v\in [\ell/\sqrt{|\tau|},\frac{8}{9}\theta]$. Then, taking $x=\tilde{X}(e^{-\frac{\tau}{2}}v,0,0,t)$ and setting $r=H^{-1}(x,0,t)$ as in the previous subsection we compute
\begin{align}
\int_{Q_{r}(x,t)} \tilde{\bw}^2 \, dx'\, ds'\, dt' & \leq C |t|^{3}\int_{\tau-2}^{\tau+1}\int_{y-2}^{y+2} w^2(y',\tau')\, dy'\, d\tau' \nonumber \\
& \leq C |t|^{3}\int_{\tau-2}^{\tau+1}\int_{0}^{\theta} \frac{W^2(v,\tau)}{|Y_v|}\, dv\, d\tau,
\end{align}
where in the last inequality we have used Corollary \ref{cor_trans_rule} (transformation rule) and the change of variables formula $dy'=|Y_v|dv$.
Now, recall that by the tip estimates from \eqref{eq_tip_est} we have
\begin{equation}
\sup_{v\leq 2\theta}\left|\frac{v}{Y_v}\right| \leq  \frac{4}{\sqrt{|\tau|}},
\end{equation}
and by \cite[Lemma 5.34]{CHH_translators} we have the density bound
\begin{equation}
\inf_{v\leq 2\theta} \frac{1}{v}e^{{\bar{\mu}}(v,\tau)} \geq e^{\frac{51}{100}\tau}.
\end{equation} 
Using this, we infer that
\begin{equation}
\int_{Q_{r}(x,t)} \tilde{\bw}^2 \, dx'\, ds'\, dt'  \leq C|t|^{\frac{351}{100}}\|W_{\cT}\|^2_{2,\infty}(\tau+1).
\end{equation}
Moreover, since by \eqref{bg_def} we have $|\tilde{\bg}| \leq 2|f|$ for points outside the soliton region, arguing similarly as above, we see that
\begin{equation}
\int_{Q_r(x,t)} \tilde{\bg}^3 dx'\, ds'\, dt' \leq C |t|^{1/2}\sup_{\Omega_h} \left(\frac{|f|}{H_{\phi}}\right)^3.
\end{equation}
Hence, applying Proposition \ref{prop:interior_estimates_cylinder}  (interior estimates for cylindrical variation) we conclude that
\begin{equation}
|t|^{1/2}|w(y,\tau)| \leq C\frac{\log|t|}{|t|}|t|^{351/200}\|W_{\cT}\|_{2,\infty}(\tau+1)+C|t|^{1/3}|t|^{1/6} \sup_{\Omega_h} \left(\frac{|f|}{H_{\phi}}\right),
\end{equation}
where $(y,\tau)=(Y(v,\tau),\tau)$. By Corollary \ref{cor_trans_rule} (transformation rule) this proves the assertion.
\end{proof}

\begin{corollary}[{$C^k$-estimate for tip variation}]\label{C^2_tip_cor}
If $\tau\in [\tau_{h_\mathrm{in}},\tau_0-1]$ and  $v\leq \ell/\sqrt{|\tau|}$, then
\begin{multline}
\sum_{0\leq i+2j\leq k}|\partial^i_v\partial^j_\tau W|(v,\tau)\\
\leq e^{-\frac{\tau}{8}} \left( \sup_{v'\leq\frac{\theta}{4},|\tau'-\tau|\leq e^{\frac{\tau}{100k}}} | W(v',\tau')|
+\|\tilde{\bG}\|_{C_H^{k-2,\alpha,(2)}\left(Q_{\exp\left((1-\frac{1}{100k})\frac{\tau}{2}\right)}(e^{-\tau/2}v,-e^{-\tau})\right)}\right).
\end{multline}
\end{corollary}

\begin{proof}
Differentiating the transformation rule
\eqref{transwbf} yields\begin{align}
&W_v=\tilde{\bW}_{x_2}\, , && W_{vv}=e^{-\frac{\tau}{2}}\tilde{\bW}_{x_2x_2}\, , && W_\tau =\tfrac{1}{2}e^{\frac{\tau}{2}}\tilde{\bW}-\tfrac{v}{2}\tilde{\bW}_{x_2}+ e^{-\frac{\tau}{2}}\tilde{\bW}_t.
\end{align}
Hence, applying Proposition \ref{prop_interior_estimates_solition} (interior estimates in soliton region) the result follows.
\end{proof}

\bigskip

\section{Fredholm theory}

\subsection{Norms and spaces}
Recall that via \eqref{unren_tilde_quant} and \eqref{def_tildeWG} to any $u\in C^{k,\alpha}_{\mathrm{loc}}(\Omega_h/S^1)$ we associate the functions $\tilde{\bw}$ and $\tilde{\bW}$, and to any $f\in C^{k-2,\alpha}_{\mathrm{loc}}(\Omega_h/S^1)$ we associate the functions $\tilde{\bg}$ and $\tilde{\bG}$. Let us abbreviate
\begin{equation}
\mathcal{C}_{h'}=\left\{ (x_1,t)\; \bigg{|}\; h_0\leq -t\leq h',\;V(x_1,t) \geq \ell\sqrt{\tfrac{|t|}{\log |t|}} \right\} ,
\end{equation}
and
\begin{equation}
\mathcal{S}_{h'}=\left\{ (x_2,t)\; \bigg{|}\; h_0\leq -t\leq h',\; x_2 \leq \ell\sqrt{\tfrac{|t|}{\log|t|}} \right\}.
\end{equation}

\begin{definition}[domain and target H\"older norms]\label{def_holder_dom_tar} Denoting by $\rho_\star$ and $\rho_\bullet$ weight functions to be specified below, and given any $h'\in [2h_0, h]$, we define
the domain H\"older norm by
\begin{align}\label{Hold_def}
\|u\|_{C^{k,\alpha}_{\star}(\Omega_{h'}/S^1)}:=&\|u\|_{C^{k,\alpha}(\Omega_{2h_0})}+\sup\left\{\frac{1}{\rho_\star(x,t)}\|\tilde{\bw}\|_{C_H^{k,\alpha}\left(P_{r}\left(x,t\right)\right)}\;\bigg{|}\; (x,t)\in\mathcal{C}_{h'},\;r=\tfrac{1}{10}H^{-1}(x,t)\right\}\nonumber\\
&+\sup\left\{\|\tilde{\bW}\|_{C_H^{k,\alpha}\left(Q_r\left(x_2,t\right)\right)}\;\bigg{|}\;(x_2,t)\in\mathcal{S}_{h'},r=\sqrt{\tfrac{|t|}{\log|t|}}\right\}\, ,
\end{align}
and the target H\"older norm by
\begin{align}\label{im_Hold_def}
\|f\|_{C^{k-2,\alpha}_{\bullet}(\Omega_{h'}/S^1)}:=&\|f\|_{C^{k-2,\alpha}(\Omega_{2h_0})}+\sup\left\{\frac{1}{\rho_\bullet(x,t)}\|\tilde{\bg}\|_{C_H^{k-2,\alpha,(2)}\left(P_{r}\left(x,t\right)\right)}\;\bigg{|}\; (x,t)\in\mathcal{C}_{h'},\;r=\tfrac{1}{10}H^{-1}(x,t)\right\}\nonumber\\
&+\sup\left\{\|\tilde{\bG}\|_{C_H^{k-2,\alpha,(2)}\left(Q_r\left(x_2,t\right)\right)}\;\bigg{|}\;(x_2,t)\in\mathcal{S}_{h'},r=\sqrt{\tfrac{|t|}{\log|t|}}\right\}\, ,
\end{align}
where the parabolic cubes $P_r(x,t)$ and $Q_r(x_2,t)$ are defined in \eqref{Q_cyl_def} and \eqref{Q_tip_def}.
\end{definition}

Here, we work with the domain weight function
\begin{equation}\label{domain_weight}
\rho_\star(x,t):=\left\{\begin{array}{ll}
        \frac{1}{\log|t|}\left(\sqrt{2}+\frac{10}{\log|t|}-\frac{V(x,t)}{\sqrt{|t|}}\right) & \text{if } V(x,t)\geq \theta \sqrt{|t|}\\
        \frac{1}{\log|t|}\frac{|t|}{V(x,t)^2} &  \text{if } V(x,t)< \theta \sqrt{|t|},
        \end{array}\right. 
\end{equation}
and the target weight function
\begin{equation}\label{target_weight}
\rho_\bullet(x,t):=\left\{\begin{array}{ll}
        \frac{1}{\log|t|}\Big(\sqrt{2}+\frac{10}{\log|t|}-\frac{V(x,t)}{\sqrt{|t|}}\Big)^{2}  & \text{if } V(x,t)\geq \theta \sqrt{|t|}\\
        \frac{1}{\log|t|}+\frac{1}{(\log|t|)^{3/2}}\frac{|t|^{3/2}}{V(x,t)^3} &  \text{if } V(x,t)< \theta \sqrt{|t|}.
        \end{array}\right.
\end{equation}
Our choice of weight functions is motivated by the following proposition, which in particular will be used to verify the assumptions of the inner-outer estimate.

\begin{proposition}[controlled pointwise quantities]\label{prop_who_contr_whom} There exists a constant $C=C(M,h_0)<\infty$, such that whenever $h'\in[2h_0,h]$ then for all $\tau\in [-\log h',\tau_0]$ we have the bounds
\begin{equation}\label{contr_X0}
\sup_{v(y,\tau)\geq \theta}|\tau|\left(\sqrt{2}+\frac{10}{|\tau|}-v(y,\tau)\right)^{-1}|w(y,\tau)|+\sup_{v\leq \theta}  |\tau|^{1/2} |W(v,\tau)| \leq C\|u\|_{C^{0}_{\star}(\Omega_{h'}/S^1)},
\end{equation}
and
\begin{multline}\label{contr_Y0}
\sup_{v(y,\tau)\geq \theta}|\tau|\left(\sqrt{2}+\frac{10}{|\tau|}-v(y,\tau)\right)^{-2}|g(y,\tau)|\\
+\sup_{v\leq \theta} |\tau| \left(|\tau|^{1/2}+\frac{1}{v^3}\min\left(1,v^2|\tau|/\ell^2\right)\right)^{-1}  |G(v,\tau)|  \leq C\|f\|_{C^{0}_{\bullet}(\Omega_{h'}/S^1)}.
\end{multline}
\end{proposition}

\begin{proof}
Using \eqref{transwbf}, whenever $v(y,\tau)\geq  \ell/\sqrt{|\tau|}$ we can estimate
\begin{equation}
|w(y,\tau)|=e^{\frac{\tau}{2}}|\tilde{\bw}|(e^{-\frac{\tau}{2}}y,0,-e^{-\tau}) \leq 2 v(y,\tau) |H\tilde{\bw}| (e^{-\frac{\tau}{2}}y,0,-e^{-\tau}),
\end{equation}
and 
\begin{equation}
|g(y,\tau)|=e^{-\frac{\tau}{2}}|\tilde{\bg}|(e^{-\frac{\tau}{2}}y,0,-e^{-\tau}) \leq \frac{2}{v(y,\tau)} \left|\frac{\tilde{\bg}}{H}\right| (e^{-\frac{\tau}{2}}y,0,-e^{-\tau}).
\end{equation}
In particular, together with Corollary \ref{cor_trans_rule} (transformation rule) and the tip estimate from \eqref{eq_tip_est}, for $ v\in[ \ell |\tau|^{-1/2},\theta]$ this yields
\begin{equation}
|W(v,\tau)|\leq 2  |\tau|^{1/2} v^2 |H\tilde{\bw}| (e^{-\frac{\tau}{2}}Y(v,\tau),0,-e^{-\tau}),
\end{equation}
and
\begin{equation}
|G(v,\tau)|\leq 2 |\tau|^{1/2} \left|\frac{\tilde{\bg}}{H}\right| (e^{-\frac{\tau}{2}}Y(v,\tau),0,-e^{-\tau}).
\end{equation}
Moreover, using \eqref{trangbf}, whenever $v\leq \ell/\sqrt{|\tau|}$ we can estimate
\begin{equation}
|W(v,\tau)|=e^{\frac{\tau}{2}}|\tilde{\bW}(e^{-\frac{\tau}{2}}v,0,0,-e^{-\tau})| \leq C |\tau|^{-1/2} \left| H \tilde{\bW} \right|(e^{-\frac{\tau}{2}}v,0,0,-e^{-\tau}),
\end{equation}
and 
\begin{equation}
|G(v,\tau)|=e^{-\frac{\tau}{2}}|\tilde{\bG}(e^{-\frac{\tau}{2}}v,0,0,-e^{-\tau})| \leq 2|\tau|^{1/2}\left|\frac{\tilde{\bG}}{H}\right|(e^{-\frac{\tau}{2}}v,0,0,-e^{-\tau}).
\end{equation}
Remembering the definitions of the norms we thus infer that \eqref{contr_X0} and \eqref{contr_Y0} hold true.
\end{proof}

Our definition is designed so that we can patch together the different Schauder estimates.

\begin{proposition}[global Schauder estimate]\label{bdd_in_holder}
For all $h'\in[6h_0,h]$ we have the weighted Schauder estimate
\begin{equation}\label{global_Schauder2}
\|u\|_{C^{k,\alpha}_{\star}(\Omega_{e^{-1}h'}/S^1)} \leq C\left( \|u\|_{C^{0}_{\star}(\Omega_{h'}/S^1)}+\|Lu\|_{C^{k-2,\alpha}_{\bullet}(\Omega_{h'}/S^1)}\right).
\end{equation}
\end{proposition}

\begin{proof}
We will first check compatibility in the transition region $\sqrt{\log|t|/|t|}V(x,t)\in [\ell/2,2\ell]$.  To this end, note that by 
Corollary \ref{cor_trans_rule} (transformation rule), remembering also \eqref{transwbf} and \eqref{trangbf}, we have 
\begin{equation}
\tilde{\bW}(V(x,t),0,0,t)=-Y_v\left(\frac{x}{\sqrt{|t|}},-\log|t|\right)\tilde{\bw}(x,0,t),
\end{equation} 
and similarly for $\tilde{\bG}$. In the transition region this implies
\begin{equation}\label{bw_comet}
C^{-1}|\tilde{\bw}| \leq |\tilde{\bW}| \leq C|\tilde{\bw}|,\qquad C^{-1}|\tilde{\bg}| \leq |\tilde{\bG}| \leq C|\tilde{\bg}|,
\end{equation}
which together with $\rho_\ast \sim \rho_\bullet\sim 1$ yields the desired compatibility in the transition region.\\

Next, to check compatibility between heights $h_0$ and $2h_0-1$, let us rewrite equation \eqref{bg_def} in the form
\begin{equation}
\bg(x,t) =\eta(x,t) f(x,V(x,t),0),
\end{equation}
where we abbreviated $\eta:=\sqrt{1+V_x^2(x,t)+V_t^2(x,t)}$. Differentiating this formula gives
\begin{equation}
\bg_x = \eta (f_{x_1}+V_x f_{x_2})+ \eta_x f,\qquad \bg_t = \eta f_{x_2} V_t + \eta_t f.
\end{equation}
Hence, local elliptic $C^{k,\alpha}$ bounds for $f$ imply local \emph{elliptic} $C^{k,\alpha}$ bounds for $\bg$, and vice versa. Moreover, remembering that $\tilde{\bg}_s=\tilde{\bg}_t$ we see that these local elliptic $C^{k,\alpha}$ bounds for $\bg$ are in turn equivalent to local parabolic $C^{k,\alpha}$ bounds for $\tilde{\bg}$. More precisely, there exist $\lambda>0$ and $C<\infty$ such that whenever $h_0\leq -t \leq 2h_0 -1$ and $V(x,t)\geq \ell \sqrt{|t|/\log|t|}$, then setting $r=\min(\tfrac{1}{10}H^{-1}(x,t),1)$ we have 
\begin{equation}
C^{-1}\|\tilde{\bg}\|_{C_H^{k-2,\alpha,(2)}\left(P_{\lambda^2 r}\left(x,t\right)\right)}\leq \|f\|_{C^{k-2,\alpha}(B_{\lambda r}(x,V(x,t),0))}\leq C \|\tilde{\bg}\|_{C_H^{k-2,\alpha,(2)}\left(P_r\left(x,t\right)\right)}.
\end{equation}
Arguing similarly for $G$, $w$ and $W$ this establishes the desired compatibility of norms.\\

Finally, observe that we have the global inequalitiy
\begin{equation}
\rho_\star^{-1} \leq C\rho_\bullet^{-1}.
\end{equation}
Thus, the assertion follows from Proposition  \ref{prop:interior_estimates_cylinder} (interior estimates for cylindrical variation) and Proposition \ref{prop_interior_estimates_solition} (interior estimates for tip variation) and standard elliptic Schauder estimates in the cap region.\end{proof}

Let us now define the norms and spaces we will be working with:

\begin{definition}[norms and spaces] Fix an integer $k\geq 4$ and a real number $\alpha\in (0,1)$. Given any sufficiently large Dirichlet height $h\leq \infty$, for  functions $u\in C^{k,\alpha}_{\mathrm{loc}}(\Omega_h/S^1)$ satisfying $u|_{\partial \Omega_h}=0$ we define\begin{equation}\label{X_def}
\|u\|_{\mathbb{X}^{k,\alpha}(\Omega_h/S^1)}:=\|u\|_{C^{k,\alpha}_{\star}(\Omega_{ \exp(\log(h)^{1/2}-1)}/S^1)}+ \|Lu\|_{C^{k-2,\alpha}_{\bullet}(\Omega_{h}/S^1)},
\end{equation}
and for functions $f \in C^{k-2,\alpha}_{\mathrm{loc}}(\Omega_h/S^1)$ we simply set
\begin{equation}
\|f\|_{\mathbb{Y}^{k-2,\alpha}(\Omega_h/S^1)}:=\|f\|_{C^{k-2,\alpha}_{\bullet}(\Omega_{h}/S^1)},
\end{equation}
where the $C^{k,\alpha}_{\ast}$-norm and the $C^{k-2,\alpha}_{\bullet}$-norm are from Definition \ref{def_holder_dom_tar} (domain and target H\"older norms).
Finally, we denote corresponding Banach spaces by $\mathbb{X}^{k,\alpha}(\Omega_h/S^1)$ and $\mathbb{Y}^{k-2,\alpha}(\Omega_h/S^1)$.
\end{definition}

Observe that by definition
$L:\mathbb{X}^{k,\alpha}(\mathbb{R}^3/S^1) \rightarrow \mathbb{Y}^{k-2,\alpha}(\mathbb{R}^3/S^1)$ is a bounded linear map. For $h<\infty$ we work with $\Omega_{ \exp(\log(h)^{1/2}-1)}$, since the inner-outer estimate needs some time to kick in.\\

To conclude this subsection, let us observe that our norms in particular control several of the integral quantities that we encountered in previous sections:

\begin{corollary}[controlled global quantities]\label{controlled_global_quantities}
For any sufficiently large $h\leq \infty$ we have
\begin{equation}\label{int_ctrl}
\sup_{\Omega_h}\frac{|f|}{H_\phi}+\|\tau^3 g_{\cC} \|_{\mathfrak{H},\infty}
+ \|\tau^2 G_\cT \|_{2,\infty}\leq C\|f\|_{\mathbb{Y}^{k-2,\alpha}(\Omega_h/S^1)}.
\end{equation}
Moreover, if $h=\infty$ then we in addition also have
\begin{equation}\label{int_ctrl2}
\| \tau^2 w_{\cC}\|_{\mathcal{D},\infty}+ \|\tau^2 W_{\cT}\|_{2,\infty} 
\leq C\|u\|_{\mathbb{X}^{k,\alpha}(\mathbb{R}^3/S^1)}.
\end{equation}
\end{corollary}

\begin{proof}
Remembering that the weight $e^{{\bar{\mu}}}$ in the tip region is exponentially small, the integral estimates for $\tau^2 G_\cT$ and $\tau^2W_{\cT}$ easily follow from Proposition \ref{prop_who_contr_whom} (controlled pointwise quantities).\\
Next, again by Proposition \ref{prop_who_contr_whom} (controlled pointwise quantities), taking also into account the profile growth estimate from \eqref{profile_growth}, for all $\tau\in [-\log(h),\tau_0]$ and $|y|\leq Y(\theta/2,\tau)$ we have
\begin{equation}
|w(y,\tau)|\leq C \frac{(1+|y|)^2}{|\tau|^2}\|u\|_{C^{0}_{\bullet}(\Omega_{h}/S^1)},
\end{equation}
and
\begin{equation}
|g(y,\tau)|\leq C \frac{(1+|y|)^4}{|\tau|^3}\|f\|_{C^{0}_{\bullet}(\Omega_{h}/S^1)}.
\end{equation}
Moreover, since $w_y(y,\tau)=\tilde{\bw}_x(e^{-\frac{\tau}{2}}y,0,-e^{-\tau})$, for all $(y,\tau)$ as above we also have
\begin{equation}
|w_y(y,\tau)|\leq C \frac{(1+|y|)^2}{|\tau|^2}\|u\|_{C^{1}_{\bullet}(\Omega_{h}/S^1)},
\end{equation}
Combining the above inequalities we obtain the integral estimates for $\tau^2 w_{\cC}$ and $\tau^3 g_{\cC}$.\\
Finally, since the prefactors in the definition of $\bg$ and $\bG$ are always greater than $1$, remembering again the definition of the target norm, we get the sup-bound for $|f|/H$ as well.
\end{proof}

\bigskip

\subsection{Uniform estimate for restricted solution map} 

Throughout this subsection, we assume that $h<\infty$ is large enough so that the estimates from the previous sections apply. As before, we consider the Dirichlet problem
 \begin{equation}\label{bdval_prob_rest}
    \begin{cases}
      Lu=f & \text{in  $\Omega_h$}\\
      u=0 & \text{on $\partial \Omega_h$},\\
    \end{cases}       
\end{equation}
where $f\in \mathbb{Y}^{k-2,\alpha}(\Omega_h/S^1)$. In addition, we now assume that the cylindrical variation $w$ associated to $u$ via \eqref{wfromu} satisfies the orthogonality conditions
\begin{equation}
\fp_{+}(w_{\cC}(\tau_0))=\fp_0(w_{\cC}(\tau_0))=0.
\end{equation}
In other words, we consider the restricted map $L_R:\mathbb{X}^{k,\alpha}_\perp(\Omega_h/S^1)\to \mathbb{Y}^{k-2,\alpha}_\perp(\Omega_h/S^1)$, where
\begin{equation}
\mathbb{X}^{k,\alpha}_\perp(\Omega_h/S^1):=\{ u\in \mathbb{X}^{k,\alpha}(\Omega_h/S^1) \, | \fp_{+}(w_{\cC}(\tau_0))=\fp_0(w_{\cC}(\tau_0))=0\},\,\, \mathbb{Y}^{k-2,\alpha}_\perp(\Omega_h/S^1)=\textrm{Im}(L_R).
\end{equation}
Our goal is to prove that the norm of the map $L_R^{-1}: \mathbb{Y}^{k-2,\alpha}_\perp(\Omega_h/S^1)\to\mathbb{X}^{k,\alpha}_\perp(\Omega_h/S^1)$ is bounded by a constant independent of $h$.\\

To begin with, we consider the function
\begin{equation}
D(a):=a^{\frac{1}{2}} \max\Big(\sup\left\{| w(y,-\log a)|:    v(y,-\log a)\geq \tfrac{8}{9}\theta \right\},\sup\left\{| W(v,-\log a)|:	 v \leq \tfrac{9}{10} \theta\right\}\Big)\, ,
\end{equation}
which captures the maximal variation size at height $a$, measured in unrescaled units.
 
\begin{proposition}[$C^0$-estimate]\label{lemma:C0_entire}
For every time $\tau\in [-\log(h),\tau_0-1]$ we have
\begin{equation}
\sup\left\{a^{-\frac{76}{100}}D(a) \;:\;a\in [e^{-\tau},h]\right\}\leq C\left(\|\tau w_\cC\|_{\fH,\infty}(\tau+1)+\| W_\mathcal{T}\|_{2,\infty}(\tau+1)+\|f\|_{\mathbb{Y}^{k-2,\alpha}(\Omega_h/S^1)}\right) .
\end{equation}
\end{proposition}

\begin{proof}
Take $a\in [e^{-\tau},h]$ and set $\tau_a=-\log a$. Proposition \ref{cor:C0_cylinder} ({$L^\infty$-estimate for cylindrical variation}) gives
\begin{equation}
e^{\frac{1}{4}\tau_a}\sup\left\{| w(y,\tau_a)|:    v(y,\tau_a)\geq \tfrac{8}{9}\theta \right\}\leq C \|\tau w_\cC\|_{\fH,\infty}(\tau_a+1)+C\sup_{\Omega_h}\frac{|f|}{H_\phi},
\end{equation}
 and Proposition \ref{cor:C0_tip} ({$L^\infty$-estimate for tip variation}) gives
\begin{equation}
e^{\frac{26}{100}\tau_a}\sup\left\{| W(y,\tau_a)|:    v(y,\tau_a)\leq  \tfrac{9}{10}\theta \right\}\leq  \| W_\mathcal{T}\|_{2,\infty}(\tau_a+1)+\sup_{\Omega_h}\frac{|f|}{H_\phi}\, .
\end{equation}
Together with Corollary \ref{controlled_global_quantities} (controlled global quantities) this implies the assertion.
\end{proof}

Next, to state the $C^2$-estimate, recall that our derivative estimates kick in at time
\begin{equation}
\tau_{h_{\textrm{in}}}=-\log(h-h^{\gamma_k}),
\end{equation}
and that we use the notation
\begin{equation}
\|w \|_{C^2|C_\tau}:=\sup_{v(y,\tau)\geq \frac{5}{8}\theta} \big(|w|+|w_y|+|w_{\tau}|+ |w_{yy}|+|w_{\tau\tau}|\big)\, ,
\end{equation}
and
\begin{equation}
\|W \|_{C^2|T_\tau}:=\sup_{v\leq 2\theta } \big(|W|+|W_v|+|W_{\tau}|+ |W_{vv}|+|W_{\tau\tau}|\big)\, .
\end{equation}

\begin{proposition}[$C^2$-estimate]\label{lemma:Schauder_entire}
For every $\tau\in [\tau_{h_\mathrm{in}},\tau_0]$ we have
\begin{multline}
\|w\|_{C^2|C_\tau} +\|W\|_{C^2|T_\tau} \\
\leq e^{\frac{3\tau}{10}} \sup \left\{ D(a):\max\{-\log(h),\tau-1\}\leq -\log a \leq \tau+\tfrac{1}{100}\right\} +e^{-\frac{\tau}{5}}\|f\|_{\mathbb{Y}^{k-2,\alpha}(\Omega_h/S^1)}.
\end{multline}
\end{proposition}

\begin{proof}
To begin with, observe that by definition of $D$ we have
\begin{equation}\label{sup_small_w}
\sup\left\{| w(y,-\log a)|:    v(y,-\log a)\geq \ell/\sqrt{|\tau|} \right\} \leq a^{-1/2}D(a).
\end{equation}
Indeed, in the nontrivial case $ v\in [\ell/\sqrt{|\tau|},\tfrac{8}{9}\theta]$ this follows from the transformation rule $w=-v_y W$ as $|v_y| \leq 1 $ away from the soliton region. Note further that for every $(y,\tau)$ in the cylindrical or collar region, then as $r:=\tfrac{1}{10}H^{-1}(ye^{-\tau/2},-e^{-\tau}) \geq \exp(-\gamma_k\tau/2)$  we have the obvious bound
\begin{equation}
\|\tilde{\bg}\|_{C_H^{0,\alpha,(2)}\left(Q_{\exp(-\gamma_k\tau/2)}(e^{-\tau/2}y,-e^{-\tau})\right)} \leq \|f\|_{\mathbb{Y}^{k-2,\alpha}(\Omega_h/S^1)}.
\end{equation}
Thus, combined with Corollary \ref{C^2_cyl_cor} ($C^k$-estimate for cylindrical variation) we get 
\begin{multline}\label{w_est_D}
\sup_{v(y,\tau)\geq {\ell}/{\sqrt{|\tau|}}}\left(|w(y,\tau)|+|w_y(y,\tau)|+|w_{yy}(y,\tau)|+|w_{\tau}(y,\tau)|+|w_{\tau\tau}(y,\tau)|\right)\\
\leq e^{-\frac{\tau}{6}}\left(\sup \left\{ a^{-1/2}D(a):\max\{-\log(h),\tau-1\}\leq -\log a \leq \tau+\tfrac{1}{100}\right\}+ \|f\|_{\mathbb{Y}^{k-2,\alpha}(\Omega_h/S^1)}\right).
\end{multline}
Next, to obtain derivative estimate for $W$ in the collar region, observe that differentiating the relation $W(v,\tau)=-Y_v(v,\tau)w(Y(v),\tau)$ we see that $W$ derivatives are expressed as a combination of $w$ derivatives, with $Y$ derivatives as coefficients. Thus, using \eqref{w_est_D} and the tip estimates from \eqref{eq_tip_est} gives the desired estimates for $|W|+|W_v|+|W_{\tau}|+ |W_{vv}|+|W_{\tau\tau}|$ at points with $\ell/\sqrt{\tau}\leq v\leq 2\theta$. Finally, Corollary \ref{C^2_tip_cor} ($C^k$-estimate for tip variation) gives the desired estimate in the soliton region.
\end{proof}

\begin{lemma}[upper-lower bounds]\label{lemma:congruency_level_set_lin}
There exists a constant $C<\infty$ such that  for every  $a\in \left[ h_0,h\right]$ we have
\begin{equation}
\sup_{a'\in [a/e^2,a]}D(a') \leq 12\,(\log a)^{\frac{1}{2}}D(a)+C( a\log a)^{\frac{1}{2}}\|f\|_{\mathbb{Y}^{k-2,\alpha}(\Omega_h/S^1)}.
\end{equation}
\end{lemma}

\begin{proof} We will first relate $D(a)$ with $\sup_{\partial \Omega_a}|u|$. To this end, recall that by \eqref{mul_vu} and the cylindrical estimates from \eqref{eq_cyl_est}, away from the tip region we have
\begin{equation}
v_\tau+\frac{y}{2}v_y-\frac{v}{2}=-\frac{1}{v} +  o(1).
\end{equation}
Hence, by the transformation rule \eqref{wfromu}, for $v\geq \frac{8\theta}{9}$ and $a\geq e^{-\tau_{0}}$ we get
\begin{align}\label{w_part_connection}
&\tfrac{1}{2}|u(\sqrt{a}y,\sqrt{a}v(y,-\log a),0)| \leq  a | w(y,-\log a)| \leq \tfrac{2}{\theta } |u(\sqrt{a}y,\sqrt{a}v(y,-\log a),0)|. 
\end{align}
Similarly, using the transformation rule \eqref{Wfromu}, and estimating the prefactor as in \eqref{mul_Yu}, for $v\leq \tfrac{9}{10}\theta$ we get
\begin{equation}\label{W_part_connection}
\tfrac{1}{2}\sqrt{\log a}\, |u(\sqrt{a}Y(v,\tau),\sqrt{a}v,0)|\leq a |W(v,-\log a)| \leq 2 \sqrt{\log a}\, |u(\sqrt{a}Y(v,\tau),\sqrt{a}v,0)|. 
\end{equation}
Together with the definition of $D$, provided $\tau_0$ is sufficiently negative, we thus infer that
\begin{equation}
\frac{1}{2\sqrt{a}}\sup_{\partial \Omega_a} |u| \leq D(a) \leq \frac{2\sqrt{\log a}}{\sqrt{a}}\sup_{\partial \Omega_a} |u|.\\
\end{equation}
On the other hand, by the global barrier from Theorem \ref{max_princ_inhom} (upper-lower estimate) that for every $a'<a$ we have
\begin{equation}
\sup_{\Omega_{a'}} |u| \leq \sup_{\Omega_{a}} |u|+Ca \|f\|_{\mathbb{Y}^{k-2,\alpha}(\Omega_h/S^1)}.
\end{equation}
Thus, for any $a'\in [a/e^2,a]$ we conclude that
\begin{multline}
D(a')  \leq \frac{2\sqrt{\log a'}}{\sqrt{a'}}\left(\sup_{\partial \Omega_{a}} |u|+Ca\|f\|_{\mathbb{Y}^{k-2,\alpha}(\Omega_h/S^1)}\right) \\
 \leq 12\,(\log a)^{\frac{1}{2}}D(a)+C( a\log a)^{\frac{1}{2}}\|f\|_{\mathbb{Y}^{k-2,\alpha}(\Omega_h/S^1)}.
\end{multline}
This proves the lemma.
\end{proof}

Combining the above three results, as well as results from prior sections, we now obtain:

\begin{proposition}[uniform integral estimates]\label{un_est_prop}
There exists a constant $C<\infty$ such that
\begin{equation}
\|\tau^2\fp_0(w_{\cC})\|_{\fH,\infty}+\|\tau^3 (w_\cC-\fp_0(w_{\cC}))\|_{\cD,\infty}+\| \tau^3 W_\cT \|_{2,\infty} \leq   C\|f\|_{\mathbb{Y}^{k-2,\alpha}(\Omega_h/S^1)}.
\end{equation}
\end{proposition}
\begin{proof}
For ease of notation, let us set $D(a)=0$ for $a\geq h$. Recalling that by definition,
\begin{equation}
\|w\|_{C^2_{\exp}(\mathcal{C})} =\sup_{\tau\in [\tau_{h_\mathrm{in}},\tau_0]}e^{\frac{49}{100}\tau}\|w\|_{C^2|C_\tau}\,,\qquad \qquad\| W\|_{C^2_{\exp}(\mathcal{T})}=\sup_{\tau\in [\tau_{h_\mathrm{in}},\tau_0]}e^{\frac{99}{100}\tau}\|W\|_{C^2|T_\tau}\, ,
\end{equation}
we can choose a time $\tau'\in [\tau_{h_\mathrm{in}},\tau_0]$ such that  
\begin{equation}\label{almost_max_choice}
e^{\frac{99}{100}\tau'}\|W\|_{C^2|T_\tau'}+e^{\frac{49}{100}\tau'}\|w\|_{C^2|C_\tau'}\geq \tfrac{1}{2}\left(\|w\|_{C^2_{\exp}(\mathcal{C})}+ \| W\|_{C^2_{\exp}(\mathcal{T})}\right).
\end{equation}
Applying Proposition \ref{lemma:Schauder_entire} ($C^2$-estimate) and Lemma \ref{lemma:congruency_level_set_lin} (upper-lower bounds) we can estimate
\begin{align}\label{exp_norm_bound_1}
\|w\|_{C^{2}|C_{\tau'}} +\|W\|_{C^{2}|T_{\tau'}} \leq e^{\frac{29\tau'}{100}} D(e^{-\tau'+1})+e^{-\frac{21\tau'}{100}}\|f\|_{\mathbb{Y}^{k-2,\alpha}(\Omega_h/S^1)}.
\end{align}
Together with Proposition \ref{lemma:C0_entire} ($C^0$-estimate) this implies
\begin{equation}
\|w\|_{C^{2}|C_{\tau'}} +\|W\|_{C^{2}|T_{\tau'}} \leq e^{-\frac{47\tau'}{100}}\left(\|\tau w_{\cC}\|_{\fH,\infty}+\|W_{\cT}\|_{2,\infty}+\|f\|_{\mathbb{Y}^{k-2,\alpha}(\Omega_h/S^1)}\right).
\end{equation}
By our choice of $\tau'$ from \eqref{almost_max_choice}, we therefore have shown that
\begin{equation}\label{wexp_bd}
 \|w\|_{C^2_{\exp}(\cC)}+\|W\|_{C^2_{\exp}(\cT)} \leq e^{\frac{\tau'}{50}}\left(\|\tau w_\cC\|_{\fH,\infty}+\| W_\mathcal{T}\|_{2,\infty}+\|f\|_{\mathbb{Y}^{k-2,\alpha}(\Omega_h/S^1)}\right).
\end{equation}
On the other hand, remembering Corollary \ref{controlled_global_quantities} (controlled global quantities), Theorem \ref{main_decay_lin_inhom} (decay estimate) reads
\begin{multline}\label{decay_recall}
\|\tau^2\fp_0(w_{\cC})\|_{\fH,\infty}+\|\tau^3 (w_\cC-\fp_0(w_{\cC}))\|_{\cD,\infty} +\| \tau^3 W_\cT \|_{2,\infty} \\
\leq   C\left( \|w\|_{C^2_{\exp}(\cC)}+\|W\|_{C^2_{\exp}(\cT)}+\|f\|_{\mathbb{Y}^{k-2,\alpha}(\Omega_h/S^1)}\right).
\end{multline}
Combining the above two estimates the assertion follows.
 \end{proof}

\begin{theorem}[uniform estimate for restricted solution map]\label{un_est_Xh}
There exists a constant $C=C(\phi)<\infty$ with the following significance. For any sufficiently large $h<\infty$ and any solution $u$ of the Dirichlet problem \eqref{bdval_prob_rest}, such that  the associated function $w$ satisfies the orthogonality condition $\fp_{+}(w_{\cC})=\fp_0(w_{\cC})=0$, we have
\begin{equation}
\|u\|_{\mathbb{X}^{k,\alpha}(\Omega_h/S^1)} \leq C\|f\|_{\mathbb{Y}^{k-2,\alpha}(\Omega_h/S^1)}.
\end{equation}
\end{theorem}

\begin{proof}
By Corollary \ref{controlled_global_quantities} (controlled global quantities) and Proposition \ref{un_est_prop} (uniform integral estimates) we know that
\begin{equation}\label{beginnig_bd_norm_ctrl}
\sup_{\Omega_h}\frac{|f|}{H_\phi}+\|\tau^2 w_{\cC}\|_{\fH,\infty} +\| \tau^3 W_\cT \|_{2,\infty} \leq   C\|f\|_{\mathbb{Y}^{k-2,\alpha}(\Omega_h/S^1)}.
\end{equation}
Hence, Corollary \ref{cor:C0_cylinder}  ($L^{\infty}$-estimates for cylindrical region) and Corollary \ref{cor:C0_tip} ($L^{\infty}$-estimate for tip variation) yield
\begin{equation}\label{max_lower_bd}
\sup_{ \partial \Omega_{2h_0}}|u| \leq C\|f\|_{\mathbb{Y}^{k-2,\alpha}(\Omega_h/S^1)},
\end{equation}
which by Proposition \ref{sub_sol_prop} (global subsolution) implies
\begin{equation}\label{max_lower_bd}
\sup_{ \Omega_{2h_0}}|u| \leq C\|f\|_{\mathbb{Y}^{k-2,\alpha}(\Omega_h/S^1)}.
\end{equation}
More crucially, Corollary \ref{cor:C0_cylinder}  ($L^{\infty}$-estimates for cylindrical region) also yields
\begin{equation}
\sup_{|y|\leq \ell}\left(\sup_{\tau\in [-\log(h)+1,\tau_0]} |\tau|^2|w(y,\tau)|+\sup_{\tau\in [-\log(h),-\log(h)+1]} |\tau|^{3/2}|w(y,\tau)|\right) \leq C\|f\|_{\mathbb{Y}^{k-2,\alpha}(\Omega_h/S^1)}.
\end{equation}
Hence, remembering also Proposition \ref{prop_who_contr_whom} (controlled pointwise quantities), we see that the assumptions of Theorem \ref{linear_sup_sol_est} (inner-outer estimate) are verified with $A=C\|f\|_{\mathbb{Y}^{k-2,\alpha}(\Omega_h/S^1)}$, and thus for all $\tau\in [-\sqrt{\log(h)},\tau_0]$ we get
\begin{equation}\label{w_closed_bd}
\sup_{v(y,\tau)\geq \theta}|\tau|(\sqrt{2}-v)^{-1}|w(y,\tau)|  + \sup_{v\leq \theta} |\tau|^{1/2} |W(v,\tau)| \leq C\|f\|_{\mathbb{Y}^{k-2,\alpha}(\Omega_h/S^1)}.
\end{equation}
Moreover, by a similarly argument as in the proof of Proposition \ref{prop_who_contr_whom} (controlled pointwise quantities), whenever $v(y,\tau)\geq \ell/{\sqrt{|\tau|}}$ we have
\begin{equation}
|H{\tilde{\bw}}|(e^{-\tau/2}y,0,-e^{-\tau})\leq 2v(y,\tau)^{-1} |w(y,\tau)|,
\end{equation}
and whenever $v(y,\tau)\leq \ell/{\sqrt{|\tau|}}$ we have
\begin{equation}
| H\tilde{\bW} |(e^{-\tau/2}v,0,0,-e^{-\tau})\leq C|\tau|^{1/2}|W(v,\tau)|.
\end{equation}
Combining the above inequalities, remembering the definition of the domain norm, and using also the fact that $ |w|\leq C v^{-1} |\tau|^{-1/2}|W|$ in the collar region, we thus infer that
\begin{equation}
\|u\|_{C^0_\star(\Omega_{\exp(\log(h)^{1/2})}/S^1)} \leq C\|f\|_{\mathbb{Y}^{k-2,\alpha}(\Omega_h/S^1)}.
\end{equation}
Finally, applying Proposition \ref{bdd_in_holder} (global Schauder estimate) we conclude that
\begin{equation}
\|u\|_{C^{k,\alpha}_{\star}(\Omega_{\exp((\log h)^{1/2}-1)})}\leq C\|f\|_{\mathbb{Y}^{k-2,\alpha}(\Omega_h/S^1)}.
\end{equation}
This finishes the proof of the theorem.
\end{proof}

\bigskip

\subsection{Fredholm property conclusion} Using our uniform estimates we can now establish the Fredholm property:

\begin{theorem}[Fredholm property]\label{Fredholm_theorem}
The map $L:\mathbb{X}^{k,\alpha}(\mathbb{R}^3/S^1)\rightarrow \mathbb{Y}^{k-2,\alpha}(\mathbb{R}^3/S^1)$ is Fredholm.
\end{theorem}

\begin{proof}
Since $L:\mathbb{X}^{k,\alpha}(\mathbb{R}^3/S^1)\rightarrow \mathbb{Y}^{k-2,\alpha}(\mathbb{R}^3/S^1)$ is a bounded linear map by definition of our Banach spaces, it is enough to show that the kernel and cokernel are finite-dimensional.

We will show first that the cokernel of $L$ has dimension at most three. For that, it suffices to show that if $\mathbb{W}\subseteq \mathbb{Y}^{k-2,\alpha}(\mathbb{R}^3/S^1)$ is any four dimensional subspace then $ \mathbb{W} \cap \mathrm{Range}(L)\neq \{0\}$. To this end, consider the obstruction 
\begin{equation}
\mathcal{O}(u):=(\fp_{+}\left(w_{\cC}),\fp_0(w_{\cC})\right),
\end{equation}
where $w$ is associated to $u$ via \eqref{wfromu}.
Now, for every integer $j\gg 1$ and every $f\in \mathbb{Y}^{k-2,\alpha}(\mathbb{R}^3/S^1)$, denote by $u_{f,j}\in \mathbb{X}(\Omega_j/S^1)$ the unique solution of the boundary value problem    
\begin{equation}\label{bdval_prob_recall}
    \begin{cases}
      Lu_{f,j}=f|_{\Omega_j} & \text{on  $\Omega_j$}\\
      u_{f,j}=0 & \text{on $\partial \Omega_j$}.\\
    \end{cases}       
\end{equation}
Consider
\begin{equation}
\mathbb{V}_j:=\left\{f\in \mathbb{Y}^{k-2,\alpha}(\mathbb{R}^3/S^1)\;|\; \mathcal{O}(u_{f,j})=0\right\},
\end{equation}
and note that since $\mathcal{O}(u_{f,j})=0$ is given by 3 linear equations, this is a subspace of codimension 3. Thus, there exists some $f_j\in \mathbb{W}\cap \mathbb{V}_j$ with $\|f_j\|_{\mathbb{Y}^{k-2,\alpha}(\mathbb{R}^3/S^1)} =1$. Setting $u_j:=u_{f_j,j}$, we have $\mathcal{O}(u_j)=0$, and thus Theorem \ref{un_est_Xh} (uniform estimate for solution map) gives the uniform estimate
\begin{equation}\label{un_est_recall}
\|u_j\|_{\mathbb{X}^{k,\alpha}(\Omega_j/S^1)} \leq C.
\end{equation}
Since $\mathbb{W}$ is a finite dimensional, there exists some $f\in \mathbb{W}$ with $\|f\|_{\mathbb{Y}^{k-2,\alpha}(\mathbb{R}^3/S^1)}=1$ such that after passing to a subsequence we have $f_j\xrightarrow{\mathbb{Y}^{k-2,\alpha}(\mathbb{R}^3/S^1)} f$. We will show that $f\in \mathrm{Range}(L)$.
Indeed, given any $R<\infty$, for all $j$ sufficiently large we have
\begin{equation}\label{comp_finite}
\|u_j\|_{C^{k,\alpha}(B_R)} \leq C(R)\|u_j\|_{\mathbb{X}^{k,\alpha}(\Omega_j/S^1)} \leq C(R),
\end{equation} 
so we can find a subsequence that converges in $C^{k}_{\mathrm{loc}}(\mathbb{R}^3/S^1)$ to a limit $u\in C^{k,\alpha}_{\mathrm{loc}}(\mathbb{R}^3/S^1)$, and we have
\begin{equation}
Lu=f.
\end{equation}
Since our norm is defined as supremum over compactly supported quantities, and all of these quantities are lower semicontinous under the convergence, we have
\begin{equation}
\|u\|_{\mathbb{X}^{k,\alpha}(\mathbb{R}^3/S^1)} \leq \liminf_{j\to \infty}\|u_j\|_{\mathbb{X}^{k,\alpha}(\Omega_j/S^1)}<\infty,
\end{equation}
hence $u\in\mathbb{X}^{k,\alpha}(\mathbb{R}^3/S^1)$.

For the proof of finite dimensional kernel observe that it suffices to show that $\mathrm{Ker}(L)\cap \mathrm{Ker}(\mathcal{O})=0$, as $\mathrm{Ker}(\mathcal{O})$ intersects nontrivially any $4$-dimensional subspace of $\mathbb{X}^{k,\alpha}(\mathbb{R}^3/S^1)$. 
Now, if $u\in \mathrm{Ker}(L)\cap \mathrm{Ker}(\mathcal{O})$, then thanks to Corollary \ref{controlled_global_quantities} (controlled global quantities) the assumptions of Corollary \ref{main_decay_cor} (decay estimate for entire homogenous solutions) are satisfied.
 Therefore, the conclusion of Proposition \ref{un_est_prop} (uniform integral estimates) holds for this $u$ with $f=0$ and $h=\infty$. Hence, we conclude that $u=0$.
\end{proof}

\bigskip

\section{Nonlinear theory}

\subsection{Quadratic error estimate}

In this subsection, to conveniently show analyticity, we consider the complexification of the spaces and maps from the previous sections.
Moreover, we do not assume that the point $\phi_0$ around which we expand is a translator, rather we only assume $\phi_0=\phi_\ast+u_0$, where $\phi_\ast$ is a translator and $u_0$ has small $\mathbb{X}^{k+2,\alpha}$-norm.
We now consider the quadratic quantity
\begin{equation}
Q_{\phi_0}[u]:= \Theta[\phi_0+u]-\Theta[\phi_0]-L_{\phi_0} u,
\end{equation}
where $\Theta$ is the graphical translator operator as defined in \eqref{theta_def}, and where
\begin{equation}
L_{\phi_0} u=\mathrm{div}(a_{\phi_0}  Du)+ b_{\phi_0}\cdot Du,
\end{equation}
with
\begin{equation}
a_{\phi_0}=\frac{\delta}{\sqrt{1+|D {\phi_0}|^2}}  -\frac{D{\phi_0}\otimes D{\phi_0}}{{(1+|D{\phi_0}|^2)^{3/2}}} , \qquad b_{\phi_0}=\frac{D {\phi_0}}{(1+|D {\phi_0}|^2)^{3/2}}.
\end{equation}

\begin{proposition}[quadratic quantity in graphical gauge]\label{lemma_quad_graph} We have
\begin{align}
Q_{\phi_0}[u]=&\mathrm{div}\left(\frac{D({\phi_0}-u)(Du)^2-2Du(D{\phi_0}\cdot Du)+D{\phi_0} K_{\phi_0}[u] (2D{\phi_0}\cdot Du - (Du)^2) (D{\phi_0}\cdot Du)}{J_{\phi_0}[u]}\right)\nonumber\\
&\qquad\qquad -\frac{(Du)^2+K_{\phi_0}[u] (2D{\phi_0}\cdot Du - (Du)^2)(D{\phi_0}\cdot Du)}{J_{\phi_0}[u]},
\end{align}
where $J_{\phi_0}[u]$ and $K_{\phi_0}[u]$ are specified below.
\end{proposition}

\begin{proof}
Note that
\begin{equation}
\frac{1}{\sqrt{1+(D({\phi_0}-u))^2}}-\frac{1}{\sqrt{1+(D{\phi_0})^2}}=\frac{2D{\phi_0}\cdot Du - (Du)^2}{J_{\phi_0}[u]},
\end{equation}
where
\begin{equation}
J_{\phi_0}[u]:=\sqrt{1+(D{\phi_0})^2}\sqrt{1+(D ({\phi_0}-u))^2}\left(\sqrt{1+(D{\phi_0})^2}+\sqrt{1+(D ({\phi_0}-u))^2} \right).
\end{equation}
Moreover, note that
\begin{equation}
\frac{2}{J_{\phi_0}[u]}-\frac{1}{(1+(D{\phi_0})^2)^{3/2}}=\frac{K_{\phi_0}[u] (2D{\phi_0}\cdot Du - (Du)^2)}{J_{\phi_0}[u]},
\end{equation}
where
\begin{equation}
K_{\phi_0}[u]:=\frac{2\sqrt{1+(D{\phi_0})^2}+\sqrt{1+(D ({\phi_0}-u))^2}}{(1+(D{\phi_0})^2)^{3/2}}.
\end{equation}
Using these algebraic identities, the assertion follows from a direct computation.
\end{proof}

\begin{corollary}[quadratic error for graphical variation]\label{quad_err_graph}
 There exist constants $\eps>0$ and $C<\infty$, such that for all $\phi_0=\phi_\ast+u_0$ with $ \| u_0 \|_{\mathbb{X}^{k+2,\alpha}(\mathbb{R}^3/S^1,\mathbb{C})}\leq \eps$ we have
\begin{align}
\| Q_{\phi_0}[ u ]\|_{C^{k-2,\alpha}(\Omega_{2h_0},\mathbb{C})}\leq C \| u \|_{C^{k,\alpha}(\Omega_{2h_0},\mathbb{C})}^2,
\end{align}
provided that $\| u \|_{C^{k,\alpha}(\Omega_{2h_0},\mathbb{C})}\leq \eps$. 
\end{corollary}

\begin{proof}
Since $\phi_0$ and its derivatives are uniformly bounded in the cap region, this immediately follows from Lemma \ref{lemma_quad_graph} (quadratic quantity in graphical gauge).
\end{proof}

We will next transform our quadratic quantity to cylindrical gauge. To set things up, we introduce the notation
\begin{equation}
(u\diamond V)(x,t):= u(x,V(x,t),0).
\end{equation}
Let $\bw_0$ and $\bw_0+\bw$ be the cylindrical variations associated to $u_0$ and $u_0+u$, respectively.  Applying the formula \eqref{bw_def} twice and taking the difference of the resulting equations we infer that
\begin{equation}
\bw=-V_t\,\, u \diamond V,
\end{equation}
where $V$ denotes the cylindrical profile function of the translator $\phi_\ast$. We would like to compute
\begin{equation}
Q_{\phi_0}[u]\diamond V= \Theta[\phi_0+u]\diamond V-\Theta[\phi_0]\diamond V-L_{\phi_0} u\diamond V.
\end{equation}
To this end, denote by $V^u$ the (local) cylindrical profile function of $\phi_0+u$ defined implicitly via
\begin{equation}
(\phi_0+u)\diamond V^u= -  t,
\end{equation}
in the region where it exists. In particular, $V^0$ denotes the profile function of $\phi_0$, which is generally different from the profile function $V$ of $\phi_\ast$.
Observe that by equation \eqref{V_eq_der} we have
\begin{multline}\label{Theta_phiu}
\Theta[\phi_0+u]\diamond V^u=\\
\frac{-1}{\sqrt{1+(V^u_x)^2+(V^u_t)^2}}\bigg(
\frac{(1+(V^u_t)^2)V^u_{xx}+(1+(V^u_x)^2)V^u_{tt}-2V^u_xV^u_tV^u_{xt}}{1+(V^u_x)^2+(V^u_t)^2} -\frac{1}{V^u} - V^u_t\bigg).
\end{multline}
Now, to relate $V^u$ and $V^0$, note that setting $\bw^u:=\bw_0+\bw$  as a consequence of the definitions we have
\begin{equation}
(\phi_0+u)\diamond V= - t + \frac{\bw^u}{V_t},
\end{equation}
hence
\begin{equation}\label{v_vs_vu}
V(x,t)=V^u\left(x, t - \frac{\bw^u}{V_t}(x,t)\right).
\end{equation}
Let us also introduce the notation
\begin{equation}
\bw_{;i}=\bw_i -\frac{V_{it}}{V_t}\bw,\qquad \bw_{;ij}=\bw_{ij} -\frac{V_{ijt}}{V_t}\bw, \qquad \textrm{where } i,j\in\{x,t\},
\end{equation}
and
\begin{equation}
I[\bw]=1-\frac{\bw_{;t}}{V_t}.
\end{equation}

\begin{lemma}[derivatives of cylindrical profile]\label{lemma_der_cyl_prof}
The first derivative of $V^u$ can be expressed as
\begin{align}\label{V_1_t_eq}
V_t^u = V_t^0+\frac{\bw_{;t}}{I[\bw_0]^2}   +  \frac{\bw_{;t}^2}{I[\bw_0]^2I[\bw_0+\bw]V_t },
\end{align}
and
\begin{align}
V_x^u = V_x^0+\frac{\bw_{;x}}{I[\bw_0]}+\frac{\bw_{0;x}\bw_{;t}}{I[\bw_0]^2V_t}+\frac{\bw_{;x}\bw_{;t}}{I[\bw_0]I[\bw_0+\bw]V_t}+\frac{\bw_{0;x}\bw_{;t}^2}{I[\bw_0]^2I[\bw_0+\bw]V_t^2},
\end{align}
and the second derivatives of $V^u$ can be expressed as 
\begin{align}
    V^u_{tt}=V^0_{tt}+\frac{\bw_{;tt}}{I[\bw_0]^3}+\frac{3}{I[\bw_0]^4}\left(\frac{\bw_{0;tt}}{V_t}-\frac{2V_{tt}\bw_{0;t}}{V_t^2}\right)\bw_{;t}+q_{tt}[\bw],
\end{align}
and
\begin{align}
    V^u_{tx}=V^0_{tx}+\frac{\bw_{;xt}}{I[\bw_0]^2}+ \frac{\bw_{0;x}\bw_{;tt}}{I[\bw_0]^3V_t}+ \frac{1}{I[\bw_0]^3} \left(\frac{\bw_{0;tt}}{V_t}-\frac{2V_{tt}\bw_{0;t}}{V_t^2}\right)\bw_{;x}+c[\bw_0]\bw_{;t}+q_{tx}[\bw],
\end{align}
and
\begin{align}
   V_{xx}^u=V^0_{xx}+\frac{\bw_{;xx}}{I[\bw_0]} +\frac{2\bw_{0;x}}{I[\bw_0]^2V_t}\bw_{;xt} +\frac{\bw_{0;x}}{I[\bw_0]^3V_t}\bw_{;tt}+a[\bw_0]\bw_{;x}+b[\bw_0]\bw_{;t}+q_{xx}[\bw],
\end{align}
where $a[\bw_0]$, $b[\bw_0]$ and $c[\bw_0]$ as well as $q_{tt}[\bw]$,  $q_{tx}[\bw]$ and  $q_{xx}[\bw]$ are specified below. Here, the profile function $V^u$ and its derivatives are evaluated at $\big(x, t - \tfrac{\bw^u}{V_t}(x,t) \big)$, and likewise $V^0$ and its derivatives are evaluated at $\big(x, t - \tfrac{\bw_0}{V_t}(x,t) \big)$.
\end{lemma}

\begin{proof}
Differentiating \eqref{v_vs_vu} we get
\begin{equation}
V^u_t = \frac{V_t}{I[\bw^u]},\qquad V^u_x =  V_x+ \frac{\bw^u_{;x}}{I[\bw^u]}.
\end{equation}
It follows that
\begin{align}
V_t^u -V_t^0=\frac{\bw_{;t}}{I[\bw_0]I[\bw_0+\bw]}, \qquad  V_x^u - V_x^0=\frac{\bw_{;x}}{I[\bw_0+\bw]}+\frac{\bw_{0;x}\bw_{;t}}{I[\bw_0]I[\bw_0+\bw]V_t}.
\end{align}
Observing also that we have the algebraic identity
\begin{align}
 \frac{1}{I[\bw_0+\bw]}  = \frac{1}{I[\bw_0]}   +  \frac{\bw_{;t}}{I[\bw_0]I[\bw_0+\bw]V_t},
\end{align}
this yields the claimed formulas for the first derivatives. Differentiating again we obtain
\begin{equation}
V_{tt}^u=\frac{V_{tt}}{I[\bw^u]^2}+\frac{\bw^u_{;tt}-2\frac{V_{tt}}{V_t}\bw^u_{;t}}{I[\bw^u]^3}  ,
\end{equation}
and
\begin{equation}
V_{xt}^u=\frac{V_{tx}}{I[\bw^u]} + \frac{\bw^u_{;xt}-\frac{V_{tx}}{V_t} \bw^u_{;t}}{I[\bw^u]^2}+\frac{\frac{1}{V_t}\bw^u_{;x}(\bw^u_{;tt}-2\frac{V_{tt}}{V_t}\bw^u_{;t})}{I[\bw^u]^3},
\end{equation}
and
\begin{align}
    V_{xx}^u=V_{xx}+ \frac{\bw^u_{;xx}}{I[\bw^u]}+\frac{\frac{1}{V_t}\bw^u_{;x}(2\bw^u_{;xt}-2\frac{V_{xt}}{V_t}\bw^u_{;t}-\frac{V_{tt}}{V_t}\bw^u_{;x})}{I[\bw^u]^2}+\frac{\frac{1}{V_t^2}(\bw^u_{;x})^2(\bw^u_{;tt}-\frac{2V_{tt}}{V_t}\bw^u_{;t})}{I[\bw^u]^3}.
\end{align}
This implies the claimed formulas for the second derivatives with\footnote{We decided to write down these formulas for concreteness, but only their structure not their precise form is important.}
\begin{equation}\label{abwdef}
a[\bw_0]=\frac{2}{I[\bw_0]^2}\left( \frac{\bw_{0;xt}}{V_t}-\frac{V_{tt}\bw_{0;x}+V_{xt}\bw_{0;t}}{V_t^2} \right)+\frac{2}{I[\bw_0]^3}\frac{\bw_{0;x}}{V_t}\left(\frac{\bw_{0;tt}}{V_t}-\frac{2V_{tt}\bw_{0;t}}{V_t^2} \right),
\end{equation}
and
\begin{multline}\label{bbwdef}
b[\bw_0]=\frac{1}{I[\bw_0]^2}\left(\frac{\bw_{0;xx}}{V_t}-\frac{2V_{xt}\bw_{0;x}}{V_t^2} \right) 
+\frac{4}{I[\bw_0]^3}\frac{\bw_{0;x}}{V_t}\left( \frac{\bw_{0;xt}}{V_t}-\frac{V_{tt}\bw_{0;x}+V_{xt}\bw_{0;t}}{V_t^2} \right)\\
+\frac{3}{I[\bw_0]^4}\frac{\bw_{0;x}^2}{V_t^2}\left(\frac{\bw_{0;tt}}{V_t}-\frac{2V_{tt}\bw_{0;t}}{V_t^2} \right),
\end{multline}
and
\begin{equation}\label{cbwdef}
c[\bw_0]=\frac{2}{I[\bw_0]^3}\left( \frac{\bw_{0;xt}}{V_t}-\frac{V_{tt}\bw_{0;x}+V_{xt}\bw_{0;t}}{V_t^2} \right)+\frac{3}{I[\bw_0]^4}\frac{\bw_{0;x}}{V_t}\left(\frac{\bw_{0;tt}}{V_t}-\frac{2V_{tt}\bw_{0;t}}{V_t^2} \right),
\end{equation}
as well as
\begin{align}
q_{tt}[\bw]=&\,\frac{\frac{3}{V_t}\bw_{;t}(\bw_{;tt}-2\frac{V_{tt}}{V_t}\bw_{;t})}{I[\bw_0]^4}+\frac{3\frac{V_{tt}}{V_t^2}\bw_{;t}^2}{I[\bw_0]^2I[\bw^u]^2}\nonumber\\
&+\frac{\frac{3}{V_t^2}\bw_{;t}^2(\bw^u_{;tt}-2\frac{V_{tt}}{V_t}\bw^u_{;t})(I[\bw^u]+I[\bw_0])}{I[\bw_0]^4I[\bw^u]^2}+\frac{\frac{1}{V_t^3}\bw_{;t}^3(\bw^u_{;tt}-2\frac{V_{tt}}{V_t}\bw^u_{;t})}{I[\bw_0]^3I[\bw^u]^3},
\end{align}
and 
\begin{align}\label{V_1_tx_eq}
q_{tx}[\bw]=&\, \frac{\frac{1}{V_t^2} \bw_{;t}^2(\bw^u_{;xt}-\frac{V_{tx}}{V_t} \bw^u_{;t})}{I[\bw_0]^2I[\bw^u]^2}+\frac{\frac{1}{V_t}\bw_{;x}(\bw_{;tt}-2\frac{V_{tt}}{V_t}\bw_{;t})}{I[\bw_0]^3}+ \frac{\frac{2}{V_t} \bw_{;t}(\bw_{;xt}-\frac{V_{tx}}{V_t} \bw_{;t})}{I[\bw_0]^3}+\frac{\frac{V_{tx}}{V_t^2}\bw_{;t}^2}{I[\bw_0]^2I[\bw^u]}\nonumber\\
&+\frac{\frac{3}{V_t^2}\bw_{0;x}\bw_{;t}(\bw_{;tt}-2\frac{V_{tt}}{V_t}\bw_{;t})+\frac{3}{V_t^2}\bw_{;x}\bw_{;t}(\bw^u_{;tt}-2\frac{V_{tt}}{V_t}\bw^u_{;t})}{I[\bw_0]^4}+ \frac{\frac{2}{V_t^2} \bw_{;t}^2(\bw^u_{;xt}-\frac{V_{tx}}{V_t} \bw^u_{;t})}{I[\bw_0]^3I[\bw^u]}\\
&+\frac{\frac{3}{V_t^3}\bw_{;t}^2\bw^u_{;x}(\bw^u_{;tt}-2\frac{V_{tt}}{V_t}\bw^u_{;t})}{I[\bw_0]^4I[\bw^u]}+\frac{\frac{3}{V_t^3}\bw_{;t}^2\bw^u_{;x}(\bw^u_{;tt}-2\frac{V_{tt}}{V_t}\bw^u_{;t})}{I[\bw_0]^3I[\bw^u]^2}+\frac{\frac{1}{V_t^4}\bw_{;t}^3\bw^u_{;x}(\bw^u_{;tt}-2\frac{V_{tt}}{V_t}\bw^u_{;t})}{I[\bw_0]^3I[\bw^u]^3},\nonumber
\end{align}
and
\begin{align}
  q_{xx}[\bw]=&\,\frac{\frac{1}{V_t}\bw_{;t}  \bw_{;xx}}{I[\bw_0]^2} + \frac{\frac{1}{V_t^2}\bw_{;t}^2  \bw^u_{;xx}}{I[\bw_0]^2I[\bw^u]}+\frac{\frac{1}{V_t}\bw_{;x}(2\bw_{;xt}-2\frac{V_{xt}}{V_t}\bw_{;t}-\frac{V_{tt}}{V_t}\bw_{;x})}{I[\bw_0]^2}\nonumber\\
    &+\frac{\frac{2}{V_t^2}\bw_{;t}\bw_{;x}(2\bw_{0;xt}-2\frac{V_{xt}}{V_t}\bw_{0;t}-\frac{V_{tt}}{V_t}\bw_{0;x})+\frac{2}{V_t^2}\bw_{;t}\bw_{0;x}(2\bw_{;xt}-2\frac{V_{xt}}{V_t}\bw_{;t}-\frac{V_{tt}}{V_t}\bw_{;x})}{I[\bw_0]^3}\nonumber\\
    &+\frac{\frac{2}{V_t^2}\bw_{;t}\bw_{;x}(2\bw_{;xt}-2\frac{V_{xt}}{V_t}\bw_{;t}-\frac{V_{tt}}{V_t}\bw_{;x})}{I[\bw_0]^3}+\frac{\frac{1}{V_t^3}\bw_{;t}^2\bw^u_{;x}(2\bw^u_{;xt}-2\frac{V_{xt}}{V_t}\bw^u_{;t}-\frac{V_{tt}}{V_t}\bw^u_{;x})(I[\bw_0]+2I[\bw^u])}{I[\bw_0]^3I[\bw^u]^2}\nonumber\\
    &+\frac{\frac{1}{V_t^2}\bw_{;x}^2(\bw^u_{;tt}-2\tfrac{V_{tt}}{V_t}\bw^u_{;t})+\frac{2}{V_t^2}\bw_{0;x}\bw_{;x}(\bw_{;tt}-2\tfrac{V_{tt}}{V_t}\bw_{;t})}{I[\bw_0]^3}+\frac{\frac{3}{V_t^3}\bw_{;t}\bw_{;x}(2\bw_{0;x}+\bw_{;x})(\bw^u_{;tt}-2\frac{V_{tt}}{V_t}\bw^u_{;t})}{I[\bw_0]^4}\\
    &+\frac{\frac{3}{V_t^3}\bw_{0;x}^2\bw_{;t}(\bw_{;tt}-2\frac{V_{tt}}{V_t}\bw_{;t})}{I[\bw_0]^4}+\frac{\frac{3}{V_t^4}\bw_{;t}^2(\bw^u_{;x})^2(\bw^u_{;tt}-2\frac{V_{tt}}{V_t}\bw^u_{;t})}{I[\bw_0]^4I[\bw^u]}+\frac{\frac{3}{V_t^4}\bw_{;t}^2(\bw^u_{;x})^2(\bw^u_{;tt}-2\frac{V_{tt}}{V_t}\bw^u_{;t})}{I[\bw_0]^3I[\bw^u]^2}\nonumber\\
    &+\frac{\frac{1}{V_t^5}\bw_{;t}^3(\bw^u_{;x})^2(\bw^u_{;tt}-2\frac{V_{tt}}{V_t}\bw^u_{;t})}{I[\bw_0]^3I[\bw^u]^3}.\nonumber
\end{align}
This proves the lemma.
\end{proof}

\begin{proposition}[quadratic quantity in cylindrical gauge]\label{prop_quad_cyl} We have
\begin{multline}
Q_{\phi_0}[u]\diamond V=-l_1[\bw]\left(\frac{l_P[\bw]}{\eta_0^2}+P_0l_2[\bw]-l_t[\bw]\right)+\eta_u N[\phi_0-u]q_1[\bw]\\
-(\eta_0^{-1}+l_1[\bw])\left(\frac{q_P[\bw]}{\eta_u^2}+(P_0+l_P[\bw])q_2[\bw]+l_P[\bw] l_2[\bw]-\frac{\bw_{;t}^2}{I[\bw_0]^2I[\bw^u]V_t}\right),
\end{multline}
where the various quantities appearing in the statement are specified below.
\end{proposition}

\begin{proof}For ease of notation, let us abbreviate
\begin{equation}
   P_u= (1+(V^u_t)^2)V^u_{xx}+(1+(V^u_x)^2)V^u_{tt}-2V^u_xV^u_tV^u_{xt},
\end{equation}
and
\begin{equation}
    \eta_u=\sqrt{1+(V_x^u)^2+(V_t^u)^2}.
\end{equation}
Using \eqref{v_vs_vu} we can express our quadratic quantity in the form
\begin{equation}\label{Theta_phiu_rest}
Q_{\phi_0}[u]\diamond V = N[\phi_0+u]-N[\phi_0]-L_{\phi_0} u\diamond V,
\end{equation}
where
\begin{equation}
N[\phi_0+u]=\frac{-1}{\eta_u}\bigg(
\frac{P_u}{\eta_u^2}-\frac{1}{V^u}-V^u_t\bigg),
\end{equation}
with the usual convention that $V^u$ and its derivatives are evaluated at $\big(x, t - \tfrac{\bw^u}{V_t}(x,t) \big)$.
To conveniently expand our quadratic quantity, let us write the statement of Lemma \ref{lemma_der_cyl_prof} (derivatives of cylindrical profile) in the schematic form
\begin{equation}
   V_{\alpha}^u=V^0_{\alpha}+l_{\alpha}[\bw]+q_{\alpha}[\bw],\qquad\qquad \alpha\in\{x,t,xx,xt,tt\}.
\end{equation}
Then, a direct computation shows that
\begin{equation}
P_u=P_0+l_P[\bw] +q_P[\bw],
\end{equation}
where
\begin{multline}
l_P[\bw] =(1+(V_t^0)^2)l_{xx}[\bw]+(1+(V_x^0)^2)l_{tt}[\bw]-2V_x^0V_t^0l_{tx}[\bw]\\
    +2(V^0_{xx}V^0_t-V^0_{tx}V^0_x) l_t[\bw]  +2(V^0_{tt}V^0_x-V^0_{tx}V^0_t) l_x[\bw],
\end{multline}
and
\begin{align}
    q_{P}[\bw]=&\,(1+(V_t^u)^2)q_{xx}[\bw]+(1+(V_x^u)^2)q_{tt}[\bw]-2V_x^uV_t^uq_{tx}[\bw]\nonumber\\
        &+2V_t^0l_t[\bw]l_{xx}[\bw]+2V_x^0l_x[\bw] l_{tt}[\bw]-2\left(V_x^0l_t[\bw] +V_t^0l_x[\bw]\right)l_{tx}[\bw]\nonumber\\
    &+(V_{xx}^0+l_{xx}[\bw])\left(\frac{I[\bw_0]^2}{I[\bw^u]^2}+\frac{2V_t^0I[\bw_0]^2}{V_t I[\bw^u]}\right) l_t[\bw]^2-2(V_{tx}^0+l_{tx}[\bw])\frac{I[\bw]^2}{I[\bw^u]^2}l_t[\bw]l_x[\bw]\\
    &+(V_{tt}^0+l_{tt}[\bw])\left(\frac{I[\bw_0]^2}{I[\bw^u]^2}l_x[\bw]^2+2\frac{V_x^0I[\bw_0]^2}{V_t I[\bw^u]}l_t[\bw]l_x[\bw]\right).\nonumber
\end{align}
Similarly, for $k=1,2$ we infer that
\begin{equation}
\frac{1}{\eta_u^{k}}=\frac{1}{\eta_0^{k}}+l_{k}[\bw] +q_{k}[\bw],
\end{equation}
where
\begin{equation}
    l_{k}[\bw]=\frac{-k}{\eta_0^{k+2}}\left(V_x^0l_x[\bw]+V_t^0 l_t[\bw]\right),
\end{equation}
and
\begin{multline}
q_1[\bw]= -\frac{1}{2\eta_0^3}(l_x[\bw]^2+2V_x^uq_x[\bw]- q_x[\bw]^2+l_t[\bw]^2+2V_t^uq_t[\bw]- q_t[\bw]^2)\\
+\frac{(2\eta_0+\eta_u)\left[2V_x^0(l_x[\bw]+q_x[\bw])+(l_x[\bw]+q_x[\bw])^2+2V_t^0(l_t[\bw]+q_t[\bw])+(l_t[\bw]+q_t[\bw])^2\right]^2}{2\eta_0^3\eta_u(\eta_0+\eta_u)^2},
\end{multline}
and
\begin{align}
q_2[\bw]=l_1[\bw]^2+\frac{2}{\eta_u}q_1[\bw]-q_1[\bw]^2.
\end{align}
Finally, observe that
\begin{equation}
    \frac{P_u}{\eta_u^2}=\frac{P_0}{\eta_0^2}+\frac{l_P[\bw]}{\eta_0^2}+P_0l_2[\bw]+\frac{q_P[\bw]}{\eta_u^2}+(P_0+l_P[\bw])q_2[\bw]+l_P[\bw]l_2[\bw].
\end{equation}
Combining the above formulas the assertion follows.
\end{proof}

Also, as before we set
\begin{equation}
\tilde{\bw}(x,s,t)=\bw(x,s+t),\qquad \tilde{V}(x,s,t)=V(x,s+t).
\end{equation}
Moreover, we define
\begin{equation}
Q_{\phi_0}[ \tilde{\bw} ](x,s,t):=\sqrt{1+\tilde{V}_x^2(x,s,t)+\tilde{V}_s^2(x,s,t)} \,\, Q_{\phi_0}[u](x,\tilde{V}(x,s,t),0).
\end{equation}
Note that $Q_{\phi_0}[ \tilde{\bf{w}} ](x,s,t)=Q_{\phi_0}[ {\bf{w}} ](x,s+t)$, provided we set $Q_{\phi_0}[ {\bf{w}} ]:=\sqrt{1+{V}_x^2+{V}_t^2} \,\, Q_{\phi_0}[u]\diamond V$.

\begin{corollary}[quadratic error for cylindrical variation]\label{quad_err_prop}
There exist constants $\eps>0$ and $C<\infty$, such that for all $\phi_0=\phi_\ast+u_0$ with $ \| u_0 \|_{\mathbb{X}^{k+2,\alpha}(\mathbb{R}^3/S^1,\mathbb{C})}\leq \eps$ and all $u$ with $ \| u \|_{\mathbb{X}^{k+2,\alpha}(\mathbb{R}^3/S^1,\mathbb{C})}\leq \eps$,  
for any $x$ and $t\leq -h_0$, setting
 $r=\tfrac{1}{10}H(x,t)^{-1}$, whenever $V(x,t)\geq \ell \sqrt{|t|/\log|t|}$ we have
\begin{align}\label{cor_ts1}
\| Q_{\phi_0}[ \tilde{\bf{w}} ] \|_{C_H^{k-2,\alpha,(2)}\left(P_{r}\left(x,t\right),\mathbb{C}\right)} \leq C\|\tilde{\bw}\|_{C_H^{k+2,\alpha}\left(P_{r}\left(x,t\right),\mathbb{C}\right)}^2 .
\end{align}
\end{corollary}

\begin{proof}
To illustrate how the corollary follows, let us consider a typical term, specifically
\begin{equation}
Q_{\phi_0}[ \tilde{\bf{w}} ] \sim \frac{1}{\tilde{V}_t^5}  \tilde{\bw}_{;t}^3\tilde{\bw}_{;x}^2 \tilde{\bw}_{;tt} + \textrm{many other terms.}
\end{equation}
Since $(x,t)$ lies in the cylindrical region, we have $\tilde{V}_t \sim H \sim r^{-1}$, hence
\begin{equation}
\left\| \frac{1}{\tilde{V}_t^5}  \tilde{\bw}_{;t}^3\tilde{\bw}_{;x}^2 \tilde{\bw}_{;tt}  \right\|_{C^{k-2,\alpha,(2)}_H(P_r(x,t))} \leq C r^5 \left\| \tilde{\bw}_{;t}^3\tilde{\bw}_{;x}^2 \tilde{\bw}_{;tt}  \right\|_{C^{k-2,\alpha,(2)}_H(P_r(x,t))} .
\end{equation}
To proceed, we need the product rule for the weighted H\"older norms, namely
\begin{equation}
\| fg\|_{C^{k-2,\alpha,(l)}_H(P_r(x,t))} \leq C r^{1+l-l_1-l_2} \| f\|_{C^{k-2,\alpha,(l_1)}_H(P_r(x,t))} \| g\|_{C^{k-2,\alpha,(l_2)}_H(P_r(x,t))}.
\end{equation}
This yields
\begin{equation}
\left\| \tilde{\bw}_{;t}^3\tilde{\bw}_{;x}^2 \tilde{\bw}_{;tt}  \right\|_{C^{k-2,\alpha,(2)}_H(P_r(x,t))} \leq Cr^{-5}  \left\| \tilde{\bw}_{;t}   \right\|_{C^{k-2,\alpha,(2)}_H(P_r(x,t))}^3   \left\|  \tilde{\bw}_{;x} \right\|_{C^{k-2,\alpha,(1)}_H(P_r(x,t))}^2  \left\|  \tilde{\bw}_{;tt} \right\|_{C^{k-2,\alpha,(4)}_H(P_r(x,t))} .
\end{equation}
Also note that
\begin{equation}
 \left\| \tilde{\bw}_{;t}   \right\|_{C^{k-2,\alpha,(2)}_H(P_r(x,t))} \leq  \left\| \tilde{\bw}   \right\|_{C^{k,\alpha}_H(P_r(x,t))},\qquad \left\| \tilde{\bw}_{;x}   \right\|_{C^{k-2,\alpha,(1)}_H(P_r(x,t))} \leq  \left\| \tilde{\bw}   \right\|_{C^{k-1,\alpha}_H(P_r(x,t))},
\end{equation}
and
\begin{equation}
 \left\| \tilde{\bw}_{;tt}   \right\|_{C^{k-2,\alpha,(4)}_H(P_r(x,t))} \leq  \left\| \tilde{\bw}   \right\|_{C^{k+2,\alpha}_H(P_r(x,t))}.
\end{equation}
Combining the above, we infer that
\begin{equation}
\left\| \frac{1}{\tilde{V}_t^5}  \tilde{\bw}_{;t}^3\tilde{\bw}_{;x}^2 \tilde{\bw}_{;tt}  \right\|_{C^{k-2,\alpha,(2)}_H(P_r(x,t))}  \leq C  \left\| \tilde{\bw} \right\|_{C^{k+2,\alpha}_H(P_r(x,t))}^6 .
\end{equation}
Arguing similarly for all the other terms and their derivatives, this yields \eqref{cor_ts1}.
\end{proof}

Finally, in the soliton region we consider the quadratic quantity
\begin{equation}
Q_{\phi_0}[ \tilde{\bW} ](x_2,x_3,x_4,t):=\sqrt{1+|D\tilde{X}(x_2,x_3,x_4,t)|^2} \,\, Q_{\phi_0}[u](\tilde{X}(x_2,x_3,x_4,t),x_2,x_3).
\end{equation}

\begin{corollary}[quadratic error for tip variation]\label{quad_err_prop_tip}
There exist constants $\eps>0$ and $C<\infty$, such that for all $\phi_0=\phi_\ast+u_0$ with $ \| u_0 \|_{\mathbb{X}^{k+2,\alpha}(\mathbb{R}^3/S^1,\mathbb{C})}\leq \eps$ and all $u$ with $ \| u \|_{\mathbb{X}^{k+2,\alpha}(\mathbb{R}^3/S^1,\mathbb{C})}\leq \eps$, whenever $|x|\leq \ell \sqrt{|t|/\log|t|}$ at some $t\leq -h_0$, then setting $r=\sqrt{|t|/\log|t|}$ we have 
\begin{align}
\| Q_{\phi_0}[ \tilde{\bf{W}} ]\|_{C^{k-2,\alpha,(2)}_H(Q_r(x,t),\mathbb{C})} \leq C\| \tilde{\bf{W}} \|_{C^{k+2,\alpha}_H(Q_r(x,t),\mathbb{C})}^2 .
\end{align}
\end{corollary}

\begin{proof}
This follows from a similar (but less delicate) argument as for the cylindrical variation.
\end{proof}

Combining the above estimates we now obtain:

\begin{theorem}[global quadratic error estimate]\label{quad_err_put_together}
There exist constants $\eps>0$ and $C<\infty$, such that for all $\phi_0=\phi_\ast+u_0$ with $ \| u_0 \|_{\mathbb{X}^{k+2,\alpha}(\mathbb{R}^3/S^1,\mathbb{C})}\leq \eps$, we have
\begin{align}
\| Q_{\phi_0}[ u ]\|_{\mathbb{Y}^{k-2,\alpha}(\mathbb{R}^3/S^1,\mathbb{C})}\leq C \| u \|_{\mathbb{X}^{k+2,\alpha}(\mathbb{R}^3/S^1,\mathbb{C})}^2,
\end{align}
provided that $\| u \|_{\mathbb{X}^{k+2,\alpha}(\mathbb{R}^3/S^1,\mathbb{C})}\leq \eps$. In particular, the map
\begin{equation}
B_{\mathbb{X}^{k+2,\alpha}(\mathbb{R}^3/S^1)}(0,\eps)\to \mathbb{Y}^{k,\alpha}(\mathbb{R}^3/S^1),\qquad u\mapsto \Theta[\phi+u]
\end{equation}
is analytic, and its derivative is given by $L_{\phi+u}$.
\end{theorem}

\begin{proof}
Note that our weight functions satisfy
\begin{equation}
\rho_\bullet\geq\rho_\star^2.
\end{equation}
Hence, 
remembering the definitions of the norms, the theorem follows by combining
Corollary \ref{quad_err_graph} (quadratic error for graphical variation),
Corollary \ref{quad_err_prop} (quadratic error for cylindrical variation)
and Corollary \ref{quad_err_prop_tip} (quadratic error for tip variation).
\end{proof}

\bigskip

\subsection{Lyapunov-Schmidt reduction}
In this final subsection, we work with the graded Frechet spaces
\begin{equation}
\mathbb{X}=\bigcap_{k\geq 4}\mathbb{X}^{k,\alpha}(\mathbb{R}^3/S^1),\qquad \mathbb{Y}=\bigcap_{k\geq 4}\mathbb{Y}^{k-2,\alpha}(\mathbb{R}^3/S^1).
\end{equation}
To conveniently deal with normalizations, let us also define the somewhat smaller space
\begin{equation}
\mathbb{X}_0:= \{ u \in \mathbb{X} \, | \, u(0)=0, Du(0)=0 \}.
\end{equation}
Recall also that we denote by $\mathcal{S}$ the space of all nontrivial noncollapsed translators in $\mathbb{R}^4$ normalized as usual, in other words
\begin{multline}
\mathcal{S} = \big\{ \phi \in C^\infty(\mathbb{R}^3/S^1)  \, |  \, \Theta[\phi]=0, \phi(0)=0, D\phi(0) = 0, \\
  \textrm{$\phi$ is strictly convex and not SO$_3$-symmetric} \big\},
\end{multline}
which is equipped with the smooth topology. Let us fix some $\phi_\ast\in \mathcal{S}$.

\begin{lemma}[compatibility]\label{lemma_compatible}
If $u\in \mathbb{X}_0$ is such that $\Theta[\phi_\ast+u]=0$, then $\phi_\ast+u\in \mathcal{S}$. Conversely, there exists an open neighborhood $\mathcal{I}\subset \mathcal{S}$ of $\phi_\ast$, such that $\iota: \mathcal{I}\to \mathbb{X}_0$, $\phi\mapsto \phi-\phi_\ast$ is well-defined and continuous.
\end{lemma}

\begin{proof}
If $u\in \mathbb{X}_0$ is such that $\Theta[\phi_\ast+u]=0$, then $M_t=\mathrm{graph}(\phi_\ast+u+t)$ is an eternal mean-convex flow that sweeps out all space, so in particular it makes sense to consider a tangent flow at $-\infty$. By Proposition \ref{prop_who_contr_whom} (controlled pointwise quantities) no such tangent flow at $-\infty$ can be a multiplicity-two plane. This implies convexity thanks to the general theory from \cite{White_nature,HaslhoferKleiner_meanconvex}. Observing also that $\phi_\ast+u$ clearly neither is $\mathrm{SO}_3$-symmetric nor splits off a line, this shows that $\phi_\ast+u \in\mathcal{S}$.

Conversely, thanks to Theorem \ref{thm_asympt_expansion_oval_bowls} (second order asymptotics), taking also into account Theorem \ref{inner_outer_intro} (inner-outer estimate), for any $\eps>0$ if $\phi_1,\phi_2\in\mathcal{S}$ are close enough in the smooth topology, then
\begin{equation}
\| \phi_1-\phi_2 \|_{C^{k+10,\alpha}_{\star}(\mathbb{R}^3/S^1)}\leq \eps.
\end{equation}
Hence, applying Theorem \ref{quad_err_put_together} (global quadratic error estimate) we infer that
\begin{align}
\| L_{\phi_1}[ \phi_1-\phi_2 ]\|_{C^{k-2,\alpha}_{\bullet}(\mathbb{R}^3/S^1)}\leq C \eps^2,
\end{align}
and
\begin{align}
\| L_{\phi_1}[ \phi_1-\phi_2 ] -L_{\phi_\ast}[ \phi_1-\phi_2 ] \|_{\mathbb{Y}^{k-2,\alpha}(\mathbb{R}^3/S^1)}\leq 2 \|  \phi_1-\phi_2  \|_{\mathbb{X}^{k+10,\alpha}(\mathbb{R}^3/S^1)},
\end{align}
provided that $\phi_1,\phi_2\in\mathcal{S}$ are close enough to $\phi_\ast$ in the smooth topology.
Combining the above facts, the assertion follows.
\end{proof}

Recall that the tip curvature map is defined by
\begin{equation}
\kappa:\mathcal{S}\to \mathbb{R},\quad \phi\mapsto \tfrac{1}{2} (\partial^2_{x_1}\phi) (0).
\end{equation}

\begin{theorem}[analyticity]\label{theorem_analyticity}
The space $\mathcal{S}$ is a finite-dimensional analytic variety over which $\kappa:\mathcal{S}\to\mathbb{R}$ is an analytic function.
\end{theorem}

\begin{proof}
Fix $\phi_\ast\in \mathcal{S}$, and define $\mathbb{X}_0^{k,\alpha}$ and $\mathbb{Y}^{k-2,\alpha}$ with respect to $\phi_\ast$. Thanks to Theorem \ref{Fredholm_theorem} (Fredholm property) the map $L=L_{\phi_\ast}^{k,\alpha}:\mathbb{X}_0^{k,\alpha}\to \mathbb{Y}^{k-2,\alpha}$ is Fredholm. By elliptic regularity, the kernel and the cokernel of $L$ are independent of $k$. Setting $\mathbb{Y}_2^{k-2,\alpha}=L(\mathbb{X}_0^{k,\alpha})$, we can thus decompose
\begin{equation}
\mathbb{X}_0^{k,\alpha}=\mathbb{X}_1\oplus \mathbb{X}_2^{k,\alpha},\qquad \mathbb{Y}_0^{k-2,\alpha}=\mathbb{Y}_1\oplus \mathbb{Y}_2^{k-2,\alpha},
\end{equation}
where $\mathbb{X}_1$ and $\mathbb{Y}_1$ are finite dimensional, and where $L|_{\mathbb{X}_2^{k,\alpha}}$ is an isomorphism from $\mathbb{X}_2^{k,\alpha}$ to $\mathbb{Y}_2^{k-2,\alpha}$. Setting $\mathbb{Y}_2:=\bigcap_{k\geq 4}\mathbb{Y}_2^{k-2,\alpha}$, let us fix a projection map $\Pi:\mathbb{Y}\to \mathbb{Y}_2$. Now, thanks to Theorem \ref{quad_err_put_together} (global quadratic error estimate), considering the map
\begin{equation}
B_{\mathbb{X}_0}(0,\eps)\to \mathbb{Y}_2,\quad u\mapsto \Pi\Theta[\phi_\ast+u],
\end{equation}
we can apply Ekeland's implicit function theorem \cite{Ekeland}, which gives us an open neighbourhood $\mathbb{U}=\mathbb{U}_1\times\mathbb{U}_2$ of the origin and an analytic function $f:\mathbb{U}_1\to\mathbb{U}_2$ such that for $(u_1,u_2)\in \mathbb{U}$ we have 
\begin{equation}
\Pi\Theta[\phi_\ast+(u_1,u_2)]=0\quad\Leftrightarrow\quad u_2=f(u_1).
\end{equation}
Here, it is most convenient to apply the implicit function theorem after temporarily passing to the complexifications, since then one only needs one derivative (also note that the notion of being analytic is unambiguous since the domain of $f$ is finite-dimensional). 
Together with Lemma \ref{lemma_compatible} (compatibility) it follows that possibly after decreasing $\mathbb{U}$ there is an open neighborhood $\mathcal{I}\subset \mathcal{S}$ of $\phi_\ast$, such that
\begin{equation}
\iota(\mathcal{I})\cap \mathbb{U}=\big\{ (u_1,f(u_1)) \, : \, u_1\in\mathbb{U}_1,\, (1-\Pi)\Theta[\phi_\ast+(u_1,u_2)]=0 \big\}.
\end{equation}
Hence, $\mathcal{I}$ is a finite dimensional analytic variety over which $\kappa\circ \iota$ is analytic. %Together with the fact that $\mathcal{S}$ is homeomorphic to an open interval \cite{CHH_translators}.
This implies the assertion.
\end{proof}

\bigskip

\appendix

\section{Second order asymptotics for translators}

In this appendix, we derive some second order asymptotics in the parabolic region. As usual, we consider nontrivial noncollapsed translators $M=\mathrm{graph}(\phi)\subset\mathbb{R}^4$, where $\phi\in\mathcal{S}$. Recall from \cite[Proposition 5.3]{CHH_translators} that the renormalized profile function $u(y,\tau)=v(y,\tau)-\sqrt{2}$ evolves by
\begin{equation}\label{eq:profile_evol}
u_\tau=\mathfrak{L}u-\frac{u^2}{2(\sqrt{2}+u)}-\frac{u_y^2u_{yy}}{1+u_y^2}+O(e^{\tau/2}).
\end{equation}
Thanks to the $\mathbb{Z}_2$-symmetry, the function $u(\cdot,\tau)$ is a linear combination of the even Hermite polynomials $H_{2k}(y)$. Here, we work with the probabilist's normalization, so in particular the first three ones are
\begin{align}
H_0(y)=1, \qquad H_2(y)=y^2-2, \qquad H_4(y)=y^4-12y^2+12.
\end{align}
Also, observe that in our space $L^2(\mathbb{R},  (4\pi)^{-1/2} e^{-y^2/4}\, dy)$ we have $\| H_0\|=1$, $\| H_2\|^2=8$ and $\| H_4\|^2=384$.\\

More precisely, fixing a small constant $\delta>0$, we work with the truncated profile function
\begin{align}
\hat u (y,\tau)=u(y,\tau)\eta(|y|/|\tau|^{\delta}),
\end{align}
and $\eta$ is a cut-off function such that $\eta(r)\equiv 1$ for $r\leq 1 $ and $\eta(r)\equiv 0$ for $r\geq 2$. Recall that by
\eqref{profile_growth} and \eqref{profile_derivative_growth}, on the support of $\hat{u}$ we have
\begin{align}\label{eq:profile_bound_support}
 |u| \leq  \frac{C(1+y^2)}{|\tau|}, \qquad |u_y|+|u_{yy}| \leq \frac{C(1+|y|)}{|\tau|}. 
\end{align}
Therefore, the evolution of our truncated profile function takes the form
\begin{equation}\label{eq:profile_eq_cutoff}
\hat u_\tau=\mathfrak{L}\hat u -2^{-\frac{3}{2}}\hat u^2+O(|\tau|^{-3+6\delta})+o(1)1_{\{|y| \geq |\tau|^{\delta}\}}.
\end{equation}
Also recall that by \cite{DH_no_rotation}, possibly after decreasing $\delta$, for $\tau\ll 0$ we have
\begin{align}\label{eq:minor_projection_improved}
 \|\hat u\|=|\tau|^{-1}+O(|\tau|^{-1-10\delta}),\qquad
 \left\| \hat{u} -P_0\hat{u}\right\|\leq C|\tau|^{-1-10\delta},
\end{align}
where $P_0$ denotes the projection to the neutral space spanned by $H_2$.\\

Now, we will first derive asymptotics for the first three spectral coefficients in the expansion
\begin{equation}
\hat u(y,\tau)=\sum_{k=0}^\infty a_{2k}(\tau)H_{2k}(y).
\end{equation}

\begin{proposition}[spectral coefficients]\label{lem:a0_improved} For $\tau\ll 0$ we have
\begin{equation}
a_0(\tau)=\frac{1}{2\sqrt{2}|\tau|^2}+O(|\tau|^{-2-4\delta}),\qquad
a_2(\tau)=-\frac{1}{2\sqrt{2}|\tau|}+O(|\tau|^{-1-4\delta}),
\end{equation}
and
\begin{equation}
a_4(\tau)=-\frac{1}{16\sqrt{2}|\tau|^2}+O(|\tau|^{-2-4\delta}).
\end{equation}
\end{proposition}

\begin{proof}
To begin with, using \eqref{eq:profile_eq_cutoff} and \eqref{eq:minor_projection_improved} we see that
\begin{equation}
a_0'=\langle 1, \hat  u_\tau\rangle =a_0-2^{-\frac{3}{2}}\|\hat u\|^2+O(|\tau|^{-3+6\delta})=a_0-2^{-\frac{3}{2}}|\tau|^{-2}+O(|\tau|^{-2-4\delta}).
\end{equation}
This yields the expansion for $a_0$. Also, the expansion for $a_2$ has already been derived in \cite{DH_no_rotation}.

Now, to derive the formula for $a_4$, using again \eqref{eq:profile_eq_cutoff} we first compute
\begin{equation}
a_4'=\|H_4\|^{-2}\langle H_4, \hat u_\tau \rangle=-a_4-2^{-\frac{3}{2}}\|H_4\|^{-2}\langle H_4, \hat u^2 \rangle
+O(|\tau|^{-3+6\delta}).
\end{equation}
To proceed, observe that thanks to \eqref{eq:minor_projection_improved} we have
\begin{equation}
\langle H_4,\hat u^2\rangle=a_2^2\langle H_4,H_2^2\rangle+O(|\tau|^{-2-4\delta}).
\end{equation}
Together with $H_2^2=H_4+8H_2+8$ and the expansion for $a_2$ this implies
\begin{equation}
a_4'=-a_4-2^{-9/2}|\tau|^{-2}+O(|\tau|^{-2-4\delta}).
\end{equation}
This yields the expansion for $a_4$, and thus concludes the proof of the proposition.
\end{proof}

Next, we consider the remainder
\begin{equation}
\hat P \hat u =\hat{u}-a_0H_0-a_2H_2-a_4H_4.
\end{equation}

\begin{lemma}[remainder]\label{lem:hatP_error}
For $\tau \ll 0$ we have
\begin{equation}
 \|\hat P\hat u\|=O(|\tau|^{-3+6\delta}).
\end{equation}
\end{lemma}

\begin{proof}
To begin with, note that we have the Gaussian tail estimate
\begin{equation}
\|(\hat P\hat u)1_{\{|y| \geq |\tau|^{\delta}\}}\| \leq C|\tau|^{-10}.
\end{equation}
Using the identity
\begin{equation}
2\langle \mathfrak{L}f,f\rangle=\langle Lf,f\rangle+\|f\|^2-\|f_y\|^2,
\end{equation}
we can thus derive from \eqref{eq:profile_eq_cutoff} that
\begin{equation}\label{remainder_evol_ineq}
\tfrac{d}{d\tau}\|\hat P\hat u\|^2 \leq -\|\hat P\hat u\|^2-\|(\hat P\hat u)_y\|^2-2^{-\frac{1}{2}}\langle \hat P\hat u,\hat u^2\rangle+C|\tau|^{-3+6\delta}\|\hat P\hat u\|+C|\tau|^{-10}.
\end{equation}
To proceed, we expand
\begin{equation}
 \langle \hat P\hat u,\hat u^2\rangle=  \langle (\hat P\hat u)^2,\hat u \rangle+ \langle (\hat P\hat u),(a_0+a_2H_2+a_4H_4)\hat u \rangle.
\end{equation}
Using the pointwise bound for $\hat{u}$ from \eqref{eq:profile_bound_support} and the weighted Poincare inequality we can estimate
\begin{equation}
|\langle (\hat P\hat u)^2,\hat u \rangle|\leq C|\tau|^{-1} \left( \| \hat P\hat u\|^2+\| (\hat P\hat u)_y\|^2 \right).
\end{equation}
Next, using also Proposition \ref{lem:a0_improved} (spectral coefficients) we see that
\begin{align}
|\langle  \hat P\hat u,(a_0+a_4H_4)\hat u\rangle|\leq C|\tau|^{-3}\|\hat P \hat u\|.
\end{align}
Moreover, using $H_2^2=H_4+8H_2+8$, we can expand
\begin{equation}
\langle \hat P\hat u,a_2H_2\hat u\rangle=a_2a_4\langle \hat P\hat u,H_2H_4\rangle+a_2\langle \hat P\hat u,H_2\hat P \hat u\rangle,
\end{equation}
and arguing as before we can estimate
\begin{equation}
|a_2a_4\langle \hat P\hat u,H_2H_4\rangle| \leq C|\tau|^{-3}\|\hat P \hat u\|,\qquad 
|a_2\langle \hat P\hat u,H_2\hat P \hat u\rangle|\leq
C|\tau|^{-1} \left( \| \hat P\hat u\|^2+\| (\hat P\hat u)_y\|^2 \right).
\end{equation}
Combining these estimates shows that
\begin{equation}\label{lemma4_eq}
|\langle \hat P\hat u,\hat u^2\rangle| \leq  C|\tau|^{-1} \left( \| \hat P\hat u\|^2+\| (\hat P\hat u)_y\|^2 \right)+C|\tau|^{-3}\|\hat P \hat u\|.
\end{equation}
Plugging this into \eqref{remainder_evol_ineq} we conclude that
\begin{equation}
\tfrac{d}{d\tau}\|\hat P\hat u\|^2 \leq -\tfrac{1}{2}\|\hat P\hat u\|^2+C|\tau|^{-3+6\delta}\|\hat P\hat u\|.
\end{equation}
This implies the assertion.
\end{proof}

\bigskip

Now, we refine equation \eqref{eq:profile_eq_cutoff} by Taylor expansion of \eqref{eq:profile_evol}. Taking into account \eqref{eq:profile_bound_support} this yields
\begin{equation}\label{further_expanded_ev_eq}
\hat u_\tau=\mathfrak{L}\hat u- \frac{\hat u^2}{2\sqrt{2}}+\frac{\hat u^3}{4}-\hat u_y^2\hat u_{yy}+O(|\tau|^{-4+8\delta})+o(1)1_{\{|y| \geq |\tau|^{\delta}\}}.
\end{equation}
To obtain refined asymptotics for $a_2$ we need the following result.

\begin{proposition}[reaction terms]\label{lem:a2_error1}
For $\tau\ll 0$ we have
\begin{equation}\label{eq_react_term1}
\langle H_2,\hat u^2\rangle=64a_2^2+10|\tau|^{-3}+O(|\tau|^{-3-4\delta}),
\end{equation}
and
\begin{equation}\label{eq_react_term2}
\langle H_2, \tfrac{1}{4} \hat u^3-\hat u_y^2\hat u_{yy}\rangle=-11\cdot 2^{-\frac{1}{2}}|\tau|^{-3}+O(|\tau|^{-3-4\delta}).
\end{equation}
\end{proposition}

\begin{proof}
To begin with, thanks to \eqref{eq:profile_bound_support} and Proposition \ref{lem:hatP_error} (remainder) we have
\begin{equation}
|\langle H_2\hat u,\hat P\hat u \rangle| \leq C|\tau|^{-4+6\delta}.
\end{equation}
Next, using $H_2^2=H_4+8H_2+8$, $\|H_2\|^2=8$ and $\|H_4\|^2=384$ we infer that
\begin{equation}
\langle H_2\hat u,a_0H_0\rangle=8a_0a_2,\qquad  \langle H\hat u,a_2H_2\rangle=384a_2a_4+64a_2^2 +8a_0a_2.
\end{equation}
Moreover, taking also into account \eqref{eq:minor_projection_improved} we see that
\begin{equation}
\langle H_2\hat u , a_4H_4\rangle= 384a_2a_4+O(|\tau|^{-3-10\delta}).
\end{equation}
Combining these formulas and Proposition \ref{lem:a0_improved} (spectral coefficients) yields \eqref{eq_react_term1}.

Similarly, using \eqref{eq:minor_projection_improved}, and the identities $\langle H_2,H_2^3\rangle = 960$ and $\| H_2\|^2=8$ we see that
\begin{equation}\label{eq:a2_sharp_error1}
\langle H_2 , \hat u^3\rangle=960a_2^3+O(|\tau|^{-3-10\delta}), \qquad \langle H_2, \hat{u}_y^2  \hat{u}_{yy} \rangle = 64 a_2^3+O(|\tau|^{-3-10\delta}).
\end{equation}
Combining these formulas and Proposition \ref{lem:a0_improved} (spectral coefficients) yields \eqref{eq_react_term2}.
\end{proof}

\begin{theorem}[second order asymptotics]\label{thm_asympt_expansion_oval_bowls}
There exist $\delta>0$ and $A(M)\in \mathbb{R}$, depending continuously on $M$, such that in each compact set $\{ |y| \leq R\}$ for sufficiently negative $\tau$ we have
\begin{equation}
 u = \left(-\frac{1}{2\sqrt{2}|\tau|}+\frac{\sqrt{2}\log |\tau|}{|\tau|^2}+\frac{A}{|\tau|^2}\right)(y^2-2)-\frac{y^4-12y^2+4}{16\sqrt{2}|\tau|^2}+O(|\tau|^{-2-\delta}).
\end{equation}
\end{theorem}

\begin{proof} 
Using the evolution equation \eqref{further_expanded_ev_eq} and Proposition \ref{lem:a2_error1} (reaction terms) we see that
\begin{equation}
 a_2'=2^{-3}\langle H_2,\hat u_\tau\rangle=-2\sqrt{2}\, a_2^2-\sqrt{2}\,|\tau|^{-3}+O(|\tau|^{-3-4\delta}).
\end{equation}
To analyze this ODE we consider the function
\begin{equation}
b(\tau)=a_2(\tau)+\frac{1}{2\sqrt{2}\,|\tau|}-\frac{\sqrt{2}\log |\tau|}{|\tau|^2}.
\end{equation}
Note that
\begin{equation}
b'=\left(\frac{2}{|\tau|}-\frac{8\log |\tau|}{|\tau|^2}-2\sqrt{2}b\right)b+O(|\tau|^{-3-4\delta}).
\end{equation}
Hence, remembering Proposition \ref{lem:a0_improved} (spectral coefficients), we have
\begin{equation}
b'=(2+O(|\tau|^{-4\delta}))|\tau|^{-1}b+O(|\tau|^{-3-4\delta}).
\end{equation} 
Therefore, $\hat b(\tau)=|\tau|^2b(\tau)$ satisfies
\begin{equation}
\hat b'=p \hat b +q,
\end{equation}
for some $|p(\tau)|,|q(\tau)| \leq C|\tau|^{-1-4\delta}$. Solving this ODE, we infer that $A=\lim_{\tau\to -\infty}\hat{b}(\tau)$ exists and that
\begin{equation}
\hat b= A+O(|\tau|^{-4\delta}).
\end{equation}
Note also that $A$ depends continuously on $M$. We have thus shown that for $A=A(M)$ we have
\begin{equation}
a_2(\tau)=-\frac{1}{2\sqrt{2}|\tau|}+\frac{\sqrt{2}\log |\tau|}{|\tau|^2}+\frac{A}{|\tau|^2}+O(|\tau|^{-2-4\delta}).
\end{equation}
Together with Proposition \ref{lem:a0_improved} (spectral coefficients) and Proposition \ref{lem:hatP_error} (remainder) this implies the assertion.
\end{proof}

\bigskip

\bibliography{LTE}
\bibliographystyle{alpha}

\vspace{5mm}

{\sc Kyeongsu Choi, School of Mathematics, Korea Institute for Advanced Study, 85 Hoegiro, Dongdaemun-gu, Seoul, 02455, South Korea}\\

{\sc Robert Haslhofer, Department of Mathematics, University of Toronto,  40 St George Street, Toronto, ON M5S 2E4, Canada}\\

{\sc Or Hershkovits, Department of Mathematics, University of Maryland, 4176 Campus Dr, College Park, MD 20742, USA and Institute of Mathematics, Hebrew University of Jerusalem, Jerusalem, 91904, Israel}\\

\end{document}